\documentclass[12pt,reqno]{amsart}
\usepackage{amssymb,amsmath,graphicx,verbatim,amsthm,bbm}
\usepackage{color}
\usepackage{xcolor}
\usepackage{enumitem}
\usepackage{bbm}
\usepackage{floatflt}
\usepackage{nicefrac}
\usepackage{calrsfs}
\usepackage{hyperref}
\usepackage{stmaryrd}

\gdef\SetFigFontNFSS#1#2#3#4#5{} 
\gdef\SetFigFont#1#2#3#4#5{} 

\definecolor{subtleblue}{rgb}{0.09,0.32,0.44} 
\hypersetup{
    linktoc=page,  
    colorlinks=true,        
    pdfborder={0 0 0},      
    pdfborderstyle={},      
    linkcolor=subtleblue,   
    citecolor=subtleblue,   
    filecolor=magenta,      
    urlcolor=blue           
}

\definecolor{env_back}{gray}{0.8}
\definecolor{thm_color}{rgb}{0,0,0}
\definecolor{conj_color}{rgb}{0,0,0}
\definecolor{dfn_color}{rgb}{0,0,0}

\newtheorem{thm}{{\color{thm_color}Theorem}}[section]

\newtheorem{lem}[thm]{{\color{thm_color}Lemma}}
\newtheorem{fact}[thm]{{\color{thm_color}Fact}}
\newtheorem{claim}[thm]{\color{thm_color}Claim}

\newtheorem{prop}[thm]{{\color{thm_color}Proposition}}

\theoremstyle{definition}
\newtheorem*{def*}{Definition}
\newtheorem*{rem*}{Remark}

\newcommand{\dist}[2]{d_{#1#2}}
\newcommand{\TV}[2]{d_{\mathit{TV}}(#1,#2)}
\newcommand{\DKL}[2]{D_{\mathit{KL}}\left(#1\|#2\right)}

\newcommand{\BDM}{\mathcal{A}}
\newcommand{\M}{\mathcal{M}}
\newcommand{\B}{\mathcal{B}}
\newcommand{\C}{\mathcal{C}}
\newcommand{\R}{\mathbb{R}}

\newcommand{\vol}{\mathrm{Vol}}
\renewcommand{\P}{\mathbb{P}}

\newcommand{\eps}{\varepsilon}
\newcommand{\G}{\mathcal{G}}
\newcommand{\E}{\mathbb{E}}
\newcommand{\T}{\mathcal{T}}
\newcommand{\1}{\mathbbm{1}}
\newcommand{\II}{\mathcal{I}}
\newcommand{\ext}{\mathrm{ex}(n,K_{t,t})}
\newcommand{\exH}{\mathrm{ex}(n,H)}
\newcommand{\poly}{\mathcal{P}}

\newcommand{\Mnd}{\M_n^{t,\delta}}
\newcommand{\Mndc}{\overline{\Mnd}}

\newcommand{\we}{\emph{\L}}
\newcommand{\HH}{\mathcal{H}}

\newcommand{\F}{\mathcal{F}}
\def\eqd{\,{\buildrel d \over =}\,}

\newcommand{\br}[1]{\llbracket{#1}\rrbracket}

\newcommand{\dontshowuselessoldcrap}{x}
\def\clap#1{\hbox to 0pt{\hss#1\hss}}

\makeatletter
\renewcommand{\andify}{%
  \nxandlist{\unskip, }{\unskip{} \@@and~}{\unskip{} \@@and~}}
\def\author@andify{%
  \nxandlist {\unskip ,\penalty-1 \space\ignorespaces}%
    {\unskip {} \@@and~}%
    {\unskip \penalty-2 \space \@@and~}%
}
\let\@wraptoccontribs\wraptoccontribs
\makeatother
\title[What does a typical metric space look like?]{What does a typical metric space \\ look like?}

\author[Kozma]{Gady Kozma}
\address{Department of Mathematics, the Weizmann Institute of Science, Rehovot 76100, Israel}
\email{gady.kozma@weizmann.ac.il}
\author[Meyerovitch]{Tom Meyerovitch}
\address{Department of Mathematics, Ben Gurion University of the Negev,  Be'er Sheva 8410501, Israel }\email{mtom@bgu.ac.il}\urladdr{http://www.math.bgu.ac.il/~mtom}
\author[Peled]{Ron Peled}
\thanks{The research of G.K.\ was supported by the ISF, by the Jesselson foundation, and by Paul and Tina Gardner. The research of T.M.\ was supported by ISF Grants~626/14 and~1052/18. The research of R.P.\ was supported by ISF Grants~1048/11,~861/15 and~1971/19, by IRG Grant SPTRF, and ERC grant LocalOrder. The research of W.S.\ was supported by ISF Grants~1147/14 and~1145/18. }
\address{School of Mathematical Sciences, Tel Aviv University, Tel Aviv 6997801, Israel}
\email{peledron@tauex.tau.ac.il}
\urladdr{http://www.math.tau.ac.il/~peledron}
\author[Samotij]{Wojciech Samotij}
\address{School of Mathematical Sciences, Tel Aviv University, Tel Aviv 6997801, Israel}
\email{samotij@tauex.tau.ac.il}
\urladdr{http://www.math.tau.ac.il/~samotij}

\dedicatory{To the memory of Dima Ioffe, our friend and colleague.\\
Mathematical physicist, probabilist and a dear person\\
who freely shared his good advice and insight.\\
His passing is a great loss to our community.}

\begin{document}

\begin{abstract}
  The collection $\M_n$ of all metric spaces on $n$
  points whose diameter is at most $2$ can naturally be viewed as a compact convex subset of  $\R^{\binom{n}{2}}$, known as the metric polytope.
  In this paper, we study the metric polytope for large $n$ and show that it is close to the cube $[1,2]^{\binom{n}{2}} \subseteq \M_n$ in the following two senses.
  First, the volume of the polytope is not much larger than that of the cube, with the following quantitative estimates:
  \[
    (\nicefrac{1}{6}+o(1))n^{3/2} \le \log \vol(\M_n)\le O(n^{3/2}).
  \]
  Second, when sampling a metric space from $\M_n$ uniformly at random, the minimum distance is at least $1 - n^{-c}$ with high probability, for some $c > 0$.
  Our proof is based on entropy techniques. We discuss alternative approaches to estimating the volume of $\M_n$ using exchangeability, Szemer\'edi's regularity lemma, the hypergraph container method, and the K\H{o}v\'ari--S\'os--Tur\'an theorem.
\end{abstract}

\maketitle

\section{Introduction}
For a positive integer $n$, let
$\br{n} :=\{1, \dotsc, n\}$ and let $\binom{\br{n}}{2}$ be the set of all
unordered pairs of distinct elements in $\br{n}$. A finite metric space
on $n\ge 2$ points can be regarded as an array $(\dist{i}{j})$ with
$\{i,j\}\in \binom{\br{n}}{2}$, where $\dist{i}{j}$ denotes the
distance between the $i^{\text{th}}$ and $j^{\text{th}}$ points in
the space. Such a metric space may also be regarded as an element of
$\R^{\binom{n}{2}}$ satisfying certain restrictions among its
coordinates. Specifically, the set of all such metric spaces is the
cone
\[
\C_n:=\{(\dist{i}{j})\in\R^{\binom{n}{2}} : \dist{i}{j} >
0\text{ and }\dist{i}{j}\le \dist{i}{k} + \dist{k}{j}\text{ for all
$i,j,k$}\}.
\]
Our goal in this work is to study a
`uniformly chosen metric space on $n$ points'. This is interpreted as a metric space sampled according to the Lebesgue measure from a suitable bounded subset of $\C_n$. There are several natural choices for such a bounded subset. In this work, we focus on the diameter normalisation, that is, we bound the maximal diameter of the space from above.
We thus define the \emph{metric polytope}
\[
\M_n:=\{(\dist{i}{j})\in \C_n : \dist{i}{j}\le 2\text{ for all $i,j$}\}.
\]
The specific upper bound on the diameter amounts only to a scaling factor, with the constant two chosen to simplify some of the later expressions.

Understanding the structure of a uniformly chosen metric space in $\M_n$ is
intimately related to understanding the volume of $\M_n$. By
construction, we have the trivial upper bound
\begin{equation}\label{eq:trivial_volume_upper_bound}
  \vol(\M_n)\le 2^{\binom{n}{2}}.
\end{equation}
To obtain a lower bound, we make the following observation: Any triple
$x,y,z\in[1,2]$ satisfies the triangle inequality $x \le y + z$ and,
consequently, $\M_n$ contains the cube $[1,2]^{\binom{n}{2}}$.
This yields the lower bound
\begin{equation}\label{eq:trivial_volume_lower_bound}
  \vol(\M_n)\ge 1.
\end{equation}
The precise behaviour of the volume $\vol(\M_n)$ seems
difficult to study. For instance, while intuitive, we do not know
whether $\vol(\M_{n+1})\ge \vol(\M_n)$ for all $n$ (see also
Section~\ref{sec:further_questions}). It is thus interesting to note
that at least the `radius' $\vol(\M_n)^{1/\binom{n}{2}}$ exhibits
some regularity.
\begin{prop}\label{prop:decreasing_radius}
  The sequence $n\mapsto\vol(\M_n)^{1/\binom{n}{2}}$ is
  non-increasing.
\end{prop}
The proposition is deduced from Shearer's inequality, see Section~\ref{sec:monotonicity}.
It allows to obtain increasingly refined volume
estimates for $\vol(\M_n)$ via finite computations. For instance,
one may check that $\vol(\M_3)=4$ (see Figure~\ref{fig:M3})
and hence we have the inequality
\begin{equation}\label{eq:second_trivial_volume_upper_bound}
  \vol(\M_n)\le 4^{\frac{1}{3}\binom{n}{2}}\quad \text{for all $n\ge 3$},
\end{equation}
improving upon the trivial upper
bound~\eqref{eq:trivial_volume_upper_bound}. Mascioni~\cite{Ma05}
calculated $\vol(\M_4) = \frac{136}{15}$ and used it to deduce a bound
on $\vol(\M_n)$ that is stronger than
\eqref{eq:second_trivial_volume_upper_bound} but weaker than what
can be deduced from Proposition~\ref{prop:decreasing_radius}. The
proposition and~\eqref{eq:trivial_volume_lower_bound} imply that the limit
\begin{equation}\label{eq:limiting_volume_constant}
  \lim_{n\to\infty} \vol(\M_n)^{1 / \binom{n}{2}}
\end{equation}
exists, which raises the natural question of finding its value.

Our main result is that a uniformly chosen metric space is `almost
in $[1,2]^{\binom{n}{2}}$', as the following two theorems make
precise. Our first theorem shows that the limiting constant
\eqref{eq:limiting_volume_constant} equals one, that is,
\begin{equation}\label{eq:limit constant is one}
  \vol(\M_n) = 2^{o(n^2)}\quad\text{as $n\to\infty$}.
\end{equation}
In fact, our analysis goes much further and determines the \emph{second-order term} in the logarithm of the volume up to a multiplicative constant.\footnote{Note that~\eqref{eq:main_volume_estimates} hides the first-order term $(2/2)^{\binom{n}{2}}$, which would become $(M/2)^{\binom{n}{2}}$ if we chose to normalise the diameter of $\M_n$ to be $M$, instead of $2$.}

\begin{thm}\label{thm:volume_estimate}
  The following asymptotic estimates hold as $n \to \infty$:
  \begin{equation}\label{eq:main_volume_estimates}
    \exp\big((\nicefrac{1}{6}-o(1))n^{3/2}\big) \le \vol(\M_n)\le \exp\big(Cn^{3/2}\big)
  \end{equation}
  for some absolute constant $C$.
\end{thm}

Our second theorem studies the minimum distance in a typical metric space in $\M_n$. Let $d$ be a uniformly sampled metric space from $\M_n$. Since
\begin{equation*}
  \P\left(\min_{i,j} d_{ij} > 1-\delta\right) \le \frac{(1+\delta)^{\binom{n}{2}}}{\vol(\M_n)},
\end{equation*}
the lower bound in Theorem~\ref{thm:volume_estimate} implies that, for any $a<\frac{1}{3}$,
\begin{equation}
  \label{eq:min-dij-lower}
  \P\left(\min_{i,j} d_{ij} \le 1-\frac{a}{\sqrt{n}}\right) \to 1\quad\text{as $n\to\infty$}.
\end{equation}
Complementing this fact, we show that, in a typical metric space, the minimum distance is polynomially close to one.

\begin{thm}\label{thm:minimal_distance}
  There exist constants $C,c>0$ such that, for all $n\ge 2$, if $d$ is a uniformly sampled metric space from $\M_n$, then
  \begin{equation*}
    \P\left(\min_{i,j} d_{ij} \le 1-n^{-c}\right)\le C n^{-c}.
  \end{equation*}
\end{thm}
\begin{figure}
\begin{centering}
\includegraphics{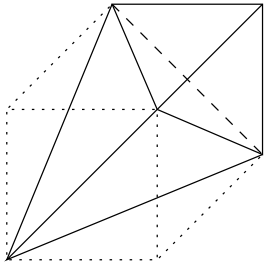}
\end{centering}
\caption{$\M_3$ inside $[0,2]^3$.\label{fig:M3}}
\end{figure}

It would be interesting to find the typical order of $1 - \min_{i,j} d_{ij}$, see also Section~\ref{sec:further_questions}.  Our proof shows that it is at most~$n^{-1/30}$.

\subsection*{Reader's guide}

The heart of this work is the proof of the upper bound on the volume of $\M_n$ in~\eqref{eq:main_volume_estimates},
which relies on entropy techniques; a conceptual outline of our argument is presented in the next section. We review the relevant background on differential entropy in Section~\ref{sec:entropy}.  The volume estimate itself is then proved in Section~\ref{sec:volume-upper-bound}.  One of the main ingredients in the argument is an upper bound on the maximum entropy of a vector of independent random variables that is almost supported in a given compact, convex set;  this result is derived in Section~\ref{sec:entropy-maximising product distributions}.
Several alternative approaches to proving upper bounds on the volume of $\M_n$, which lead to results weaker
than Theorem~\ref{thm:volume_estimate}, are reviewed in Section~\ref{sec:other_approaches}.

Our proof of the lower bound on the volume is a fairly simple application of the Local Lemma of Erd\H{o}s and Lov\'asz~\cite{ErLo75}, see Section~\ref{sec:lower}.

The starting point for the proof of Theorem~\ref{thm:minimal_distance}, given in Section~\ref{sec:distance}, is an upper bound on the probability
that the distance between a \emph{fixed} pair of points is shorter than one (Proposition~\ref{prop:distance-lower-tail}), which is a by-product of
our proof of the upper bound on the volume of $\M_n$. The assertion of the theorem is then deduced via elementary, but nontrivial,
combinatorial arguments.

Section~\ref{sec:discussion_and_open_questions} contains some further discussion and a selection of open questions.

\subsection*{A remark on precedence}

This paper has been long in writing and a number of results have appeared in the interim, notably Mubayi and Terry~\cite{mubayi2019discrete} and Balogh and Wagner~\cite{BalWag16}, who considered the number of metric spaces with distances in the discrete set $\{1, \dotsc, M\}$. The methods of~\cite{BalWag16} also yielded the upper bound $\vol(\M_n) \le \exp(n^{11/6+o(1)})$. (We elaborate on the relation between the discrete and the continuous model in Section~\ref{sec:discrete-problem}.) These papers have kindly acknowledged our precedence, but, for fairness, it should be noted that we did not have the upper bound of Theorem~\ref{thm:volume_estimate} then, only a bound of the form $\vol(\M_n)\le\exp(n^{2-c})$; in particular, the volume bound of~\cite{BalWag16} was stronger than ours. Several months before we streamlined the entropy argument underlying the proof of the upper bound in Theorem~\ref{thm:volume_estimate} to yield the estimate $\vol(\M_n) \le \exp(Cn^{3/2})$, in a joint work with Rob Morris, we found a more efficient version of the argument of Balogh and Wagner~\cite{BalWag16}, based on the method of hypergraph containers, that gives the estimate $\vol(\M_n) \le \exp\big( C n^{3/2} (\log n)^3 \big)$; we present a detailed sketch of this argument in Section~\ref{sec:container-method}.

\subsection*{Further directions and related work}

We believe that the general method of relating entropy and independence that underlies our proof of Theorem~\ref{thm:volume_estimate} will find many further applications. In particular, the first and the fourth named authors adapted the methods of this work, and combined them with the arguments underlying the proofs of the hypergraph container theorems~\cite{BalMorSam, SaxTho}, to study lower tails of random variables that can be expressed as polynomials of independent Bernoulli random variables~\cite{KozSam}.

Similar ideas of relating entropy and independence were used by Tao~\cite[Lemma~4.3]{MR2212136} to develop a probabilistic interpretation of Szemer\'edi's regularity lemma. Concurrently with the writing of this paper, Ellis, Friedgut, Kindler, and Yehudayoff~\cite{EllFriKinYeh15} used a related approach to prove stability versions of the Loomis--Whitney inequality and the more general Uniform Cover inequality.
The pigeonhole principle argument that appears in our proof outline below (see also Lemma~\ref{lem:h_m_bound_with_difference}) is somewhat reminiscent of the  Lov\'asz--Szeg\'edy Hilbert space regularity lemma, see the proof of \cite[Lemma~4.1]{MR2306658} (we thank Bal\'azs R\'ath for pointing out this connection).

\subsection{Proof outline}\label{sec:proof outline}

Suppose that $d$ is a uniformly chosen metric space from $\M_n$. Conceptually, our argument consists of three steps.

\subsubsection*{Step I (conditioning)}

We say that a subset $F \subseteq \binom{\br{n}}{2}$ has the \emph{conditioned almost independence} property if the following holds: Conditioned on all the distances $d_f$ with $f \in F$, for each triangle $\{i,j,k\}$ whose edges lie outside of $F$, the distances $d_{ij}$, $d_{ik}$, and $d_{jk}$ become close to mutually independent.

The goal of the first step is to find a `small' set $F$ with the above property.
In order to show this, for $m \ge 0$, define the set
\[
  F_m := \Big\{\{s,t\} \in \binom{\br{n}}{2} : \max\{s,t\} > n-m\Big\}
\]
and examine the conditional entropy
\[
  h(F_m) := H\big(d_{12} \mid\{d_f : f \in F_m\}\big).
\]
Since $F_{m+1} \supseteq F_m$, monotonicity of conditional entropy implies that the sequence $m \mapsto h(F_m)$ is nonincreasing.
Moreover, it is not difficult to bound $h(F_0)$ from above and $h(F_{\sqrt{n}})$ from below by absolute constants.
Thus, the pigeonhole principle produces an $m_0$ with $0 \le m_0 \le \sqrt{n}$ for which
\begin{equation}
  \label{eq:hm-pigeonhole}
  h(F_{m_0})-h(F_{{m_0}+1}) \le \frac{C}{\sqrt{n}}
\end{equation}
for some absolute constant $C$. The set $F$ described above is taken to be $F_{m_0}$. Its cardinality is at most $m_0 n \le n^{3/2}$. We now argue that $F$ has the conditioned almost independence property. Since $\{1,n-m_0\}, \{2,n-m_0\} \in F_{m_0+1} \setminus F_{m_0}$, inequality~\eqref{eq:hm-pigeonhole} gives (again using the monotonicity of conditioned entropy)
\begin{equation*}
  h(F) - h\big(F \cup \{\{1,n-m_0\}, \{2,n-m_0\}\}\big) \le \frac{C}{\sqrt{n}}.
\end{equation*}
Symmetry considerations show that, in fact, for every ordered triple of distinct $i, j, k \in \br{n-m_0}$, we have
\begin{equation}
  \label{eq:dijk-relative-entropy}
  H\big(d_{ij} \mid \{d_f : f \in F\}\big) - H\big(d_{ij} \mid \{d_f : f \in F\} \cup \{d_{ik}, d_{jk}\}\big) \le \frac{C}{\sqrt{n}}.
\end{equation}
Inequality~\eqref{eq:dijk-relative-entropy} is the notion of almost independence that we need. It may be conveniently restated in terms of the average \emph{Kullback--Leibler divergence} between the conditional (on all $d_f$ with $f \in F$) joint distribution of $d_{ij}$, $d_{ik}$, $d_{jk}$ and the product of the (conditional) marginal distributions of $d_{ij}$, $d_{ik}$, and $d_{jk}$:
\begin{equation}
  \label{eq:dijk-relative-DKL}
  \E\big[\DKL{(d_{ij}, d_{ik}, d_{jk})}{d_{ij} \times d_{ik} \times d_{jk}}\big] \le \frac{C}{\sqrt{n}}.
\end{equation}

\subsubsection*{Step II (subadditivity)}

Since $d$ is a uniformly sampled metric space from $\M_n$,
\[
  H(d) = \log\big(\vol(\M_n)\big).
\]
Using the chain rule for conditional entropy, we may write
\begin{equation}
  \label{eq:chain-rule-entropy}
  H(d) = \underbrace{H(\{d_f : f \in F\})}_{\alpha} + \underbrace{H(\{d_f : f \notin F\} \mid \{d_f : f \in F\})}_{\beta}.
\end{equation}
Since $d_f \in [0,2]$ for every $f \in F$, we have $\alpha \le |F| \log 2$. By considering an arbitrary Steiner triple system on $n - O(1)$ vertices, one sees that the complement of $F$ can be partitioned into a family $\T$ of edge-disjoint triangles and a leftover set of pairs $G$ with $|G| \le C (|F| + n)$. Using subadditivity of entropy,
\begin{equation}
  \label{eq:beta-upper-bound}
  \beta \le |G| \log 2 + \sum_{\{i,j,k\} \in \T} H(d_{ij}, d_{ik}, d_{jk} \mid \{d_f : f \in F\}).
\end{equation}

\subsubsection*{Step III (entropy-maximising distributions)}

Combining the above two steps, we arrive at the problem of bounding $H(d_{ij}, d_{ik}, d_{jk} \mid \{d_f : f \in F\})$, which, by conditioned almost independence, amounts to estimating the largest entropy of a vector that is supported on~$\M_3$ and satisfies~\eqref{eq:dijk-relative-DKL}. We first observe that~\eqref{eq:dijk-relative-DKL} implies the following inequality:
\[
  \E\big[\P(d_{ij} \times d_{ik} \times d_{jk} \notin \M_3)\big] \le \frac{C}{\sqrt{n}},
\]
see Lemma~\ref{lem:Kullback_Leibler_and_support}. This allows us to bound $H(d_{ij}, d_{ik}, d_{jk} \mid \{d_f : f \in F\})$ by the largest entropy of a vector of (fully) independent random variables that is almost supported on~$\M_3$. To this end, we prove a general statement (Theorem~\ref{thm:entropy_bound_for_almost_independent}) showing that the largest entropy of a vector of independent random variables that is almost supported on a convex set $\poly$ cannot be much larger than the logarithm of the volume of the largest box contained in $\poly$. In the case $\poly = \M_3$, the (unique) largest such box is $[1,2]^3$ (Lemma~\ref{lem:independent_max_volume}) and, consequently, the entropy cannot be much larger than zero.

The three steps suffice to show that the volume of $\M_n$ is $\exp(o(n^2))$, that is, that the limiting constant in~\eqref{eq:limiting_volume_constant} is one. Moreover, the various error terms are polynomially related and the above argument shows the quantitative estimate $\vol(\M_n) \le \exp(Cn^{2-c})$ for an explicit $c > 0$. To obtain the sharp exponent $3/2$, as in the statement of Theorem~\ref{thm:volume_estimate}, several enhancements to the above argument are made. In particular, the following two bounds are proved:
\begin{align}
  \label{eq:sketch-chain-rule-Fm}
  \log(\vol(\M_n)) & \le \sum_{m=0}^{n-2} |F_{m+1} \setminus F_m| \cdot h(F_m) \le n \cdot \sum_{m=0}^{n-2} h(F_m), \\
  \label{eq:hm-upper-bound}
  h(F_m) & \le C \big(h(F_m) - h(F_{m+1})\big)^{1/3}.
\end{align}
The bound~\eqref{eq:hm-upper-bound} implies that $h(F_m) \le C' (m+1)^{-1/2}$, as shown in Lemma~\ref{lem:h_m_bound}, which gives the claimed estimate after substituting it into~\eqref{eq:sketch-chain-rule-Fm}. The estimate~\eqref{eq:sketch-chain-rule-Fm} improves upon the subadditivity step, making use of the symmetry inherent in the specific choice of the sets~$F_m$, see~\eqref{eq:volume_decomposition}. Inequality~\eqref{eq:hm-upper-bound} is obtained using a more careful analysis in steps one and three above.

\section{Entropy and almost independence}
\label{sec:entropy}

\subsection{Differential entropy}
\label{sec:differential-entropy}

We now recall the notion and some properties of differential
entropy, the entropy of continuous random variables. Readers who are
used to the entropy of discrete random variables (Shannon's entropy)
should keep in mind that in the continuous case entropies can be
either positive or negative, the value $0$ plays no special role.

Given an absolutely continuous probability measure $\mu$ on $\R^k$
with density $f$ and a~random variable $X \sim \mu$, the
\emph{differential entropy} (or simply \emph{entropy}) of $\mu$ (or
of $X$) is defined as
\begin{equation}\label{eq:H_def}
  H(\mu) := H(X) := -\int \log(f(x)) f(x)dx = -\int \log\left(f(x)\right)d\mu(x),
\end{equation}
whenever the above integral is well-defined. As is customary, if random variables
$X_1,\dotsc,X_m$ have a joint density function, we will write $H(X_1, \dotsc, X_m)$
for the entropy of the random vector $(X_1, \dotsc, X_m)$.  Throughout the paper,
we write $\log$ to denote the natural logarithm.

Observe that if $X$ takes values in a compact set $K \subseteq \R^k$, then by Jensen's
inequality,
\begin{equation}\label{eq:entropy_on_compact}
  \begin{split}
    H(X) & = -\int_K \log(f(x)) f(x)dx = \int_{K\cap\{f>0\}} \log\left(\frac{1}{f(x)}\right) f(x) dx \\
    & \le \log\left(\int_{K \cap \{f > 0\}} \frac{1}{f(x)} f(x)dx\right) \le \log(\vol(K)).
  \end{split}
\end{equation}

Differential entropy may be negative, e.g., if $\vol(K) < 1$ above.
It could even happen that $H(X) = -\infty$. However, one easily
checks that if the density $f$ is bounded, then $H(X) >
-\infty$. In view of this, for the sake of simplicity, from now on we
focus on probability measures on $\R^k$ that are
compactly supported and admit a bounded density. We denote
the family of all such measures by $\BDM(\R^k)$. We emphasize that $\BDM(\R^k)$
is closed under projections.

\begin{fact}
  \label{fact:hereditary-density}
  If $X\in\R^{k_1}$ and $Y\in\R^{k_2}$ have a joint distribution in $\BDM(\R^{k_1+k_2})$,
  then the distribution of $X$ is in $\BDM(\R^{k_1})$.
\end{fact}

Keeping in mind the
case of equality in Jensen's inequality and applying it to
\eqref{eq:entropy_on_compact}, let us note the following for future
reference.

\begin{lem}
  \label{lem:entropy-compact-support}
  If the distribution of a random variable $X$ is in $\BDM(\R^k)$ and $X$ takes values in a compact set $K$, then
  \[
  -\infty < H(X) \le \log(\vol(K)).
  \]
  The second inequality holds with equality if and only if $X$ is uniform on $K$.
\end{lem}

Further use is made of the following generalisation of~\eqref{eq:entropy_on_compact} for which we also provide a quantitative `stability' estimate.

\begin{lem}\label{lem:entropy_bound_for_sub_probability}
  Let $K\subseteq\R^k$ be a bounded measurable set and let $f\colon\R^k\to[0,\infty)$ be a bounded measurable function. Set $p := \int_K f(x)dx$. Then
  \begin{equation}\label{eq:entropy on bounded set}
    -\int_K f(x)\log(f(x))dx \le
    p\log\left(\frac{\vol(K)}{p}\right),
  \end{equation}
  where we interpret the right-hand side as $0$ if $p=0$, and define $0\log 0=0$ for the left-hand side.

  Moreover, if $K$ admits a partition $K = K_1 \cup K_2$ for measurable $K_1, K_2$ and either
  \[
  \frac{\int_{K_1} f(x)dx}{\int_K f(x)dx} \le \frac{1}{10} \cdot \frac{\vol(K_1)}{\vol(K)}\quad\text{or}\quad\frac{\int_{K_1} f(x)dx}{\int_K f(x)dx} \ge 10 \cdot \frac{\vol(K_1)}{\vol(K)}
  \]
  then
  \begin{multline}
    \label{eq:entropy-bound-non-uniform-sub-probability}
    -\int_K f(x)\log(f(x)) \, dx \le p\log\left(\frac{\vol(K)}{p}\right)\\ - \frac{p}{4} \cdot \max\left\{\frac{\vol(K_1)}{\vol(K)},\frac{\int_{K_1} f(x)dx}{\int_K f(x)dx}\right\}.
  \end{multline}
\end{lem}
\begin{proof}
  The bound is trivial if $p=0$. Otherwise, since $g := f/p$ satisfies
  \[
    -\int_Kf(x) \log(f(x)) \, dx = - p\int_Kg(x) \log(g(x)) \, dx - p\log p,
  \]
  it suffices to prove the results when $p=1$, as we now assume.

  The estimate~\eqref{eq:entropy on bounded set} follows from the same calculation as in~\eqref{eq:entropy_on_compact}.

  We proceed to prove~\eqref{eq:entropy-bound-non-uniform-sub-probability}. Set $r := \frac{\int_{K_1} f(x)dx}{\int_K f(x)dx} = \int_{K_1} f(x)dx$ and $q := \frac{\vol(K_1)}{\vol(K)}$ so that either $r\le \frac{1}{10}q$ or $r\ge 10q$. Invoking~\eqref{eq:entropy on bounded set} twice yields
  \begin{align*}
    - \int_K f(x) \log(f(x))dx & =- \left(\int_{K_1} + \int_{K_2}\right) f(x) \log(f(x))dx \\
    & \le r \log\left(\frac{\vol(K_1)}{r}\right) +  (1-r) \log\left(\frac{\vol(K_2)}{(1-r)}\right),
  \end{align*}
  which, by the definition of $q$, is easily seen to be equivalent to:
  \begin{multline*}
    \log(\vol(K)) + \int_K f(x) \log(f(x))dx \\
    \ge r \log\left(\frac{r}{q}\right) + (1-r) \log \left(\frac{1-r}{1-q}\right).
  \end{multline*}
  Define $D(x) := x\log(x/q) + (1-x)\log((1-x)/(1-q))$, so that the right-hand side above is $D(r)$. (An observant reader will recognise that $D(r)$ is the Kullback--Leibler divergence between Bernoulli random variables with success probabilities $r$ and $q$, respectively.) Note that $D(q) = 0$ and that
  \begin{equation}
    \label{eq:DKL-derivative}
    D'(x) = \log\left(\frac{x}{q}\right) - \log\left(\frac{1-x}{1-q}\right).
  \end{equation}
  Observe further that $D'(x)$ is an increasing function of $x$ and $D'(q)=0$. This implies that, in the case where $r \le q/10$,
  \[
    D(r) \ge \frac{q-r}{2} \cdot \left(-D'\left(\frac{r+q}{2}\right)\right) \ge \frac{9q}{20} \cdot \log \left(\frac{20}{11}\right) \ge \frac{q}{4}
  \]
  and, in the case where $r\ge 10 q$,
  \[
    D(r) \ge \frac{r-q}{2} \cdot D'\left(\frac{r+q}{2}\right) \ge \frac{9r}{20} \cdot \log \left(\frac{11}{2}\right) \ge \frac{r}{4}.\qedhere
  \]
\end{proof}

\subsection{Conditional entropy}
\label{sec:conditional-entropy}

For random variables $X \in \R^{k_1}$ and $Y \in \R^{k_2}$ having a joint density $f$ on
$\R^{k_1 + k_2}$, the \emph{conditional entropy} of $X$ given $Y$,
denoted by $H(X \mid Y)$, is the average over $Y$ of the entropy of
the conditional distribution of $X$ given~$Y$. Formally, if we write
\begin{equation}
  \label{eq:conditioned-density}
  g(y):=\int f(x,y)\,dx\quad\text{and}\quad f_y(x):=\frac{f(x,y)}{g(y)},
\end{equation}
with $f_y$ defined for almost every $y$ with respect to the
distribution of~$Y$, then
\begin{equation}
  \label{eq:conditional_entropy_def}
  H(X \mid Y) := \int H(X_{\{Y = y\}}) g(y)\,dy = -\iint
  \log\left(f_y(x)\right) f_y(x) dx\, g(y)\,dy,
\end{equation}
whenever the above integral is well defined, where $X_{\{Y=y\}}$ denotes the random variable $X$ conditioned on the event $\{Y=y\}$.
Note that, using Fact~\eqref{fact:hereditary-density}, $H(X \mid Y)$ is well defined and finite whenever the joint distribution of $X$ and $Y$ is in $\BDM(\R^{k_1+k_2})$.

\subsection{Kullback--Leibler divergence}
\label{sec:KL-divergence}

Given two absolutely continuous probability measures $\mu$ and $\nu$
on~$\R^k$ with densities $f$ and $g$, respectively, and random
variables $X \sim \mu$ and $Y \sim \nu$, we define the
\emph{Kullback--Leibler divergence} between $\mu$ and $\nu$ (or between $X$
and $Y$) by
\begin{equation}
  \label{eq:DKL-def}
  \DKL{\mu}{\nu} := \DKL{X}{Y} := \int \log\left( \frac{f(x)}{g(x)}\right)f(x)\,dx.
\end{equation}
Since $\log y  \ge 1 - 1/y$ for every $y > 0$, we see that
\begin{equation}
  \label{eq:DKL-nonnegative}
  \DKL{\mu}{\nu} \ge \int \left(1-\frac{g(x)}{f(x)}\right) f(x)\, dx = 0.
\end{equation}
In particular, the integral in~\eqref{eq:DKL-def} is well defined,
possibly as $+\infty$.

We note a simple relation between the entropy of a pair of random variables and their Kullback--Leibler divergence.
Let $X$ and $Y$ be random variables taking values in $\R^{k_1}$ and $\R^{k_2}$, respectively, with a joint distribution in $\BDM(\R^{k_1+k_2})$. A direct calculation shows that
\begin{equation}\label{eq:entropy_KL_relation}
  H(X,Y) = H(X) + H(Y) - \DKL{(X,Y)}{X\times Y},
\end{equation}
where we use the notation $X\times Y$ to denote a random variable whose distribution is the product of the marginal distribution of $X$ and the marginal distribution of $Y$; in other words, $X \times Y$ is composed of independent copies of $X$ and $Y$.

\subsection{Properties of entropy}
\label{sec:properties-entropy}

We now recall some standard facts about entropy.

\begin{lem}\label{prop:basic_entropy_properties}
  Suppose $X\in\R^{k_1}$, $Y\in\R^{k_2}$, $Z\in\R^{k_3}$ have a joint distribution in $\BDM(\R^{k_1+k_2+k_3})$. Then
  \begin{enumerate}[label={\rm(\textit{\roman*})}]
    \item
      \label{item:entropy-prop-1}
      $H(X,Y) = H(X \mid Y) + H(Y)$,
    \item
      \label{item:entropy-prop-2}
      $H(X \mid Y)\le H(X)$,
    \item
      \label{item:entropy-prop-3}
      $H(X,Y) \le H(X) + H(Y)$,
   \item
      \label{item:entropy-prop-4}
      $H(X \mid Y,Z)\le H(X \mid Y)$.
  \end{enumerate}
\end{lem}
\begin{proof}
  Note first that our assumption on the joint distribution of $X$, $Y$, and $Z$ implies that all entropies appearing
  in the statement of the lemma are well defined, see Fact~\ref{fact:hereditary-density} and Lemma~\ref{lem:entropy-compact-support}. To see~\ref{item:entropy-prop-1}, let $f$ be the joint density of $X$ and $Y$ and define $g$ and $f_y$
  as in~\eqref{eq:conditioned-density}. Then
  \[
  \begin{split}
    \lefteqn{H(X,Y)  = - \iint \log(f(x,y)) f(x,y)\, dx\, dy}\quad & \\
    & = - \iint \log(f_y(x)g(y)) f_y(x)g(y)\, dx\,dy \\
    & = - \iint \log(f_y(x)) f_y(x)\, dx\, g(y)\, dy - \int \log(g(y)) g(y) \int f_y(x)\,dx\,dy \\
    & = H(X \mid Y) + H(Y).
  \end{split}
  \]
  Inequality~\ref{item:entropy-prop-3} is a direct consequence of~\eqref{eq:DKL-nonnegative} and~\eqref{eq:entropy_KL_relation}
  whereas~\ref{item:entropy-prop-2} follows immediately from~\ref{item:entropy-prop-1} and~\ref{item:entropy-prop-3}.
  To see~\ref{item:entropy-prop-4}, let $f$ be the joint density of $X$, $Y$, and $Z$, and let
  \[
  g(y) := \iint f(x,y,z)\, dx\, dz.
  \]
  It is not hard to see that
  \begin{align*}
    H(X \mid Y,Z) &= \int H(X_{\{Y=y\}} \mid Z_{\{Y=y\}}) g(y) dy \\
    &\le \int H(X_{\{Y=y\}}) g(y) dy = H(X \mid Y),
  \end{align*}
  where the inequality follows from~\ref{item:entropy-prop-2}.
\end{proof}

A powerful tool for comparing entropies is the following inequality
originally proved by Shearer (see~\cite{shearer_ineq}) for Shannon's
entropy. As the literature usually deals with the discrete case, we provide a short proof based on the treatment in~\cite{AlSp}.

\begin{thm}[Shearer's inequality]\label{thm:Shearers_inequality}
  Let $X_1,\dotsc,X_m$ be random variables with a joint density which is bounded and compactly supported.
  Let $\II \subseteq 2^{\br{m}}$ be a collection of subsets which $r$-covers the set $\br{m}$, i.e.,~has the property
  that for each $i \in \br{m}$,
  \begin{equation}\label{eq:r_cover}
  |\{I\in \II : i\in I\}| = r.
  \end{equation}
Then
\begin{equation}
  \label{eq:Shearer}
  H(X_1,\dotsc, X_m) \le \frac{1}{r}\sum_{I \in \II} H(\{X_i: i\in I\}).
\end{equation}
\emph{(since we did not preclude $\emptyset\in\II$, let us define that the entropy of an empty collection of random variables is zero).}
\end{thm}
\begin{proof}
  We prove the statement by induction on $r$. The case $r=1$ follows immediately from~\ref{item:entropy-prop-3} in Proposition~\ref{prop:basic_entropy_properties}.
  Suppose now that $r > 1$. If $\br{m} \in \II$, then we easily obtain~\eqref{eq:Shearer} invoking the inductive assumption with $\II$
  replaced by the $(r-1)$-cover $\II \setminus \{\br{m}\}$. Otherwise, assume $I_1, I_2 \in \II$ satisfy that both $I_1 \setminus I_2$
  and $I_2 \setminus I_1$ are non-empty. It follows
  from~\ref{item:entropy-prop-4} in
  Proposition~\ref{prop:basic_entropy_properties} that
\begin{equation}\label{eq:17.5}
H(I_1\setminus I_2\mid I_2)\le H(I_1\setminus I_2\mid I_1\cap I_2)
\end{equation}
where we denote $H(I)$ as a short for $H(\{X_i:i\in I\})$, and
similarly for conditioned entropies.
  Consequently, by~\ref{item:entropy-prop-1} in
  Proposition~\ref{prop:basic_entropy_properties},
\begin{multline*}
H(I_1\cup I_2)+H(I_1\cap I_2)
  \stackrel{\textrm{\ref{item:entropy-prop-1}}}{=}
H(I_1\setminus I_2 \mid I_2) + H(I_2) + H(I_1\cap I_2 ) \\
  \stackrel{\textrm{(\ref{eq:17.5})}}{\le}
H(I_1\setminus I_2 \mid I_1\cap I_2) + H(I_2) + H(I_1\cap I_2 )
  \stackrel{\textrm{\ref{item:entropy-prop-1}}}{=}
H(I_1)+H(I_2)
\end{multline*}
  If we now replace $I_1$ and $I_2$ with $I_1 \cup I_2$ and $I_1 \cap I_2$, then $\II$ remains an $r$-cover and the sum in the right-hand side
  of~\eqref{eq:Shearer} can only decrease. It is clear that after a finite number of such modifications we will eventually arrive at the case when
  $\br{m}\in \II$.
\end{proof}
We remark that, due to the fact that differential entropy may be negative and unlike the Shannon entropy case (the entropy of discrete random variables), inequality \eqref{eq:Shearer} need not hold when the equals sign in the $r$-cover condition \eqref{eq:r_cover} is changed to a greater-or-equal sign.

\subsection{A triangle inequality for the Kullback--Leibler divergence}

The proof of Theorem~\ref{thm:volume_estimate} will require the following simple `triangle inequality' for Kullback--Leibler divergences.

\begin{lem}
  \label{lem:DKL-triangle-ineq}
  Suppose that $X$, $Y$, and $Z$ are $\R$-valued random variables with a joint distribution in $\BDM(\R^3)$. Then
 \begin{multline*}
   \DKL{(X,Y,Z)}{X\times Y\times Z} \\
   \le \DKL{(X,Y,Z)}{X\times (Y,Z)} + \DKL{(X,Y,Z)}{(X,Y)\times Z}.
 \end{multline*}
\end{lem}
\begin{proof}
  The definition of Kullback--Leibler divergence gives
  \begin{multline*}
    \DKL{(X,Y,Z)}{X\times Y\times Z} \\
    =\DKL{(X,Y,Z)}{X\times (Y,Z)} + \DKL{(Y,Z)}{Y\times Z}.
  \end{multline*}
  The second term in the right-hand side may be bounded from above, using~\eqref{eq:entropy_KL_relation} and Lemma~\ref{prop:basic_entropy_properties}~\ref{item:entropy-prop-4}, as follows:
  \begin{align*}
    \DKL{(Y,Z)}{Y\times Z} & = H(Z) - H(Z \mid Y) \le H(Z) - H(Z \mid X,Y) \\
    & = \DKL{(X,Y,Z)}{(X,Y)\times Z}.\qedhere
  \end{align*}
\end{proof}

\subsection{Relations between entropy and independence}
\label{sec:entr-total-vari}

As explained above, a key step in our proof of
Theorem~\ref{thm:volume_estimate} is showing that the individual
distances in a uniformly sampled metric space from $\M_n$ become
almost independent random variables after we condition on the values
of some small fraction of all $\binom{n}{2}$ distances. We shall
establish this almost independence property by bounding the
entropies of various vectors of distances in the random metric
space. The connection between almost independence and entropy will
be provided by the following lemma, relating the Kullback--Leibler divergence of two measures with the difference of their supports. The lemma will be used to bound from above the error term in the upper bound on entropy given by Theorem~\ref{thm:entropy_bound_for_almost_independent}, see Claim~\ref{claim:PXnotinM3}.
\begin{lem}\label{lem:Kullback_Leibler_and_support}
  Let $\mu, \nu$ be probability measures
  in $\BDM(\R^k)$. Then
  \begin{equation}
    \label{eq:mu-nu-supp-diff}
    \DKL{\mu}{\nu}\ge \sup\{\nu(A) : A\subseteq\R^k\text{ Borel satisfying }
    \mu(A) = 0\}.
  \end{equation}
\end{lem}
\begin{proof}
  Denote by $f(x)$ the density of $\mu$ and by $g(x)$ the density of $\nu$. Let $A\subseteq\R^k$ be a Borel subset and suppose that $\mu(A)
  = 0$. Recalling that $\log(y)\ge 1 - \frac{1}{y}$ for all $y>0$ we
  conclude that
  \begin{align*}
    \DKL{\mu}{\nu} &= \int \log\left( \frac{f(x)}{g(x)}\right)f(x)\,dx = \int_{A^c} \log\left(
    \frac{f(x)}{g(x)}\right)f(x)\,dx\\
    &\ge \int_{A^c} \left(1 -
    \frac{g(x)}{f(x)}\right)f(x)\,dx = 1 - \nu(A^c) = \nu(A).\qedhere
  \end{align*}
\end{proof}

Observe that the quantity in the right-hand side of~\eqref{eq:mu-nu-supp-diff} is a lower bound on the \emph{total variation distance} between $\mu$ and $\nu$. Therefore, it seems natural to relate it to the Kullback--Leibler divergence between $\mu$ and $\nu$ using Pinsker's inequality~\cite{Pin60}, which states that\footnote{Originally, Pinsker proved~\eqref{eq:Pinsker} with the multiplicative constant $2$ replaced by $1/(2\log 2)$. The version stated in~\eqref{eq:Pinsker} was obtained somewhat later by Csisz\'ar~\cite{csiszar_inequality1966}, Kemperman~\cite{Ke69}, and Kullback~\cite{Ku67}.}
\begin{equation}
  \label{eq:Pinsker}
  \DKL{\mu}{\nu} \ge 2 \left(\TV{\mu}{\nu}\right)^2.
\end{equation}
This, in fact, was done in our proof of an earlier, weaker version of Theorem~\ref{thm:volume_estimate}.  While~\eqref{eq:Pinsker} is optimal for certain pairs of $\mu$ and $\nu$, in our setting, the more specialised Lemma~\ref{lem:Kullback_Leibler_and_support} yields much better dependence between the two quantities involved. Similar considerations are discussed in~\cite{EllFriKinYeh15}, which also uses a version of Lemma~\ref{lem:Kullback_Leibler_and_support} in place of Pinsker's inequality.

\section{Entropy-maximising product distributions}\label{sec:entropy-maximising product distributions}

The first part of this section is devoted to deriving an upper bound on the largest entropy
of a vector of independent random variables that is almost supported in a given compact, convex subset $\poly \subseteq \R^d$.
It turns out that this largest entropy is close to the logarithm of the largest volume of a box that is fully contained in $\poly$.
In a short, second part of the section, we compute this volume in the specific case that $\poly$ is the (closure of the) 3-dimensional metric polytope $\M_3$.
The results of this section are a central ingredient in the proof of the volume upper bound in Theorem~\ref{thm:volume_estimate}.

\subsection{Entropy-maximising product distributions on convex sets}

The following theorem is the main result of this section.

\begin{thm}\label{thm:entropy_bound_for_almost_independent}
Let $M>0$ and let $\poly\subseteq[-M,M]^d$, $d\ge 2$, be a closed, convex
set with non-empty interior. Let $V_0$ be the maximal volume of an axis-parallel box fully contained
in $\poly$, that is,
\begin{equation}\label{eq:maximal volume of box in polytope}
  V_0 := \max\left\{\prod_{i=1}^d (b_i - a_i) : [a_1,b_1]\times\cdots\times[a_d,b_d]\subseteq \poly\right\}.
\end{equation}
There exists a finite $C = C(M,\poly)$ such that the following holds. Suppose
$X_1,\dotsc, X_d$ are \emph{independent} random variables with
bounded densities supported in $[-M,M]$. Then
\begin{equation*}
  H(X_1, \dotsc, X_d) = \sum_{i=1}^d H(X_i)\le \log(V_0) + C\cdot\P\big((X_1,\dotsc,X_d)\notin\poly\big)^{1/d}.
\end{equation*}
\end{thm}
Let us comment on the assumptions and the conclusion of the theorem. First,
since we will only use this result for a very specific $\poly$ (the closure of the 3-dimensional metric polytope $\M_3$), we do not need the exact dependence of $C$ on $\poly$, but let us nonetheless note that $C$ depends only on $d$, $M$, and $V_0$.

 The assumption $d \ge 2$ is required for the conclusion. Indeed, when $d=1$, there exist examples where the error term has an additional logarithmic factor (see~\eqref{eq:poly theorem weaker bound} below for a complementary upper bound). To see this, consider $\poly = [0,1] \subseteq [-2,2]$ and let $X_1$ be a random variable that, with probability $1-\eps$, is uniform on $\poly$ and, with probability $\eps$, is uniform on $[1,2]$. Here, $\P(X_1 \notin \poly) = \eps$ whereas $H(X_1) = (1-\eps)\log(1/(1-\eps)) + \eps \log(1/\eps) \ge \log(V_0) + \eps \log(1/\eps)$.

 Aside from the constant factor $C$, the dependence of the error term on $\P\big((X_1, \dotsc, X_d) \notin \poly\big)$ is optimal. To see this, consider the simplex
\[
  \poly := \left\{(x_1, \dotsc, x_d) \in [0,d]^d : x_1 + \dotsb + x_d \le d\right\}\subseteq[-d,d]^d.
\]
The AM--GM inequality implies that $[0,1]^d$ is the largest box contained in $\poly$ and thus $\log(V_0) = 0$. Let $X_1, \dotsc, X_d$ be i.i.d.\ random variables distributed uniformly on the interval $[0,1+\delta]$, for some $\delta \le 1/d$. On the one hand, we have
\[
  \P\big((X_1, \dotsc, X_d) \notin \poly\big) \le \P\big(\min_i |X_i| > 1-(d-1)\delta\big) \le (d \delta)^d.
\]
On the other hand,
\[
  H(X_1, \dotsc, X_d) = d\cdot H(X_1) = d\log(1+\delta) \ge d\delta/2.
\]
Thus, $H(X_1, \dotsc, X_d) \ge \log(V_0) + \P\big((X_1, \dotsc, X_d) \notin \poly\big)^{1/d}/2$.

The first step in our proof of Theorem~\ref{thm:entropy_bound_for_almost_independent} is Lemma~\ref{lem:independent-box-poly}, below. The lemma supplies an axis-parallel box fully contained in $\poly$ that supports most of the distribution of the vector $(X_1, \dotsc, X_d)$. The existence of such a box already implies an upper bound on $H(X_1, \dotsc, X_d)$ that differs from the one stated in Theorem~\ref{thm:entropy_bound_for_almost_independent} by an extra logarithmic factor in the error term (see~\eqref{eq:poly theorem weaker bound}). The proof of the lemma is short; following it, the bulk of the proof of Theorem~\ref{thm:entropy_bound_for_almost_independent} is devoted to removing this superfluous logarithmic term. (If one substitutes the bound~\eqref{eq:poly theorem weaker bound} for Theorem~\ref{thm:entropy_bound_for_almost_independent} in the argument presented in Section~\ref{sec:volume-upper-bound}, one obtains the following weaker version of the upper bound in Theorem~\ref{thm:volume_estimate}: $\vol(\M_n)\le \exp\big(C(n\log n)^{3/2}\big)$.)

The second step in the proof of Theorem~\ref{thm:entropy_bound_for_almost_independent} is Proposition~\ref{prop:entropy_bound_for_almost_independent}, below. The proposition (combined with Lemma~\ref{lem:independent-box-poly}) may be regarded as a strengthening of the conclusion of the theorem. The extra information it provides will be used in our analysis of the minimum distance in a typical sample from the metric polytope (Theorem~\ref{thm:minimal_distance}).

\medskip
We start the proof of Theorem~\ref{thm:entropy_bound_for_almost_independent} with several definitions that we will use throughout.

Let $M>0$ and fix a closed convex set $\poly\subseteq[-M,M]^d$ with non-empty interior. At this point the dimension is allowed to be any $d\ge 1$ but the restriction $d\ge 2$ will be placed in Proposition~\ref{prop:entropy_bound_for_almost_independent}. Write $V_0$ for the maximal volume of an axis-parallel box fully contained in $\poly$, defined formally in~\eqref{eq:maximal volume of box in polytope}. Our assumptions on $\poly$ imply that $V_0>0$. Let $X_1,\dotsc, X_d$ be \emph{independent} random variables with bounded densities supported in $[-M,M]$. Define
  \begin{equation}\label{eq:eps def for poly}
    \eps := \P\big((X_1,\dotsc,X_d)\notin\poly\big)^{1/d},
  \end{equation}
so that our goal is to show that, for a finite $C = C(M,\poly)$,
\begin{equation}\label{eq:entropy_bound_goal}
  H(X_1, \dotsc, X_d) = \sum_{i=1}^d H(X_i)\le \log(V_0) + C\eps.
\end{equation}
We may (and will) assume without loss of generality that $\eps\le \frac{1}{6}$, as otherwise the statement follows by taking $C=6(d \log(2M) -\log(V_0)) \ge 0$ (as $H(X_i) \le \log(2M)$ by Lemma~\ref{lem:entropy-compact-support} and $V_0 \le (2M)^d$).

For each $i \in \br{d}$, define the upper and lower $\eps$-quantiles of the distribution of $X_i$,
\begin{equation}\label{eq:LQ UQ def}
\begin{split}
  a_i & := \sup\{a \in \R : \P(X_i < a) \le \eps\}, \\
  b_i & := \inf\{b \in \R : \P(X_i > b) \le \eps\},
\end{split}
\end{equation}
so that
\begin{equation}
  \label{eq:X-quantiles}
  \P\big(X_i \le a_i\big) = \P\big(X_i \ge b_i\big) = \eps.
\end{equation}
In particular, the interval $[a_i,b_i]$ is nonempty by our assumption that $\eps\le \frac{1}{6}$. Denote the volume spanned by these intervals by
\begin{equation}\label{eq:V def}
  V := \vol\big([a_1,b_1] \times \dotsb \times [a_d,b_d]\big) = \prod_{i=1}^d (b_i-a_i).
\end{equation}
Finally, writing $f_i$ for the density of $X_i$, let
\begin{equation}\label{eq:entropy portion}
  H(X_i;A) := -\int_A f_i(x)\log(f_i(x))dx
\end{equation}
be the contribution to the differential entropy of $X_i$ from the measurable set $A$.

  Our first lemma shows that the box spanned by the intervals $([a_i, b_i])$ is fully contained in $\poly$.
\begin{lem}
  \label{lem:independent-box-poly}
  In every dimension $d\ge 1$,
  \begin{equation}\label{eq:box in polytope}
    [a_1,b_1] \times \dotsb \times [a_d, b_d] \subseteq \poly.
  \end{equation}
  In particular, $V\le V_0$.
\end{lem}
\begin{proof}
  Suppose, to obtain a contradiction, that
  \eqref{eq:box in polytope} fails. Hence, as $\poly$ is closed, there exists $(x_1,
  \dotsc, x_d)\notin \poly$ with $a_i<x_i<b_i$ for all $i$. This implies, as $\poly$ is convex, that there is a choice of signs $(s_1,\ldots, s_d)\in\{-1,1\}^d$ so that the orthant
  \begin{equation*}
    O := \big\{(y_1,\dotsc, y_d) : s_i(y_i - x_i)\ge 0 \text{ for all $i \in \br{d}$}\big\}
  \end{equation*}
  does not intersect $\poly$. In particular,
  \begin{align*}
    \eps^d & = \P\big((X_1,\dotsc,X_d)\notin\poly\big)\ge \P\big((X_1,\dotsc,X_d)\in
    O\big)\\
    &\ge \prod_{i=1}^d \min\{\P(X_i \le x_i),\P(X_i\ge x_i)\}\\
    &>\prod_{i=1}^d \min\{\P(X_i \le a_i),\P(X_i\ge b_i)\} =
    \eps^d,
  \end{align*}
  where the strict inequality uses that $a_i<x_i<b_i$ and the definition~\eqref{eq:LQ UQ def}.
  This contradiction shows that~\eqref{eq:box in polytope} must in fact hold.

  The volume statement is now deduced from the definition of $V_0$.
\end{proof}
We digress from the proof of Theorem~\ref{thm:entropy_bound_for_almost_independent} to note that Lemma~\ref{lem:independent-box-poly} implies the following version of the theorem, which holds in every dimension $d\ge 1$ but has an extra logarithmic factor in the error term,
\begin{equation}\label{eq:poly theorem weaker bound}
  \sum_{i=1}^d H(X_i) \le \log(V_0) + C\eps\log(2/\eps)
\end{equation}
with $\eps$ as in~\eqref{eq:eps def for poly} and $C = C(d,M,V_0)$ finite. The case $\eps>\frac{1}{6}$ is handled directly as before. For $\eps\le \frac{1}{6}$, note first that, for some $C' = C'(d, M)$,
  \begin{equation}\label{eq:weaker upper bound on entropy}
    \sum_{i=1}^d H(X_i) \le (1-2\eps)\log(V) + C'\eps\log(2/\eps)
  \end{equation}
  Indeed, for each $i \in \br{d}$, by~\eqref{eq:X-quantiles} and Lemma~\ref{lem:entropy_bound_for_sub_probability},
  \begin{equation*}
  \begin{split}
    H(X_i) & = H(X_i;[-M,M]\setminus[a_i,b_i]) + H(X_i;[a_i,b_i])\\
    & \le 2\eps \log\left(\frac{2M-(b_i-a_i)}{2\eps}\right) +
    (1-2\eps)\log\left(\frac{b_i-a_i}{1-2\eps}\right)\\
    & \le (1-2\eps)\log(b_i-a_i) + 2\eps\log\left(\frac{M}{\eps}\right) +
      (1-2\eps)\log\left(\frac{1}{1-2\eps}\right)
  \end{split}
  \end{equation*}
  and the bound~\eqref{eq:weaker upper bound on entropy} follows by summing this estimate over all $i$. To deduce~\eqref{eq:poly theorem weaker bound}, replace $V$ by $V_0$ using Lemma~\ref{lem:independent-box-poly} and absorb the factor $2\eps\log(V_0)$ in the error term.

We return to the proof of Theorem~\ref{thm:entropy_bound_for_almost_independent} and will show the following key proposition.

\begin{prop}\label{prop:entropy_bound_for_almost_independent}
In dimensions $d\ge 2$, there exists a finite $C = C(M,\poly)$ such that
\begin{equation*}
  \sum_{i=1}^d H(X_i) \le \frac{1}{2}\left(\log(V_0) + \sum_{i=1}^d H(X_i; [a_i, b_i])\right) + C\eps.
\end{equation*}
\end{prop}

It will be convenient to denote by $c$ and $C$ finite positive constants which depend only on $d$, $M$, and $V_0$. These constants, and their numbered versions, may change from line to line.

Let us see how Proposition~\ref{prop:entropy_bound_for_almost_independent} and Lemma~\ref{lem:independent-box-poly} imply Theorem~\ref{thm:entropy_bound_for_almost_independent}. Combining the proposition with Lemma~\ref{lem:entropy_bound_for_sub_probability}, recalling~\eqref{eq:X-quantiles},~\eqref{eq:V def}, and our assumption that $\eps\le\frac{1}{6}$, and applying Lemma~\ref{lem:independent-box-poly} we have
\begin{equation*}
\begin{split}
  H(X_1, \dotsc, X_d) &\le \frac{1}{2}\left(\log(V_0) + (1-2\eps)\sum_{i=1}^d \log\left(\frac{b_i-a_i}{1-2\eps}\right)\right) + C\eps\\
  &\le \frac{1}{2}\big(\log(V_0) + (1-2\eps) \log(V)\big) + C\eps\le \log(V_0) + C\eps.
\end{split}
\end{equation*}

\begin{proof}[Proof of Proposition~\ref{prop:entropy_bound_for_almost_independent}]
  There are three constants in the proof that deserve their own letters, $\beta$, $\mu$, and $K$, also depending only on $d$, $M$, and $V_0$. We will not specify these explicitly, and only point out here that we first choose $\beta$ (small), then $\mu$ (even smaller), and then $K$ (very large). In symbols (treating $d$, $M$, and $V_0$ as constants),
\[
  K^{-1} \ll \mu \ll \beta \ll 1.
\]

Let us introduce the following quantity,
\[
  t:=\sup\left\{s > 0:\exists i \; \max\left\{\P(X_i\le a_i-s), \P(X_i\ge b_i+s)\right\} \ge \frac{K\eps^2}{s}\right\}
\]
where we set $t = 0$ if the above set is empty. As each $X_i$ is supported on $[-M,M]$, we have $t \le 2M$. On the other hand, by the definition of $a_i$ and $b_i$, see~\eqref{eq:X-quantiles}, either $t = 0$ or $t \ge K\eps$.

The proposition is obtained by summing the following inequalities:
\begin{multline}
  \label{eq:X_i entropy on sides}
  M_t := \sum_{i=1}^d H(X_i;[a_i-t, a_i]\cup[b_i, b_i+t]) + \frac{1}{2}H(X_i;[a_i,b_i]) \\
  \le \frac{1}{2}\log(V_0) + C\eps.
\end{multline}
and, for each $i\in\br{d}$,
\begin{equation}\label{eq:X_i entropy on tail}
  H(X_i; \R\setminus[a_i - t, b_i + t])\le C\eps.
\end{equation}
We first prove~\eqref{eq:X_i entropy on tail}. For each $k \ge 1$, let $\lambda_k$ be the Lebesgue measure of the set
  \[
    \{x \notin [a_i-t,b_i+t] : e^{-k} \le f_i(x) \le e^{-k+1}\}
  \]
  and note that, as $-y\log y<0$ for $y>1$,
  \[
    H(X_i; \R\setminus[a_i - t, b_i + t]) \le \sum_{k=1}^\infty ke^{-k+1} \lambda_k.
  \]
  For every $k \ge 1$,
  \[
    \begin{split}
      \lambda_k & \le 2\eps k^2 + e^k \cdot \left(\P(X_i \le a_i - t - \eps k^2) + \P(X_i \ge b_i+t+\eps k^2)\right) \\
      & \stackrel{(*)}{\le} 2 \eps k^2 + e^k \cdot \frac{2K\eps^2}{t + \eps k^2} \le \left(2k^2 + \frac{2Ke^k}{k^2}\right) \cdot \eps
    \end{split}
  \]
  where $(*)$ follows from the definition of $t$. Hence
  \[
    \sum_{k=1}^\infty ke^{-k+1}\lambda_k \le \left(\sum_{k=1}^\infty \frac{2k^3}{e^{k-1}} + \sum_{k=1}^\infty \frac{2eK}{k^2}\right) \cdot \eps \le C\eps,
  \]
  which proves~\eqref{eq:X_i entropy on tail}.

  It remains to argue that~\eqref{eq:X_i entropy on sides} holds as well. For this we apply Lemma~\ref{lem:entropy_bound_for_sub_probability} and get for each $i$ (using also $\log y\le y-1$ for $y>0$),
  \begin{equation}
    \label{eq:I-ai-bi}
    \begin{split}
      H(X_i;[a_i,b_i]) & \le (1-2\eps)\log\left(\frac{b_i-a_i}{1-2\eps}\right) \\
      & \le (1-2\eps)\log(b_i-a_i) + 2\eps
    \end{split}
  \end{equation}
  and, if $t > 0$,
  \begin{equation}
    \label{eq:I-ai-t-bi-t}
    H(X_i;[a_i-t,a_i]\cup[b_i, b_i+t]) \le 2\eps\log\left(\frac{t}{\eps}\right),
  \end{equation}
  where we used the fact that $t \ge K\eps \ge e\eps$, which implies that the function $\delta \mapsto \delta \log(t/\delta)$ is increasing for $\delta \in [0,\eps]$. We split the remainder of the argument into two cases, depending on how close $V$, the volume of $[a_1, b_1] \times \dotsb \times [a_d, b_d]$, is to $V_0$.

  \subsubsection*{Case 1.}
  We first assume that
  \begin{equation}
    \label{eq:cube-small}
    \log(V) \le \log(V_0) - \beta t.
  \end{equation}
  In this case, summing~\eqref{eq:I-ai-t-bi-t} and half of~\eqref{eq:I-ai-bi} over all $i$ gives
  \begin{align*}
    M_t & \le \frac{1}{2}(1-2\eps) \sum_{i=1}^d \log(b_i-a_i) + d\eps + 2d\eps\log\left(\frac{t}{\eps}\right) \cdot \1_{\{t > 0\}}\\
        & \le \frac{1}{2}(1-2\eps)\left(\log(V_0) - \beta t\right) + d\eps + 2d\eps\log\left(\frac{t}{\eps}\right) \cdot \1_{\{t > 0\}}\\
        & \le \frac{1}{2}\log(V_0) + \left(d + 2d\log\left(\frac{t}{\eps}\right) \cdot \1_{\{t > 0\}} - \frac{\beta}{4} \cdot \frac{t}{\eps} - \log(V_0)\right) \cdot \eps.
  \end{align*}
  The claimed estimate~\eqref{eq:X_i entropy on sides} now follows as, for every $c > 0$, the function $y \mapsto \log(y) - cy$ is bounded from above.

  \subsubsection*{Case 2.}
  Assume now that~\eqref{eq:cube-small} does not hold. This means, in particular, that $t > 0$ (due to Lemma~\ref{lem:independent-box-poly}). Since our variables are continuous,
  \[
    \max\left\{\P(X_i\le a_i-t), \P(X_i\ge b_i+t)\right\} = \frac{K\eps^2}{t}
  \]
  for some index $i$. By permuting and reflecting the coordinates, if necessary, we may assume that
  \begin{equation}\label{eq:def t}
    \P(X_1\le a_1-t)=K\eps^2/t.
  \end{equation}
  We claim that
  \begin{equation}
    \label{eq:box-modified-not-in-poly}
    [a_1- t,b_1]\times\prod_{i=2}^d[a_i+\mu t,b_i-\mu t]\nsubseteq\poly.
  \end{equation}
  Indeed, if this were not true, then
  \begin{equation}
    \label{eq:case-2-vol-V-modified}
    \log(V_0) - \log(V) \ge \log\left(\frac{b_1-a_1+t}{b_1-a_1}\right) +\sum_{i=2}^d\log\left(\frac{b_i-a_i-2\mu t}{b_i-a_i}\right).
  \end{equation}
  Since~\eqref{eq:cube-small} does not hold, we have
  \[
    \min_i (b_i-a_i) \ge \frac{V}{(2M)^{d-1}} \ge \frac{e^{-\beta t} \cdot V_0}{(2M)^{d-1}} \ge \frac{V_0}{(2M)^d},
  \]
  where the last inequality holds as $t \le 2M$ and $\beta$ is small. It follows that the first term in the right-hand side of~\eqref{eq:case-2-vol-V-modified} is at least $c_1t$, for some positive constant $c_1=c_1(d, M, V_0)$ (independent of $\beta$ as long as $\beta$ is small), and, if $\mu$ is sufficiently small, each of the $d-1$ summands is at least $-C_1\mu t$, for some positive constant $C_1 = C_1(d, M, V_0)$. In particular, if $\beta$ and $\mu$ are sufficiently small, then the right-hand side of~\eqref{eq:case-2-vol-V-modified} is larger than $\beta t$, contradicting our assumption.

  Let $x$ be some point demonstrating~\eqref{eq:box-modified-not-in-poly}, namely
  \[
    x\in \left([a_1- t,b_1]\times\prod_{i=2}^d[a_i+\mu t,b_i-\mu t]\right)\setminus\poly,
  \]
  and notice that $x_1\in[a_1-t,a_1)$, as $\prod_i [a_i,b_i] \subseteq \poly$ by Lemma~\ref{lem:independent-box-poly}. Since $\poly$ is convex there is a hyperplane separating $x$ from $\poly$, that is, a vector $v$ such that
  \begin{equation}\label{eq:separating}
    \forall y \in \poly \quad \langle v, x \rangle < \langle v, y \rangle.
  \end{equation}

  First, let us apply \eqref{eq:separating} to $y=(a_1,x_2,\dotsc,x_d)$, which we may since $x_i\in[a_i,b_i]$ for all $i\ge 2$. We get $v_1(a_1-x_1)> 0$, so $v_1>0$ and we may normalise $v$ to assume $v_1=1$. Moreover, by permuting the coordinates, if necessary, we may assume that $v_i \ge 0$ for $i \in \{2, \dotsc, j\}$ and $v_i < 0$ for $i \in \{j+1, \dotsc, d\}$. We may also assume that $j \ge 2$, the complementary case $j+1 = 2$ being essentially identical. (Note that we do use the assumption that $d \ge 2$ here.) These assumptions imply that
  \[
    (-\infty, a_1-t] \times (-\infty,a_2 + \mu t] \times \prod_{i = 3}^j (-\infty, a_i] \times \prod_{i = j+1}^d [b_i, \infty)
  \]
  is disjoint from $\poly$. Indeed, if $z$ is an arbitrary point in this orthant, we have $v_iz_i \le v_ix_i$ for each $i$ and hence $\langle v, z \rangle < \langle v, y \rangle$ for all $y\in\poly$. It follows that
  \begin{multline*}
    \P(X_1 \le a_1-t) \cdot \P(a_2 \le X_2 \le a_2 + \mu t) \\
    \cdot \prod_{i=3}^j \P(X_i \le a_i) \cdot \prod_{i=j+1}^d \P(X_i \ge b_i) \le \eps^d,
  \end{multline*}
  and therefore, as $\P(X_i \le a_i) = \P(X_i \ge b_i) = \eps$,
  \[
    \P(a_2 \le X_2 \le a_2+\mu t) \le  \frac{\eps^2}{\P(X_1 \le a_1-t)}
    \stackrel{\textrm{(\ref{eq:def t})}}{=}
    \frac{t}{K}.
  \]
  Thus we have shown some inhomogeneity in the distribution of $X_2$. If we require $K>10(b_2-a_2)/\mu(1-2\eps)$, then this inhomogeneity is strong enough to apply \eqref{eq:entropy-bound-non-uniform-sub-probability} in Lemma~\ref{lem:entropy_bound_for_sub_probability},
so let us make this requirement. We get
  \begin{align}
      H(X_2;[a_2,b_2]) & \le (1-2\eps)\log\left(\frac{b_2-a_2}{1-2\eps}\right) - \frac{1-2\eps}{4} \cdot \frac{\mu t}{b_2-a_2} \nonumber\\
      & \le (1-2\eps)\log(b_2-a_2) + 2\eps - \frac{\mu}{20M} \cdot t.    \label{eq:int-a2-b2}
  \end{align}
  Summing~\eqref{eq:I-ai-t-bi-t} and half of~\eqref{eq:I-ai-bi} over all $i$, as in Case 1, but using the improved estimate~\eqref{eq:int-a2-b2} in place of~\eqref{eq:I-ai-bi} when $i=2$ yields
  \[
    M_t  \le \frac{1}{2}(1-2\eps) \log(V) + d\eps + 2d\eps\log\left(\frac{t}{\eps}\right) - \frac{\mu}{40M} \cdot t.
  \]
  A calculation similar to the one done in Case 1, using that $V\le V_0$ by Lemma~\ref{lem:independent-box-poly}, yields the upper bound~\eqref{eq:X_i entropy on sides} on $M_t$. This completes the proof of the proposition (and hence also of Theorem~\ref{thm:entropy_bound_for_almost_independent}).
\end{proof}

\ifx\dontshowuselessoldcrap\undefined

  Now, \eqref{eq:entropy-on-ap-bp} follows by summing~\eqref{eq:int-a-b}, \eqref{eq:int-aip-ai}, and \eqref{eq:int-bi-bip} over all $i$, but using the stronger estimate~\eqref{eq:int-a2-b2} in place of~\eqref{eq:int-a-b} when $i=2$.

  ***************
  In order to remove the superfluous logarithmic term, we shall assume now that $d \ge 2$ and that $\poly$ is a polytope.
  In particular, there is a \emph{finite} set $V \subseteq \left(\R^d \setminus \{0\}\right) \times \R$ such that
  \[
  \poly = \bigcap_{(v,\tau) \in V} \left\{x \in \R^d : \sum_{i=1}^d v_ix_i \ge \tau\right\}.
  \]
  Let $\rho$ be the smallest (in absolute value) nonzero ratio of the coordinates of a normal vector defining $\poly$. In other words,
  \[
  \rho := \min\left\{|v_i| / |v_j| : \big((v_1, \dotsc, v_d), \tau\big) \in V \text{ and $i, j \in [d]$ satisfy $v_i, v_j \neq 0$}\right\} \le 1.
  \]
  Moreover, let $K = K(\rho,d,M)$ be sufficiently large and define for each $1 \le i \le d$,
  \begin{equation}\label{eq:ap_i_bp_i_def}
    \begin{split}
      a_i' &:= \inf \left\{a \in [-M, a_i) : \P(X_i \le a) > K\eps^2 / (a_i-a) \right\},\\
      b_i' &:= \sup \left\{b \in (b_i,M] : \P(X_i \ge b) > K\eps^2 / (b-b_i) \right\},
    \end{split}
  \end{equation}
  where $a_i' = a_i$ (resp.\ $b_i' = b_i$) if the set in the infimum (resp.\ supremum) is empty.

  Since for every $1 \le i \le d$, the definition of $a_i'$ and $b_i'$ yields $\P(X_i \le a) \le K\eps^2 / (a_i-a)$
  for all $a \le a_i'$ and $\P(X_i \ge b) \le K\eps^2 / (b-b_i)$ for all $b \ge b_i'$,
  Lemma~\ref{lem:entropy_bound_under_tail_estimate} implies that
  \[
  -\int_{-M}^{a_i'} f_i(x)\log(f_i(x)) dx - \int_{b_i'}^M f_i(x)\log(f_i(x))dx \le 2C\sqrt{K}\eps.
  \]
  Therefore, we only need to show that
  \begin{equation}
    \label{eq:entropy-on-ap-bp}
    - \sum_{i=1}^d \int_{a_i'}^{b_i'} f_i(x)\log(f_i(x))dx \le \log(V_0) + C\eps.
  \end{equation}
  As Lemma~\ref{lem:entropy_bound_for_sub_probability} yields
  \begin{equation}
    \label{eq:int-a-b}
    \begin{split}
      -\int_{a_i}^{b_i} f_i(x)\log(f_i(x))dx & \le \P(X_i\in[a_i,b_i])\log\left(\frac{\vol([a_i,b_i])}{\P(X_i\in[a_i,b_i])}\right) \\
      & \le (1-2\eps)\log(b_i-a_i) - \log\left(1-2\eps\right)
    \end{split}
  \end{equation}
  and $\sum_{i=1}^d \log(b_i - a_i) \le \log(V_0)$ by~\eqref{eq:box_in_poly}, we shall be focusing on estimating the integrals of $f_i\log(f_i)$ on the sets $[a_i',a_i] \cup [b_i,b_i']$.

  To this end, let us define
  \[
  t := \max\left\{\max\{b_i'-b_i, a_i-a_i'\} : i \in [d]\right\}.
  \]
  In the easy case $t = 0$, we have $a_i' = a_i$ and $b_i' = b_i$ and therefore we easily obtain~\eqref{eq:entropy-on-ap-bp}
  by summing~\eqref{eq:int-a-b} over all $i$, as $\log(1-2\eps) \ge -3\eps$ by our assumption that $\eps \le 1/6$.

  Hence, from now one we shall assume that $t > 0$. Without loss of generality, we may further assume that $t = a_1 - a_1'$,
  as otherwise we can permute the coordinates (which results in the same permutation of the indices of $a_i$, $b_i$, $a_i'$, and $b_i'$)
  or reflect them in zero (which results in exchanging the roles of $a_i$ and $a_i'$ with $b_i$ and $b_i'$).
  Note that
  \[
  \frac{K\eps^2}{t} = \frac{K\eps^2}{a_1-a_1'} \le \P(X_1 \le a_1') \le \P(X_1 \le a_1) = \eps,
  \]
  and hence $t \ge K\eps$. In particular, Lemma~\ref{lem:entropy_bound_for_sub_probability} implies that for every $1 \le i \le d$,
  \begin{equation}
    \label{eq:int-aip-ai}
    \begin{split}
      -\int_{a_i'}^{a_i} f_i(x) \log(f_i(x))dx & \le \P\left(a_i' \le X \le a_i\right) \log\left(\frac{a_i-a_i'}{\P\left(a_i' \le X_i \le a_i\right)}\right)  \\
      & \le \eps \log\left(\frac{t}{\eps}\right) \le \frac{t}{K} \log K \le \frac{\rho t}{50dM}
    \end{split}
  \end{equation}
  and similarly,
  \begin{equation}
    \label{eq:int-bi-bip}
    - \int_{b_i}^{b_i'} f_i(x)\log(f_i(x))dx \le \frac{\rho t}{50dM}.
  \end{equation}
  Let
  \[
  s := \max\left\{ s \ge 0 : [a_1-s,b_1] \times [a_2, b_2] \times \dotsb \times [a_d,b_d] \subseteq \poly\right\}.
  \]
  We shall split the remainder of the proof into two cases, depending on whether $s \ge t/2$.

  \medskip
  \noindent
  \textbf{Case 1. $s \ge t/2$.}
  \smallskip

  Since $(b_1-a_1+s) \cdot \prod_{i=2}^d (b_i-a_i) \le V_0$, then
  \begin{equation}
    \label{eq:sum-log-ai-bi}
    \begin{split}
      \sum_{i=1}^d \log(b_i-a_i) & \le \log(V_0) - \log\left(\frac{b_1-a_1+s}{b_1-a_1}\right) \le \log(V_0) - \log\left(1+\frac{s}{2M}\right) \\
      & \le \log(V_0) - \log\left(1+\frac{t}{4M}\right) \le \log(V_0) - \frac{t}{6M},
    \end{split}
  \end{equation}
  where the last inequality follows as clearly $t \le 2M$. Consequently, \eqref{eq:entropy-on-ap-bp} follows by summing~\eqref{eq:int-a-b}, \eqref{eq:int-aip-ai}, and \eqref{eq:int-bi-bip} over all $i$ and then invoking~\eqref{eq:sum-log-ai-bi}.

  \medskip
  \noindent
  \textbf{Case 2. $s < t/2$.}
  \smallskip

  As $[a_1-t/2,b_1] \times [a_2,b_2] \times \dotsb \times [a_d, b_d] \not\subseteq \poly$ by the definition of $s$, there must exist a $(v,\tau) \in V$ and $(x_1, \dotsc, x_d) \in \R^d$ such that $a_1 - t/2 \le x_1 \le b_1$ and $a_i \le x_i \le b_i$ for $i \in \{2, \dotsc, d\}$ such that
  \begin{equation}
    \label{eq:v-tau-violated}
    \sum_{i=1}^d v_ix_i < \tau.
  \end{equation}
  As $(a_1, x_2, \dotsc, x_d) \in \poly$ by~\eqref{eq:box_in_poly}, then necessarily $v_1 > 0$. Without loss of generality, we may assume that $v_i \ge 0$ for $i \in \{2, \dotsc, j\}$ and $v_i < 0$ for $i \in \{j+1, \dotsc, d\}$. We may also assume that $j \ge 2$, the complementary case $j+1 = 2$ being essentially identical. (Note that we do use the assumption that $d \ge 2$ here.) We claim that
  \begin{equation}
    \label{eq:modified-tail-outside-poly}
    (-\infty, a_1'] \times (-\infty,a_2 + \rho t/2] \times \prod_{i = 3}^j (-\infty, a_i] \times \prod_{i = j+1}^d [b_i, \infty) \subseteq \R^d \setminus \poly.
  \end{equation}
  Indeed, note first that by our definition of $j$, inequality~\eqref{eq:v-tau-violated} still holds if we decrease $x_i$ for $i \in \{1, \dotsc, j\}$ and increase $x_i$ for $i \in \{j+1, \dotsc, d\}$. Moreover, as $a_1' = a_1 - t$ and $v_1 \ge \rho v_2$ by the definition of $\rho$, then
  \[
  v_1 a_1' + v_2(a_2+\rho t/2) + \sum_{i=3}^d v_ix_i = v_1(a_1-t/2) + v_2a_2 + \sum_{i=3}^d v_ix_i +  (\rho v_2 - v_1) t/2 < \tau.
  \]
  It follows from~\eqref{eq:modified-tail-outside-poly} that
  \[
  \P(X_1 \le a_1') \cdot \P(a_2 \le X_2 \le a_2 + \rho t/2) \cdot \prod_{i=3}^j \P(X_i \le a_i) \cdot \prod_{i=j+1}^d \P(X_i \ge b_i) \le \eps^d,
  \]
  and therefore
  \[
  \P(a_2 \le X_2 \le a_2+\rho t/2) \le \frac{\eps^2}{\P(X_1 \le a_1')} \le \frac{a_1-a_1'}{K} = \frac{t}{K} \le \frac{1}{10} \cdot \frac{\rho t}{4M} \le \frac{1}{10} \cdot \frac{\rho t/2}{b_2-a_2}.
  \]
  Therefore, Lemma~\ref{lem:entropy_bound_for_non_uniform_sub_probability} implies the following improvement of the estimate~\eqref{eq:int-a-b}
  \begin{equation}
    \label{eq:int-a2-b2}
    \begin{split}
      -\int_{a_2}^{b_2} f_2(x)\log(f_2(x)) dx & \le (1-2\eps)\log(b_2-a_2) - \log(1-2\eps) - \frac{1-2\eps}{4} \cdot \frac{\rho t /2}{b_2-a_2} \\
      & \le (1-2\eps)\log(b_2-a_2) - \log(1-2\eps) - \frac{\rho t}{24M}.
    \end{split}
  \end{equation}
  Now, \eqref{eq:entropy-on-ap-bp} follows by summing~\eqref{eq:int-a-b}, \eqref{eq:int-aip-ai}, and \eqref{eq:int-bi-bip} over all $i$, but using the stronger estimate~\eqref{eq:int-a2-b2} in place of~\eqref{eq:int-a-b} when $i=2$.
\end{proof}
\fi

\subsection{The largest box in \texorpdfstring{$\M_3$}{M3}}

Theorem~\ref{thm:entropy_bound_for_almost_independent} will be applied to the (closure of the) 3-dimensional metric polytope. To this end, we study here the largest box contained in $\M_3$.

\begin{lem}\label{lem:independent_max_volume}
  Suppose that $P$ is an axis-parallel box contained in the closure of the metric polytope $\M_3$, that is,
  \[
  P = [a_1,b_1] \times [a_2,b_2] \times [a_3,b_3] \subseteq \overline{\M_3} \subseteq \R^{\binom{3}{2}} \cong \R^3.
  \]
  Then
  \begin{equation}\label{eq:vol_P_lemma_estimate}
  \vol(P) \le 1
  \end{equation}
   and equality holds if and only if
  $[a_i,b_i]=[1,2]$ for each $i \in \{1,2,3\}$.
Furthermore, for some absolute constant $C$,
\[
\sum_{i=1}^3 \big( |a_i - 1| + |b_i - 2| \big) \le C(1-\vol(P)).\]
\end{lem}
The furthermore clause is not required for our analysis of the volume of the metric polytope (Theorem~\ref{thm:volume_estimate}) but will be used in analysing the typical minimum distance (Theorem~\ref{thm:minimal_distance}).

The following proof was suggested to us by Shoni Gilboa; it replaced our previous, less transparent argument.

\begin{proof}[Proof of Lemma~\ref{lem:independent_max_volume}]
  Our assumption that $P \subseteq \overline{\M_3}$ implies that $b_1 \le a_2 + a_3$ and, similarly, $b_2 \le a_1 + a_3$ and $b_3 \le a_1 + a_2$. Summing these three inequalities yields
  \begin{equation}
    \label{eq:three-triangles}
    b_1+b_2+b_3 \le 2(a_1+a_2+a_3).
  \end{equation}
  Consequently, by the AM--GM inequality,
  \begin{equation}
    \label{eq:AMGM}
    \vol(P) = \prod_{i=1}^3 (b_i-a_i) \le \left(\sum_{i=1}^3\frac{b_i-a_i}{3}\right)^3 \le \left(\frac{b_1+b_2+b_3}{6}\right)^3 \le 1,
  \end{equation}
  where the second inequality is precisely~\eqref{eq:three-triangles} and the last inequality holds by our assumption that $P \subseteq \overline{\M_3} \subseteq [0,2]^3$, which ensures that $b_1, b_2, b_3 \le 2$.

  For the second and third assertions of the lemma, suppose that $\vol(P) \ge 1-\eps$ for some $\eps \ge 0$. Inequality~\eqref{eq:AMGM}
  implies that
  \begin{equation}
    \label{eq:AMGM-stability}
    \frac{b_1+b_2+b_3}{2} \ge \sum_{i=1}^3 (b_i-a_i) \ge 3(1-\eps)^{1/3} \ge 3(1-\eps)
  \end{equation}
  and, consequently, as $\max_i b_i \le 2$, that $\min_i b_i \ge 2 - 6\eps$. Moreover, summing the inequalities $a_1+a_2 \ge b_3$ and $a_1 + a_3 \ge b_2$, we obtain
  \[
    a_1 \ge b_2+b_3 - (a_1+a_2+a_3) = \sum_{i=1}^3 (b_i-a_i) - b_1 \ge 1 - 3\eps,
  \]
  where the last inequality follows from~\eqref{eq:AMGM-stability} and the assumption $b_1 \le 2$. By symmetry, $\min_i a_i \ge 1-3\eps$.
 Using again~\eqref{eq:AMGM-stability}, we have
  \[
    3(1-\eps) \le \sum_{i=1}^3 (b_i-a_i) \le 6 - 2\min_i a_i -a_1,
  \]
  giving $a_1 \le 1+9\eps$. By symmetry, $\max_i a_i \le 1+9\eps$. Hence
  \[
  \sum_{i=1}^3\big(|a_i-1|+|b_i-2|\big)\le\sum_{i=1}^3\big( 9\eps+6\eps \big)=45\eps,
  \]
  as needed.
\end{proof}

\section{Estimating the volume of the metric polytope}

In this section, we prove Proposition~\ref{prop:decreasing_radius} and Theorem~\ref{thm:volume_estimate}. The proof of Proposition~\ref{prop:decreasing_radius} is a short application of Shearer's inequality (Theorem~\ref{thm:Shearers_inequality}) and is presented in Section~\ref{sec:monotonicity}. The lower bound on the volume of $\M_n$ is obtained via the Local Lemma of Erd\H{o}s and Lov\'asz~\cite{ErLo75}; it is presented in Section~\ref{sec:lower}. The proof of the upper bound on the volume, whose outline is given in the introduction, is presented in Section~\ref{sec:volume-upper-bound}.

\subsection{Monotonicity}
\label{sec:monotonicity}

In this section, we prove Proposition~\ref{prop:decreasing_radius}. Recall
that said proposition states that the sequence $n \mapsto \vol(\M_n)^{1/\binom{n}{2}}$ is non-increasing.
We will show that this fact is a simple consequence of Shearer's inequality
(Theorem~\ref{thm:Shearers_inequality}). Alternatively, one can derive
it from a generalisation of the Loomis--Whitney inequality~\cite{LoWh49} due to
Bollob\'as and Thomason~\cite{BoTh95}.

\begin{proof}[{Proof of Proposition~\ref{prop:decreasing_radius}}]
  Let $n\ge 2$ and let $d$ be a uniformly sampled metric space in $\M_{n+1}$.
  By Lemma~\ref{lem:entropy-compact-support},
  \begin{equation}\label{eq:entropy_as_volume_of_M_n+1}
    H(d) = \log\left(\vol(\M_{n+1})\right).
  \end{equation}
  For each $i \in \br{n+1}$, let $J_i := \br{n+1}\setminus\{i\}$ and let
  $I_i$ be the set of all unordered pairs of distinct elements in
  $J_i$. Observe that, for each $\{j,k\}\in\binom{\br{n+1}}{2}$, we have
  \[
    |\{i \in \br{n+1} : \{j,k\}\in I_i\}| = n-1.
  \]
  Since $d$ may be naturally viewed as a random vector of distances, Shearer's inequality
  (Theorem~\ref{thm:Shearers_inequality}) implies that
  \begin{equation}\label{eq:Shearers_inequality_for_metric_space}
    H(d)\le \frac{1}{n-1}\sum_{i=1}^{n+1}H\big(\{\dist{j}{k}: \{j,k\}\in I_i\}\big).
  \end{equation}
  Finally, observe that for each $i \in \br{n+1}$, the restriction of
  $d$ to the pairs in $I_i$ is a metric space on
  $n$ points that belongs to $\M_n$. Thus, by
  Lemma~\ref{lem:entropy-compact-support},
  \begin{equation}\label{eq:entropy_at_most_volume_for_M_n}
    H\big(\{\dist{j}{k}: \{j,k\}\in I_i\}\big) \le \log(\vol(\M_n))
  \end{equation}
  Putting together \eqref{eq:entropy_as_volume_of_M_n+1},
  \eqref{eq:Shearers_inequality_for_metric_space}, and
  \eqref{eq:entropy_at_most_volume_for_M_n}, we conclude that
  \[
    \log(\vol(\M_{n+1})) \le \frac{n+1}{n-1}\log(\vol(\M_n)).
  \]
  Since this inequality holds for any $n \ge 2$ and $(n+1)/(n-1) = \binom{n+1}{2} / \binom{n}{2}$, we conclude that the
  sequence $n \mapsto \vol(\M_n)^{1/\binom{n}{2}}$ is non-increasing, as claimed.
\end{proof}

\subsection{Lower bound}
\label{sec:lower}

In this section, we prove the lower bound on $\vol(\M_n)$ from Theorem~\ref{thm:volume_estimate}. Recall that it was
$\vol(\M_n)\ge \exp((\nicefrac{1}{6}+o(1))n^{3/2})$. The proof below, which uses the Local Lemma of Erd\H{o}s and Lov\'asz~\cite{ErLo75}, is due to Dor Elboim (our original argument, based on Harris's inequality~\cite{harris1960lower}, gave $\nicefrac{1}{24}$ instead of $\nicefrac{1}{6}$).

We may and will assume that $n$ is sufficiently large. Let $\delta = 1/(2\sqrt{n})$ and let $(\dist{i}{j})$, $\{i,j\}\in\binom{\br{n}}{2}$, be an array
of independent and identically distributed random variables, each uniform on the interval $\left[1 - \delta, 2\right]$.
Define the event
\[
G:=\{(\dist{i}{j})\in\M_n\}.
\]
Observe that, by definition,
\begin{equation}\label{eq:vol_M_n_G_bound}
  \vol(\M_n) \ge \left(1 + \delta \right)^{\binom{n}{2}}\P(G) = \left(1 + \frac{1}{2\sqrt{n}} \right)^{\binom{n}{2}}\P(G).
\end{equation}
We shall derive a lower bound on $\P(G)$ from the Local Lemma.

\begin{lem}[{\cite[Lemma~5.1.1]{AlSp}}]
  \label{lem:local-lemma}
  Let $B_v$, $v \in V$, be events in an arbitrary probability space. Suppose that there is an integer $k$ such that, for each $v \in V$, there is a set $D_v \subseteq V \setminus \{v\}$ with at most $k$ elements such that $B_v$ is mutually independent of all the events $B_w$ with $w \in V \setminus (D_v \cup \{v\})$. If a real $p \in [0,1]$ satisfies $\P(B_v) \le p (1-p)^k$ for each $v \in V$, then
  \[
    \P\bigg(\bigcap_{v \in V} B_v^c\bigg) \ge (1-p)^{|V|}.
  \]
\end{lem}

For a triple of distinct indices $\{i,j,k\} \in \binom{\br{n}}{3}$, let $B_{ijk}$ denote the event that $(\dist{i}{j}, \dist{i}{k}, \dist{j}{k})$ is
not in $\M_3$, that is, one of the three triangle inequalities is violated. Observe that
\[
\begin{split}
  \P(B_{ijk}) & = 3 \P(d_{ij} + d_{jk} < d_{ik}) = \frac{3}{1+\delta} \int_0^{2\delta} \P(d_{ij} + d_{jk} < 2 - x) \, dx \\
  & = \frac{3}{(1+\delta)^3} \int_0^{2\delta} \frac{(2\delta - x)^2}{2} \, dx = \frac{4\delta^3}{(1+\delta)^3}.
\end{split}
\]
Since the event $B_{ijk}$ is mutually independent of all events $B_{i'j'k'}$ such that $|\{i,j,k\} \cap \{i',j',k'\}| \le 1$, we may invoke Lemma~\ref{lem:local-lemma} to conclude that, for every $p$ satisfying
\begin{equation}
  \label{eq:local-lemma-condition}
  \frac{4\delta^3}{(1+\delta)^3} \le p(1-p)^{3(n-3)},
\end{equation}
we have
\begin{equation}
  \label{eq:local-lemma-conclusion}
  \P(G) = \P\left(\bigcap_{\smash{\{i,j,k\} \in \binom{\br{n}}{3}}} B_{ijk}^c\right) \ge (1-p)^{\binom{n}{3}}.
\end{equation}
It is easy to see that if $p = an^{-3/2}$ for some constant $a > 1/2$, then~\eqref{eq:local-lemma-condition} is satisfied for all sufficiently large $n$.
In particular, \eqref{eq:vol_M_n_G_bound} and~\eqref{eq:local-lemma-conclusion} imply that, for each $a > 1/2$,
\begin{align*}
  \vol(\M_n) &\ge \left(1 + \frac{1}{2\sqrt{n}} \right)^{\binom{n}{2}} \left(1 - \frac{a}{n^{3/2}}\right)^{\binom{n}{3}} \\
  & = \exp\left(\left(\frac{1}{4} - \frac{a}{6} + o(1)\right) n^{3/2}\right).
\end{align*}
Since $a$ was an arbitrary constant greater than $1/2$, this yields the lower bound in~\eqref{eq:main_volume_estimates}.\qed

\subsection{Upper bound}
\label{sec:volume-upper-bound}

In this section, we deduce the upper bound on $\vol(\M_n)$ stated in
Theorem~\ref{thm:volume_estimate} from the entropy estimate of Theorem~\ref{thm:entropy_bound_for_almost_independent}.

Let $n \ge 3$ and let $d$ be a uniformly sampled metric space in $\M_n$, which we view as a vector in $\R^{\binom{\br{n}}{2}}$. By Lemma~\ref{lem:entropy-compact-support},
\[
  \log(\vol(\M_n)) = H(d).
\]
For each $m \in \{0, \dotsc, n-1\}$, let us denote by $F_m$ the set of all pairs $ij \in
\binom{\br{n}}{2}$ with $\max\{i,j\} > n-m$, that is,
\begin{equation}\label{eq:F_m_def}
  F_ m := \binom{\br{n}}{2} \setminus \binom{\br{n-m}}{2}
\end{equation}
and set, for $m \le n-2$,
\begin{equation}\label{eq:h_m def}
  h_m := H(d_{12} \mid (d_e)_{e \in F_m}),
\end{equation}
where $h_0=H(d_{12})$. Observe that, by symmetry, $h_m = H(\dist{i}{j} | (d_e)_{e \in F_m})$
for every $ij \in \binom{\br{n-m}}{2}$. Since $d_{12}\in[0,2]$, then $h_0 \le \log 2$,
by Lemma~\ref{lem:entropy-compact-support}.
Since $F_m \subseteq F_{m+1}$ for every $m$, it follows from Lemma~\ref{prop:basic_entropy_properties}~\ref{item:entropy-prop-4} that
\begin{equation}\label{eq:conditional_entropy_monotonicity}
  h_{n-2} \le \cdots \le h_1 \le h_0 \le \log 2.
\end{equation}
Additionally, as $F_0 = \emptyset$ and $F_{n-1}= \binom{\br{n}}{2}$, Lemma~\ref{prop:basic_entropy_properties}~\ref{item:entropy-prop-1} and~\ref{item:entropy-prop-3} give
\begin{align}
  H\left((d_e)_{e \in \binom{\br{n}}{2}}\right) & = \sum_{m = 0}^{n-2} H\left((d_e)_{e \in F_{m+1} \setminus F_m} \mid (d_f)_{f \in F_m}\right) \nonumber\\
                                                & \le \sum_{m = 0}^{n-2} \sum_{e \in F_{m+1} \setminus F_m} H\left(d_e \mid (d_f)_{f \in F_m}\right) \nonumber\\
                                                & = \sum_{m = 0}^{n-2} |F_{m+1} \setminus F_m| \cdot h_m
                                                  = \sum_{m=0}^{n-2} (n-m-1) \cdot h_m. \label{eq:volume_decomposition}
\end{align}
In particular, it suffices to prove the following estimate.

\begin{lem}
  \label{lem:h_m_bound}
  There exists a $K > 0$ such that, for all $m \in \{0, \dotsc, n-2\}$,
  \begin{equation*}
    h_m\le \frac{K}{\sqrt{m+1}}.
  \end{equation*}
\end{lem}

Indeed, substituting this bound into \eqref{eq:volume_decomposition} gives
\begin{equation*}
  \log(\vol(\M_n)) = H(d) \le K \sum_{m=1}^{n-1}\frac{n-m}{\sqrt{m}} \le C n^{3/2}
\end{equation*}
for some absolute constant $C>0$, establishing the theorem.

Lemma~\ref{lem:h_m_bound} is a fairly simple consequence of the monotonicity of the sequence $(h_m)$ and the following estimate, which lies at the heart of the matter.

\begin{lem}
  \label{lem:h_m_bound_with_difference}
  There exists a $C>0$ such that, for all $m \in \{0, \dotsc, n-3\}$,
  \begin{equation*}
    h_m\le C\left(h_m - h_{m+1}\right)^{1/3}.
  \end{equation*}
\end{lem}
\begin{proof}
  Fix an $m \in \{0, \dotsc, n-3\}$. Since $\{1,n-m\}, \{2,n-m\} \in F_{m+1} \setminus F_m$, Lemma~\ref{prop:basic_entropy_properties}~\ref{item:entropy-prop-4} implies that
  \[
    h_{m+1} \le  H(d_{12} \mid (d_e)_{e \in F_m}, d_{1,n-m}, d_{2,n-m}) \le h_m.
  \]
  By symmetry, we may replace $n-m$ in the above inequality with any element of $\{3, \dotsc, n-m\}$. In particular,
  \begin{equation*}
    H(d_{12} \mid (d_e)_{e \in F_m}, d_{13}, d_{23}) \ge h_{m+1}
  \end{equation*}
  and thus,
  \begin{equation}
    \label{eq:entropy-bound-hm}
    H(d_{12} \mid (d_e)_{e \in F_m}) - H(d_{12} \mid (d_e)_{e \in F_m}, d_{13},
    d_{23}) \le h_m - h_{m+1}.
  \end{equation}

  Condition on all the distances $d_{ij}$ with $ij \in F_m$ and denote by $(X_1, X_2, X_3)$ the random vector whose distribution is the conditioned distribution of $(d_{12}, d_{13}, d_{23})$, so that
  \[
    h_m = H(d_{12} \mid (d_e)_{e \in F_m}) = \E[H(X_1)] = \E[H(X_2)] = \E[H(X_3)].
  \]
  We write $X_1 \times X_2 \times X_3$ to denote the random variable whose distribution is the product of the marginal distributions of $X_1$, $X_2$, and $X_3$.

  \begin{claim}
    \label{claim:PXnotinM3}
    We have
    \[
      \E\big[\P(X_1 \times X_2 \times X_3 \notin \overline{\M_3})\big] \le 2(h_m - h_{m+1}).
    \]
  \end{claim}
  \begin{proof}[Proof of Claim~\ref{claim:PXnotinM3}]
  By~\eqref{eq:entropy_KL_relation}, inequality~\eqref{eq:entropy-bound-hm} is equivalent to
  \begin{equation}
    \label{eq:DKL-bound-hm}
    \E\big[\DKL{(X_1, X_2, X_3)}{X_1 \times (X_2, X_3)}\big] \le h_m - h_{m+1}.
  \end{equation}
  By symmetry, inequality~\eqref{eq:entropy-bound-hm} continues to hold for any permutation of $(d_{12}, d_{13}, d_{23})$ and hence~\eqref{eq:DKL-bound-hm} continues to hold for any permutation of $(X_1, X_2, X_3)$. It thus follows from Lemma~\ref{lem:DKL-triangle-ineq} that
  \begin{equation}
    \label{eq:conditioned DKL}
    \E\big[\DKL{(X_1,X_2,X_3)}{X_1 \times X_2 \times X_3}\big] \le 2(h_m - h_{m+1}).
  \end{equation}
  Since $(X_1, X_2, X_3) \in \overline{\M_3}$ with probability one, Claim~\ref{claim:PXnotinM3} now follows from~\eqref{eq:conditioned DKL} and Lemma~\ref{lem:Kullback_Leibler_and_support}.
  \end{proof}

  We may now apply Theorem~\ref{thm:entropy_bound_for_almost_independent} (to the distribution of $X_1 \times X_2 \times X_3$) to bound the entropies of $X_1$, $X_2$, and $X_3$. Lemma~\ref{lem:independent_max_volume} states that the largest volume of an axis-parallel box contained in $\overline{\M_3}$ is one and thus, by Theorem~\ref{thm:entropy_bound_for_almost_independent},
  \[
    H(X_1) + H(X_2) + H(X_3) \le C \cdot \P(X_1 \times X_2 \times X_3 \notin \overline{\M_3})^{1/3}.
  \]
  By Jensen's inequality (applied to the concave function $x \mapsto x^{1/3}$) and Claim~\ref{claim:PXnotinM3},
  \[
    3h_m = \E\big[H(X_1) + H(X_2) + H(X_3) \big] \le C \cdot \big( 2(h_m-h_{m+1}) \big)^{1/3},
  \]
    as we wanted to prove.
\end{proof}

\begin{proof}[Proof of Lemma~\ref{lem:h_m_bound}]
  Let $K$ be a sufficiently large absolute constant, to be fixed later. If $h_m \le \frac{K}{\sqrt{m+1}}$ for all $m$ then there is nothing to prove. Otherwise, aiming to obtain a contradiction, define
  \begin{equation*}
    m_0:=\min\left\{m : 0\le m\le n-2\text{ and }h_m > \frac{K}{\sqrt{m+1}}\right\}.
  \end{equation*}
  Taking $K\ge \log 2$ we necessarily have $m_0\ge 1$ by \eqref{eq:conditional_entropy_monotonicity}. It follows from Lemma~\ref{lem:h_m_bound_with_difference} and the definition of $m_0$ that
  \begin{equation*}
    h_{m_0-1}\le C\left(h_{m_0-1} - h_{m_0}\right)^{1/3} \le C\left(\frac{K}{\sqrt{m_0}} - \frac{K}{\sqrt{m_0+1}}\right)^{1/3}.
  \end{equation*}
  As $\frac{1}{\sqrt{x}} - \frac{1}{\sqrt{x+1}}\le \frac{1}{2x^{3/2}}$ for all $x>0$, we may continue the above inequality to obtain
  \begin{equation*}
    h_{m_0-1}\le \frac{CK^{1/3}}{2^{1/3}\sqrt{m_0}}\le \frac{K}{\sqrt{m_0+1}}
  \end{equation*}
  if only $K$ is sufficiently large compared with $C$. Fix $K$ to satisfy this condition. This contradicts the definition of $m_0$ as $h_{m_0}\le h_{m_0-1}$ by~\eqref{eq:conditional_entropy_monotonicity}. This finishes the proof of Lemma \ref{lem:h_m_bound}, and thus also of Theorem \ref{thm:volume_estimate}.
\end{proof}

\subsection{The lower tail of a typical distance}
\label{sec:lower-tail-typical}

The proof of Theorem~\ref{thm:minimal_distance}, presented in Section~\ref{sec:distance} below, uses as input an upper bound on $\P(d_{12}<1)$ where, as before, $(d_{ij})$ denotes a uniformly chosen metric space in $\M_n$. We record this in the following result, which further points out a nearly-matching lower bound.

\begin{prop}\label{prop:distance-lower-tail}
There are absolute constants $C, c>0$ such that
\[
  \frac{c}{\sqrt{n} \log(n+1)}\le \P( d_{12} < 1) \le \frac{C}{\sqrt{n}}.
\]
\end{prop}
We continue to use the notation $H(X;A):=-\int_A f(x)\log(f(x))dx$ for a random variable $X$ with bounded and compactly-supported density $f$ and measurable $A$.

The lower bound in Proposition~\ref{prop:distance-lower-tail} is a simple consequence of the volume lower bound in Theorem~\ref{thm:volume_estimate}. To see this, assume, to reach a contradiction, that the lower bound does not hold. Then, by Lemma~\ref{lem:entropy_bound_for_sub_probability} and the fact that $\log y\le y-1$ for all $y>0$,
\begin{equation*}
\begin{split}
  H(d_{12}) &= H(d_{12};[0,1)) + H(d_{12};[1,2])\\
  &\le \P(d_{12}<1)\log\left(\frac{1}{\P(d_{12}<1)}\right) + \P(d_{12}\ge1)\log\left(\frac{1}{\P(d_{12}\ge1)}\right)\\
  &\le \P(d_{12}<1)\log\left(\frac{1}{\P(d_{12}<1)}\right) + 1-\P(d_{12}\ge1)\\
  &= \P(d_{12}<1)\log\left(\frac{e}{\P(d_{12}<1)}\right)\le \frac{2c}{\sqrt{n}}
\end{split}
\end{equation*}
for each $c>0$ and $n$ sufficiently large. Thus, by the subadditivity of entropy,
\begin{equation*}
  \log(\vol(\M_n)) \le \binom{n}{2}H(d_{12})\le c n^{3/2}
\end{equation*}
for all large $n$. For small $c$, this leads to a contradiction with the lower bound in Theorem~\ref{thm:volume_estimate}.

We proceed to prove the upper bound in Proposition~\ref{prop:distance-lower-tail}. The following lemma, which relies on the entropy bounds of Section~\ref{sec:entropy-maximising product distributions}, is the main ingredient.
\begin{lem}\label{lem:sum of entropies with lower tail}
  There exist absolute constants $C,c>0$ such that the following holds. Let $X_1, X_2, X_3$ be \emph{independent} random variables supported in $[0,2]^3$. Then
  \begin{equation*}
    \sum_{i=1}^3 H(X_i)\le C\,\P((X_1, X_2, X_3)\notin\overline{\M_3})^{1/3} - c\sum_{i=1}^3\P(X_i < 1).
  \end{equation*}
\end{lem}
\begin{proof}
  Set $\eps:=\P((X_1, X_2, X_3)\notin\overline{\M_3})^{1/3}$ and use~\eqref{eq:LQ UQ def} to define $(a_i), (b_i)$ for $1\le i\le 3$. Proposition~\ref{prop:entropy_bound_for_almost_independent} and \eqref{eq:vol_P_lemma_estimate}
  imply that
  \begin{equation}\label{eq:initial entropy bound for minimum distance}
    \sum_{i=1}^3 H(X_i)\le \frac{1}{2}\sum_{i=1}^3 H(X_i; [a_i, b_i]) + C\eps
  \end{equation}
  for an absolute $C>0$.

  We proceed to estimate the right-hand side of~\eqref{eq:initial entropy bound for minimum distance}. First, Lemma~\ref{lem:independent-box-poly} shows that $P:=[a_1, b_1]\times [a_2, b_2]\times [a_3,b_3]$ is contained in the closure of $\M_3$. Hence, Lemma~\ref{lem:independent_max_volume} implies that
  \begin{equation}\label{eq:V upper bound}
    V:=\vol(P)\le 1 - c_1\sum_{i=1}^3 \big( |a_i - 1| + |b_i - 2| \big)
  \end{equation}
  for an absolute $0<c_1<1$. Second, we shall prove that
  \begin{multline}\label{eq:middle entropy bound}
    H(X_i;[a_i,b_i]) \le (1-2\eps)\log(b_i-a_i) + 2\eps \\
     + c_1\left(|a_i-1| + |b_i - 2| - \frac{\P(X_i<1)}{20}\right) + \eps
  \end{multline}
  for each $1\le i\le 3$. Lastly, plugging this estimate in~\eqref{eq:initial entropy bound for minimum distance} gives
  \begin{multline*}
    \sum_{i=1}^3 H(X_i)\le \frac{1-2\eps}{2}\log(V) + (C+\tfrac{9}{2})\eps\\ + \frac{c_1}{2}\sum_{i=1}^3 \left(|a_i-1| + |b_i-2| - \frac{\P(X_i<1)}{20}\right).
  \end{multline*}
  Using $\log(V)\le V-1$ and \eqref{eq:V upper bound} gives
  \[
  \frac{1-2\eps}{2}\log(V)\le\frac{1}{2}\bigg(-c_1\sum_{i=1}^3(|a_i-1|+|b_i-2|)\bigg)+9c_1\eps,
  \]
  where we used the inequality $|a_i-1| + |b_i-2| \le 3$ to bound the error term. We see that the terms containing $\sum_i |a_i-1|+|b_i-2|$ cancel and we are left with
  \[
  \sum_{i=1}^3H(X_i)\le C\eps-\frac{c_1}{40}\sum_{i=1}^3\P(X_i<1),
  \]
  as needed.

  It remains to prove~\eqref{eq:middle entropy bound}. Fix $1\le i\le 3$. Lemma~\ref{lem:entropy_bound_for_sub_probability} and the inequality $\log y \le y -1$, valid for all $y > 0$, give that
  \[
      H(X_i;[a_i, b_i]) \le(1-2\eps)\log\left(\frac{b_i-a_i}{1-2\eps}\right)\le (1-2\eps)\log(b_i-a_i) + 2\eps.
  \]
  Thus it suffices to show that the sum of the third and fourth terms in~\eqref{eq:middle entropy bound} is non-negative. This is the case if: (i) $a_i\ge1$, since $c_1<1$ and the definition of $a_i$ implies that $\P(X_i<a_i)=\eps$ (see~\eqref{eq:X-quantiles}); (ii) $b_i\le1$; (iii) $b_i - a_i\le \frac{1}{2}$; or (iv) $a_i<1$ and $\P(a_i<X_i<1)\le 20(1-a_i)$ since
  \begin{equation}\label{eq:less than 1 and a_i}
    \P(X_i<1) = \P(a_i<X_i<1) + \P(X_i<a_i) = \P(a_i<X_i<1) + \eps.
  \end{equation}
  We thus assume that $a_i<1$, $b_i>1$, $b_i-a_i>\frac{1}{2}$ and $\P(a_i<X_i<1)>20(1-a_i)$. In particular,
  \begin{equation*}
    \frac{\P(a_i<X_i<1)}{\P(a_i<X_i<b_i)}\ge 10 \cdot \frac{1-a_i}{b_i-a_i}
  \end{equation*}
  Applying the second clause of Lemma~\ref{lem:entropy_bound_for_sub_probability}, with the partition $[a_i,b_i]=[a_i,1)\cup[1,b_i]$, then shows that
  \begin{equation*}
    H(X_i;[a_i, b_i])\le(1-2\eps)\log\left(\frac{b_i-a_i}{1-2\eps}\right) -  \frac{\P(a_i<X_i<1)}{4}.
  \end{equation*}
  This implies~\eqref{eq:middle entropy bound}, again using~\eqref{eq:less than 1 and a_i} and the fact that $c_1<1$.
\end{proof}
Now recall the notation $F_m$ and $h_m$ from~\eqref{eq:F_m_def} and~\eqref{eq:h_m def}, respectively. The following simple lemma is our second ingredient in the proof of the upper bound in Proposition~\ref{prop:distance-lower-tail}.
\begin{lem}\label{eq:m for conditioning}
  There exists an absolute constant $K>0$ and some $\frac{1}{3}n \le m \le \frac{2}{3}n$ for which
  \begin{equation*}
    h_m>-\frac{K}{\sqrt{n}}\quad\text{and}\quad h_m - h_{m+1}\le \frac{K}{n^{3/2}}.
  \end{equation*}
\end{lem}
\begin{proof}
  Recall that $m \mapsto h_m$ is decreasing (\ref{eq:conditional_entropy_monotonicity})  and $h_m\le C/\sqrt{m+1}$ for each $m$, by Lemma \ref{lem:h_m_bound}. We first claim that, for some $K$ sufficiently large,
  \begin{equation}\label{eq:hm from below}
    h_{\lceil 2n/3 \rceil}>-\frac{K}{\sqrt{n}}.
  \end{equation}
  Indeed, if it were not the case then from (\ref{eq:volume_decomposition}) we would get
  \begin{align*}
    \log(\vol(\M_n))
    & \stackrel{\textrm{\clap{(\ref{eq:volume_decomposition})}}}{\le}\;
      \sum_{m=0}^{n-2}(n-m-1)h_m\\
    &\le \sum_{m=0}^{\lceil 2n/3 \rceil-1}\frac {C(n-m-1)}{\sqrt{m+1}} -
      \sum_{m=\lceil 2n/3 \rceil}^{n-2}\frac{K(n-m-1)}{\sqrt{n}}\\
    & \le \;Cn^{3/2}-Kn^{3/2},
  \end{align*}
  which would contradict the fact that $\vol(\M_n) \ge 1$, provided that $K$ is sufficiently large.

  Finally, it follows from the pigeonhole principle that for some $m$ with $\lceil n/3 \rceil \le m< \lceil 2n/3 \rceil$, we have
  \begin{equation}
    \label{eq:hm32}
    h_m-h_{m+1}\le \frac{h_{\lceil n/3 \rceil} - h_{\lceil 2n/3 \rceil}}{\lfloor n/3 \rfloor} \le \frac{C/\sqrt{n/3+1} + K/\sqrt{n}}{\lfloor n/3 \rfloor} \le \frac{K}{n^{3/2}}.
  \end{equation}
  Moreover, $h_m \ge h_{\lceil 2n/3 \rceil} > -K/\sqrt{n}$, as claimed.
\end{proof}
We now finish the proof of the upper bound in Proposition~\ref{prop:distance-lower-tail}.

Let $\frac{1}{3}n \le m \le \frac{2}{3}n$ be as in Lemma~\ref{eq:m for conditioning}. Condition on all the distances $d_{ij}$ with $ij \in F_m$ and write $(X_1,X_2,X_3)$ for the conditional versions of $(d_{12},d_{13},d_{23})$. The distribution of $(X_1, X_2, X_3)$ is regarded as random (a function of the variables conditioned upon). Lemma~\ref{lem:sum of entropies with lower tail} shows that
\begin{equation*}
  \sum_{i=1}^3 H(X_i)\le C\,\P(X_1 \times X_2 \times X_3 \notin \M_3)^{1/3} - c\sum_{i=1}^3\P(X_i < 1).
\end{equation*}
Averaging over the conditioning, using Jensen's inequality (for the concave function $x\mapsto x^{1/3}$), and applying Claim~\ref{claim:PXnotinM3}, we conclude that
\begin{equation*}
\begin{split}
  3h_m\le C\,(2h_m - 2h_{m+1})^{1/3} - 3c\P(d_{12}<1)\\
\end{split}
\end{equation*}
Thus, by Lemma~\ref{eq:m for conditioning},
\begin{equation*}
  \P(d_{12}<1)\le \frac{C'}{\sqrt{n}}
\end{equation*}
for some absolute constant $C'$, finishing the proof.

\section{The shortest distance in the metric space}

\label{sec:distance}

In this section, we prove Theorem~\ref{thm:minimal_distance}, showing that, with high probability, the minimum distance in a uniformly chosen metric space from $\M_n$ is only polynomially shorter than one. In order to introduce several key ideas used in the proof of the theorem, we first sketch an argument yielding the weaker result that all distances are larger than $2^{-8}$. This result will not need any fine estimates on the volume and it will yield an exponential bound on the probability of having a short distance. The first step is the following simple proposition.

\begin{prop}
  \label{prop:exp-bound-alpha}
  For every $\alpha \in (0,1/2]$,
  \[
    \vol\big(\{d \in \M_n : \min_{i,j}d_{ij}\le \alpha\}\big) \le \binom{n}{2} (2\alpha)^{n-2} \cdot \vol(\M_{n-1}).
  \]
\end{prop}
\begin{proof}
  By symmetry, it suffices to show that the volume of those $d \in \M_n$ for which $d_{n-1,n} \le \alpha$ is at most $(2\alpha)^{n-2} \cdot \vol(\M_{n-1})$. Assume that $d_{n-1,n} \le \alpha$ and note that, for each $i \in \br{n-2}$, the distance $d_{in}$ must belong to the interval $[d_{i,n-1}-\alpha, d_{i,n-1}+\alpha]$. In other words, the volume of the possible values for $(d_{in})_{i=1}^{n-2}$, given all the other distances, is at most $(2\alpha)^{n-2}$. This gives the desired estimate.
\end{proof}

Suppose that $d$ is sampled uniformly from $\M_n$. We could already conclude that $\P(\min_{ij} d_{ij} \le \alpha)$ is exponentially small in $n$, for every constant $\alpha < 1/2$, if we knew that $\vol(\M_{n-1}) \le e^{o(n)} \cdot \vol(\M_n)$. Such an estimate does indeed hold, as will be shown in Proposition~\ref{prop:volume_comparison}. Since the proof of Proposition~\ref{prop:volume_comparison} is rather involved (even though it is quite natural to conjecture that $n \mapsto \vol(\M_n)$ is increasing, see Section~\ref{sec:further_questions}) and it crucially relies on the volume estimate provided by Theorem~\ref{thm:volume_estimate}, let us sketch here a self-contained argument showing that
\begin{equation}
  \label{eq:vol-Mn-comparison}
  \vol(\M_n) \ge 2^{-6n} \cdot \vol(\M_{n-1}),
\end{equation}
which is enough to deduce that, for some constants $c, C > 0$,
\begin{equation}
  \label{eq:exponential_probability_for_very_small_distances}
  \P(\min_{i,j}d_{ij}\le 2^{-8})\le Ce^{-cn}.
\end{equation}

\begin{proof}[Sketch of a proof of~\eqref{eq:vol-Mn-comparison}]
  Define
  \[
    F(d) := \min_{A \subseteq \br{n}} \prod_{i \in A} \left(2\min_{j\in\br{n}\setminus\{i\}}d_{ij}\right)
  \]
  (so that $F(d) = 1$ if $d_{ij} \ge 1/2$ for all $\{i, j\}$).  We claim that, for all sufficiently large $n$ and all $F \in (0,1)$,
  \begin{equation}
    \label{eq:F-lemma-basic}
    \vol\big(\{d\in\M_n: F(d) \le F\}\big) \le F^{n/10} \cdot 2^{\binom{n}{2}}.
  \end{equation}
  Since a stronger estimate will be proved in Lemma~\ref{lem:F-lemma}, we only sketch the main idea here. The proof of~\eqref{eq:F-lemma-basic} is similar in spirit to the calculation done in the proof of Proposition~\ref{prop:exp-bound-alpha}. It relies on the key observation that, if $d_{ij}$ is small, the $n-2$ pairs of distances $(d_{ik}, d_{jk})$ are constrained to a strip in $[0,2]^{2}$ of width $2d_{ij}$. In particular, if $F(d)$ is small, then this significantly constrains all distances. For details, we refer the reader to the proof of Lemma~\ref{lem:F-lemma}.

  Examine the set
  \[
    \M_n^1:=\{d\in\M_n:F(d)>2^{-5n}\}.
  \]
  It follows from~\eqref{eq:F-lemma-basic} that
  \[
    \vol(\M_n\setminus \M_n^1)\le \frac{1}{2}\le \frac{1}{2}\vol(\M_n),
  \]
  so that $\vol(\M_n^1)\ge \frac 12 \vol(\M_n)$. We claim that the volume of possible extensions of any fixed $d\in\M_n^1$ to a metric space in $\M_{n+1}$ is reasonably large. Indeed, denote $I(\rho):=[3/2-\rho/2,\,3/2+\rho/2]$ and extend $d$ to $[0,2]^{\binom{\br{n+1}}{2}}$ by requiring that, for all $i \in \br{n}$,
  \[
    d_{i,n+1}\in I\left(\min\left\{\min_{j\in \br{n}\setminus\{i\}}d_{ij},1\right\}\right)
  \]
  It is straightforward to check that one obtains a metric space, and further, that the volume of the extension is at least $F(d)/2^n$. (A version of this argument is presented in the proof of Proposition~\ref{prop:volume_comparison}.) Hence,
  \[
    \vol(\M_{n+1})\ge 2^{-6n} \cdot \vol(\M_n^1)\ge 2^{-6n-1} \cdot \vol(\M_n).\qedhere
  \]
\end{proof}

Let us point out here that, regardless of other losses in the argument above, using Proposition~\ref{prop:exp-bound-alpha} or examining the quantity $F$ gives absolutely no information about distances between $\frac{1}{2}$ and $1$; for these, more involved analysis is required.


\subsection{{Proof of Theorem~\ref{thm:minimal_distance}}}
\label{sec:minimum-distance}

Giving up optimising various estimates in favour of simplifying the presentation (and because we believe that further ideas would be needed to obtain the optimal value of $c$), we shall prove the theorem with
\[
  c = 1/30.
\]
The starting point of our proof is Proposition~\ref{prop:distance-lower-tail}, which states that there exists a constant $C$ such that, when $d$ is a uniformly sampled metric space from $\M_n$,
\begin{equation}\label{eq:unlikely_short_edge}
  \P(d_{ij}<1) \le Cn^{-1/2} \quad \text{for every $\{i,j\}\in\binom{\br{n}}{2}$}.
\end{equation}
This allows us to conclude that a typical metric space sampled from $\M_n$ has relatively few distances shorter than one. More precisely, letting
\[
  \G_n := \big\{d\in \M_n: \text{$d_{ij} < 1$ for at most $n^{5/3}$ pairs $\{i,j\}$}\big\},
\]
we have
\begin{equation}
  \label{eq:used_to_be_lem}
  \vol(\G_n) > \left(1 - Cn^{-1/6}\right)\vol(\M_n).
\end{equation}
To see \eqref{eq:used_to_be_lem}, let $d$ be a uniformly sampled metric space in $\M_n$ and let $X$ be the number of pairs $\{i,j\}$ such that $d_{ij}<1$. By Markov's inequality and~\eqref{eq:unlikely_short_edge}, we have
\[
  \P(X > n^{5/3}) < \frac{\E[X]}{n^{5/3}} \le Cn^{-1/6},
\]
as needed. In particular, we may restrict our attention to spaces in $\G_n$. Define
\[
  \B_n :=\{d\in\G_n : \min_{i,j} d_{ij} < 1-n^{-c}\}.
\]
Our argument will comprise two independent parts. First, we will show that the volume of $\B_n$ is extremely small when compared to the volume of $\M_{n-2}$.

\begin{prop}
  \label{prop:volume-Bn}
  For all sufficiently large $n$, we have
  \[
    \vol(\B_n) \le \exp\left(-\frac{n^{1-2c}}{5}\right) \cdot \vol(\M_{n-2}).
  \]
\end{prop}

This bound would yield the desired result if we knew that $\vol(\M_{n-2})$ is not much larger than $\vol(\M_n)$. It seems plausible that, in fact,
\begin{equation}
  \label{eq:growing_volume}
  \vol(\M_n)\ge\vol(\M_{n-2})
\end{equation}
holds for all $n$. However, we have been unable to establish this, see Section~\ref{sec:discussion_and_open_questions}. We should point out that the volume estimates of Theorem~\ref{thm:volume_estimate} imply that \eqref{eq:growing_volume} holds for an infinite sequence of $n$ and thus Proposition~\ref{prop:volume-Bn} is sufficient to yield the assertion of Theorem~\ref{thm:minimal_distance} for that sequence. In order to establish the theorem for all sufficiently large $n$, we shall prove the following weaker bound, which still suffices for our purposes.

\begin{prop}\label{prop:volume_comparison}
  For all sufficiently large $n$, we have
  \[
    \vol(\M_{n+1}) \ge \exp\left(-n^{1-3c} \log(n) \right) \cdot \vol(\M_n).
  \]
\end{prop}

We postpone the proofs of Propositions~\ref{prop:volume-Bn} and~\ref{prop:volume_comparison} to the next two sections and finish the current section with a short derivation of Theorem~\ref{thm:minimal_distance}.

\begin{proof}[Proof of Theorem~\ref{thm:minimal_distance}]
  Recalling the definitions of $\G_n$ and $\B_n$, we have
  \[
    \P(\min_{i,j} d_{ij}<1-n^{-c}) \le \frac{\vol(\M_n\setminus \G_n)}{\vol(\M_n)}+\frac{\vol(\B_n)}{\vol(\M_n)}.
  \]
  Estimate \eqref{eq:used_to_be_lem} states that the first term in the right-hand side is at most $C n^{-1/6}$ whereas Propositions~\ref{prop:volume-Bn} and~\ref{prop:volume_comparison} give
  \[
    \begin{split}
      \frac{\vol(\B_n)}{\vol(\M_n)} & \le \exp\left(-\frac{n^{1 - 2c}}{5}\right) \cdot \frac{\vol(\M_{n-2})}{\vol(\M_n)} \\
      & \le \exp\left(-\frac{n^{1-2c}}{5} + 2n^{1-3c}\log(n)\right) \le \exp\left(-\frac{n^{1-2c}}{6}\right),
    \end{split}
  \]
  provided that $n$ is sufficiently large.
\end{proof}

\subsection{Bounding the volume of spaces with a short distance}
\label{sec:bound-volume-spac}

In this section, we prove Proposition~\ref{prop:volume-Bn}. We shall split the set $\B_n$ into two parts, depending on whether or not there is a point $i \in \br{n}$ at distance significantly shorter than one from many other points, and use different arguments to estimate the volume of each of these parts. More precisely, for a metric space $d\in\M_n$ and an $i \in \br{n}$, we define the set of close neighbours of $i$ by
\begin{equation*}
  S_i(d) :=\left\{j\in \br{n}\setminus\{i\} : \dist{i}{j} < 1 - \frac{n^{-2c}}{4}\right\}
\end{equation*}
and let $m := \lfloor n^{1-3c} \rfloor$.

Our first lemma uses Theorem~\ref{thm:volume_estimate} to provide a very strong upper bound on the volume of all spaces $d \in \G_n$ (and not only $d \in \B_n$) for which $|S_i(d)| > m$ for some $i \in \br{n}$.
\begin{lem}
  \label{lem:Si-large}
  For all sufficiently large $n$, we have
  \[
    \vol\big(\{d\in\G_n : \max_i |S_i(d)| > m\}\big) \le \exp\left(-\frac{n^{2-8c}}{16}\right).
  \]
\end{lem}

Our second lemma bounds the volume of those $d \in \B_n$ for which $|S_i(d)| \le m$ for all $i \in \br{n}$ in terms of the volume of $\M_{n-2}$.

\begin{lem}
  \label{lem:Si-small}
  For all sufficiently large $n$, we have
  \[
    \vol\big(\{d\in\B_n : \max_i |S_i(d)| \le m\}\big) \le \exp\left(-\frac{n^{1 -2c}}{4}\right) \cdot \vol(\M_{n-2}).
  \]
\end{lem}

\begin{proof}[Proof of Proposition~\ref{prop:volume-Bn}]
  Using the estimates of the two lemmas, we may conclude that, for all sufficiently large $n$,
  \[
    \begin{split}
      \vol(\B_n) &\le \exp\left(-\frac{n^{1 -  2c}}{4}\right) \cdot \vol(\M_{n-2}) + \exp\left(-\frac{n^{2-8c}}{16}\right) \\
      & \le \exp\left(-\frac{n^{1-2c}}{5}\right) \cdot \vol(\M_{n-2}),
    \end{split}
  \]
  as $c < 1/6$ and $\vol(\M_{n-2}) \ge 1$.
\end{proof}

\begin{proof}[Proof of Lemma~\ref{lem:Si-large}]
  For $i \in \br{n}$, $S \subseteq \br{n}$ with $|S| = m$, and $T\subseteq\binom{\br{n}}{2}$ with $|T| = \lfloor  n^{5/3}\rfloor$, we let
  \begin{equation*}
    \G_n^{i,S,T}:=\big\{d\in\G_n : S_i(d) \supseteq S\text{ and } d_{jk} \ge 1\text{ if }\{j,k\}\notin T\big\}.
  \end{equation*}
  Note that if $d\in\G_n^{i,S,T}$ and $\{j,k\}\in\binom{S}{2}$, then necessarily $d_{jk} \le 2(1 - n^{-2c}/4)$, as follows from the triangle inequality $d_{jk} \le d_{ij} + d_{ik}$. Thus, $\G_n^{i,S,T}$ is contained in the product set
  \begin{multline*}
    \left\{\left(d_{jk} \right)_{\{j,k\}\in\binom{S}{2}}\in\left(1 -
    \frac{n^{-2c}}{4}\right)\cdot\M_{|S|}\right\}\\
    \times\prod_{\{j,k\}\in
  T\setminus\binom{S}{2}}\{d_{jk}\le 2\}\prod_{\{j,k\}\in
  \binom{\br{n}}{2}\setminus\left(T\cup\binom{S}{2}\right)}\{1\le d_{jk} \le
  2\}.
  \end{multline*}
  It follows that
  \begin{equation*}
    \vol(\G_n^{i,S,T})\le \left(1 - \frac{n^{-2c}}{4}\right)^{\binom{|S|}{2}}\vol(\M_{|S|})\cdot 2^{|T|}\cdot 1.
  \end{equation*}
  Estimating $\vol(\M_{|S|})$ using Theorem~\ref{thm:volume_estimate} gives
  \begin{equation*}
    \vol(\G_n^{i,S,T})\le \exp\left(-\frac{n^{-2c}}{4}\binom{m}{2}+
    C_1 m^{3/2} + n^{5/3}\right)\le
    \exp\left(-\frac{n^{2-8c}}{10}\right),
  \end{equation*}
  where we have used that $m = \lfloor n^{1-3c} \rfloor$, that $c$ is sufficiently small (so that $2-8c > 5/3$) and that $n$ is sufficiently large. Summing over all possible choices for $i$, $S$, and $T$ yields
  \begin{multline*}
    \vol\big(\{d\in\G_n : \max_i |S_i(d)| > m\}\big)\le n\binom{n}{m}\binom{n^2}{\lfloor
  n^{5/3}\rfloor}\exp\left(-\frac{n^{2-8c}}{10}\right) \\
    \le \exp\left((m+1) \log(n) + n^{5/3} \log(n^2)-\frac{n^{2-8c}}{10}\right)\le
    \exp\left(-\frac{n^{2-8c}}{16}\right),
  \end{multline*}
  where we again used the assumption that $2-8c > 5/3$.
\end{proof}

\begin{proof}[Proof of Lemma~\ref{lem:Si-small}]
  For $S \subseteq \br{n}$ with $|S| = 2m$, we let
  \begin{equation*}
    \B_n^S := \big\{d \in \G_n : S_1(d)\cup S_2(d)\subseteq S \text{ and } d_{12} < 1-n^{-c}\big\}
  \end{equation*}
  and note that, by symmetry,
  \begin{equation}
    \label{eq:Bn-sum-BnS}
    \vol\big(\{d \in \B_n : \max_i |S_i(d)| \le m\}\big) \le \binom{n}{2} \sum_{S \subseteq \br{n}, |S|=2m} \vol(\B_n^S).
  \end{equation}
  The crucial observation is that if $d_{12}<1-n^{-c}$, then, for any $j\not\in S$, we have $(d_{1j},d_{2j})\in W$, where
  \begin{figure}
    \begin{centering}
      \input{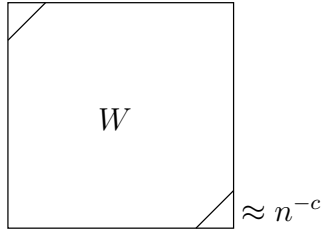}
    \end{centering}
    \caption{$W$ inside $[1-n^{-2c}/4,2]^2$\label{fig:W}.}
  \end{figure}
  \[
    W := \left\{(a,b) : 1 - \frac{n^{-2c}}{4}\le a,b\le 2,\, |a-b|\le 1-n^{-c}\right\}.
  \]
  Since
  \begin{equation*}
    \vol(W) = \left(1 + \frac{n^{-2c}}{4}\right)^2 -
    \left(\frac{n^{-2c}}{4} + n^{-c}\right)^2\le 1 -
    \frac{n^{-2c}}{2},
  \end{equation*}
  bounding the volume of $(d_{ij}:i\in\{1,2\},j\in S)$ crudely by $2^{2|S|}$, we get
  \begin{align*}
    \vol(\B_n^S) & \le 2^{2|S|} \cdot \vol(W)^{n-2-|S|} \cdot \vol(\M_{n-2}) \\
    & \le 16^m \cdot \exp\left(-\frac{1}{3}n^{1 - 2c}\right) \cdot \vol(\M_{n-2}),
  \end{align*}
  provided that $n$ is sufficiently large. Substituting this bound into~\eqref{eq:Bn-sum-BnS} gives the result, since
  \[
    \binom{n}{2} \binom{n}{2m} \le \exp\big((2m+2)\log(n)\big) \le 16^{-m} \cdot \exp\left(\frac{n^{1-2c}}{12}\right).\qedhere
  \]
\end{proof}

\subsection{Comparing volumes of metric polytopes}

This section is devoted to the proof of Proposition~\ref{prop:volume_comparison}. We show that a large portion of the spaces in $\M_n$ admit a
significant volume of extensions to spaces in $\M_{n+1}$. To this end, we study certain typical properties of metric spaces in $\G_n$.
The first step is establishing that, in a typical space in $\G_n$, there are not too many vertices that are incident to a distance that is significantly
shorter than $\frac{1}{2}$. Define, for a set $A\subseteq \br{n}$ and a space $d\in \M_n$,
\begin{equation}
  \label{eq:F_A_def}
  F_A(d):=\prod_{i\in A} \left(2\min_{j\in\br{n}\setminus\{i\}}\dist{i}{j}\right).
\end{equation}
(In particular, $F_\emptyset(d) = 1$.)

\begin{lem}
  \label{lem:F-lemma}
  If $n$ is sufficiently large, then for any $F \in (0,1)$, we have
  \begin{equation*}
    \vol\big(\{d \in \G_n : \min_{A \subseteq \br{n}} F_A(d) \le F \}\big)\le F^{n/10} \cdot \exp\left(n^{5/3}\log(n)\right).
  \end{equation*}
\end{lem}
\begin{proof}
  For a metric space $d\in\M_n$ and a set $B\subsetneq\br{n}$, define
  \[
    F_B^*(d):=\prod_{i\in B} \left(2\min_{j \in \br{n} \setminus B} d_{ij} \right).
  \]
  The difference between $F_B$ and $F_B^*$ is that, in the definition of $F_B$, the index $j$ minimising $d_{ij}$ is chosen arbitrarily while, in the definition of $F_B^*$, it is chosen from outside of $B$. For each $i \in B$, let $j_i^B(d)$ denote an (arbitrary) such index, that is, $j_i^B(d)$ is an arbitrary $j\in \br{n}\setminus B$ for which $\dist{i}{j} = \min_{k\in\br{n}\setminus B}\dist{i}{k}$. We shall shorthand $j_i(d) := j_i^{\{i\}}(d)$.

  Suppose that $d \in \G_n$ and that $F_A(d)\le F$ for some $A \subseteq \br{n}$. We first show that there exists a subset $B\subseteq A$
  such that
  \begin{equation}\label{eq:good_B_property}
    \text{$|B|\le \frac{n}{2}$\quad and\quad$F_B^*(d)\le F^{1/4}$}.
  \end{equation}
  To see this, let $R$ be a uniformly chosen random subset of $\br{n}$ with $\lfloor n/2 \rfloor$ elements, let
  \[
    B = \{i \in A \cap R :  j_i(d) \notin R\},
  \]
  and note that
  \[
    \E\left[\log\big(F_B^*(d)\big)\right] = \sum_{i \in A} \P(i \in B) \cdot \log\big(2d_{i,j_i(d)}\big) = p \cdot \log\big(F_A(d)\big),
  \]
  where
  \[
    p = \frac{\lfloor n/2 \rfloor \lceil n/2 \rceil}{n(n-1)} \ge \frac{1}{4}.
  \]
  Since $F_A(d) \le F \le 1$, there exists a choice of $R$ for which the set $B$ satisfies $F_B^*(d)\le F_{A}(d)^{1/4} \le F^{1/4}$.

  For a set $B \subseteq \br{n}$ with at most $n/2$ elements, a function $J \colon \br{n} \to \br{n}$, and a set $T \subseteq \binom{\br{n}}{2}$ with $|T| = \lfloor n^{5/3} \rfloor$, define
  \begin{multline*}
    \G_n^{B,T,J} := \Big\{d \in \M_n : F_B^*(d)\le F^{1/4}, j_i^B(d) = J(i) \text{ for all $i\in B$}, \\
    \text{ and } \dist{i}{j}\ge 1\text{ if }\{i,j\}\notin T\Big\}.
  \end{multline*}
  We may construct each space in $\G_n^{B,T,J}$ as follows. We first choose all the distances $d_{ij}$ with $\{i, j\} \in \binom{\br{n} \setminus B}{2} \cup \binom{B}{2}$ and the $|B|$ distances $d_{i,J(i)}$ with $i \in B$. Since $d_{ij} \in [1,2]$ when $ij \notin T$ and $d_{ij} \in [0,2]$ otherwise, the volume of all such choices is at most $2^{|T|}$. Now, for every $i\in B$ and $k\notin B\cup\{J(i)\}$, the distance $d_{ik}$ must satisfy $|d_{ik}-d_{J(i),k}|\le d_{i,J(i)}$. As a result, given all the other distances, the volume of the set of valid choices for all such $d_{ik}$ is not more than
  \begin{equation*}
    \prod_{i\in B\vphantom{k\in \br{n}\setminus(B\cup\{J(i)\})}} \prod_{k \notin B\cup\{J(i)\}} 2\dist{i,}{J(i)} =F_B^*(d)^{n - |B| -
      1} \le F^{(n-|B|-1)/4}.
  \end{equation*}
  We thus get
  \begin{equation*}
      \vol(\G_n^{B,T,J}) \le 2^{|T|} \cdot F^{(n-|B|-1)/4} \le 2^{n^{5/3}}  F^{n/10},
  \end{equation*}
  for $n$ sufficiently large. Summing over all possible choices for $B$, $T$, and $J$, we have
  \begin{align*}
    \vol\big(\{d \in \G : \min_{A \subseteq \br{n}} F_A(d) \le F\}\big)  & \le 2^n n^n \binom{n^2}{\lfloor n^{5/3} \rfloor} \cdot 2^{n^{5/3}}F^{n/10} \\
                                                                         & \le F^{n/10} \cdot \exp\big(n^{5/3}\log (n)\big),
  \end{align*}
  as claimed.
\end{proof}

The second step in the proof of Proposition~\ref{prop:volume_comparison} is showing that, in a typical metric space in $\G_n$, distances significantly shorter than one do not form large matchings. To this end, for a constant $\rho > 0$ and $d \in \M_n$, we define
\begin{equation}
  \label{eq:T_rho_d_def}
  T^\rho(d) := \left\{\{i,j\} \in \binom{\br{n}}{2} : d_{ij} <  1 - n^{-\rho} \right\}.
\end{equation}

\begin{lem}
  \label{lem:short-matching}
  If $\mu$ and $\rho$ are positive constants satisfying
  \begin{equation}\label{eq:mu_rho_cond}
    \mu + 2\rho < \frac{1}{3},
  \end{equation}
  then, for all sufficiently large $n$,
  \begin{multline*}
    \vol\big(\{d\in\G_n : \text{$T^\rho(d)$ contains a matching}
    \text{ of size at least $n^{1-\mu}$}\}\big) \\
    \le \exp\left(-\frac{n^{2-2\rho-\mu}}{4}\right).
  \end{multline*}
\end{lem}

\begin{proof}
  Let $\mu$ and $\rho$ be positive constants satisfying~\eqref{eq:mu_rho_cond}. For disjoint $M, T \subseteq \binom{\br{n}}{2}$ such that $M$ is a matching with $|M| = \lceil n^{1-\mu} \rceil$ and $|T| = \lfloor n^{5/3}\rfloor$, let
  \[
    \G_n^{M,T} := \left\{d \in \M_n : T^\rho(d) \supseteq M \text{ and } \dist{i}{j}\ge 1\text{ if }\{i,j\}\notin T\right\}.
  \]
  Denote by $V(M)$ the set of $2|M|$ endpoints of edges of $M$ and let $\we^{M,T}$ be the set of all triangles that contain an edge of $M$ and two edges that are not in $T$ and whose common endpoint is not in $V(M)$, that is,
  \[
    \we^{M,T} := \bigg\{(\{i,j\}, k) \in M \times \br{n} : \{i,k\}, \{j,k\} \not\in T, k \not\in V(M) \bigg\}.
  \]
  Observe that every edge in $T$ can `prevent' no more than one triangle from belonging to $\we^{M,T}$, since $M$ is a matching and since $k$ is not  allowed to be in $V(M)$. Hence,
  \begin{equation}
    \label{eq:LMT-size}
    |\we^{M,T}| \ge |M|(n-2|M|) - |T| \ge \frac{n^{2-\mu}}{2},
  \end{equation}
  as $2-\mu > 5/3$, by~\eqref{eq:mu_rho_cond}, and $n$ is sufficiently large.

  As in the proof of Lemma~\ref{lem:Si-small}, the crucial observation is that, if $(\{i, j\}, k) \in \we^{M,T}$, then $(d_{ik},d_{jk})\in W'$, where
  \begin{equation*}
    W' := \left\{(a,b) : 1\le a,b \le 2,\, |a-b|\le 1-n^{-\rho}\right\}.
  \end{equation*}
  Consequently, $\G_n^{M,T}$ is contained in the following product set:
  \begin{equation*}
    \prod_{\{i,j\}\in T}\{\dist{i}{j}\le 2\}
    \prod_{(\{i,j\},k) \in \we^{M,T}} \{(\dist{i}{k}, \dist{j}{k})\in W'\}
    \prod_{\textrm{remaining }\{i,j\}}\{1\le \dist{i}{j}\le 2\}.
  \end{equation*}
  Since
  \begin{equation*}
    \vol(W') = 1 - n^{-2\rho},
  \end{equation*}
  we conclude, using~\eqref{eq:LMT-size}, that
  \begin{align*}
    \vol(\G_n^{M,T}) & \le 2^{|T|}\cdot\vol(W')^{|\we^{M,T}|} \le 2^{n^{5/3}} \cdot \left(1 - n^{-2\rho}\right)^{\frac12 n^{2-\mu}} \\
    & \le \exp\left(- n^{2-\mu-2\rho}/3 \right),
  \end{align*}
  where in the last inequality we used \eqref{eq:mu_rho_cond}.
  Summing over all possible choices for $M$ and $T$, we have
  \begin{multline*}
    \vol\big(\{d\in\G_n : \text{$T^\rho(d)$ contains a matching of size at least $n^{1-\mu}$}\}\big) \\
    \le\binom{n^2}{\lceil n^{1-\mu} \rceil} \binom{n^2}{\lfloor n^{5/3} \rfloor} \exp\left(- n^{2-\mu-2\rho}/3 \right),
  \end{multline*}
  from which the lemma follows, again using~\eqref{eq:mu_rho_cond}.
\end{proof}

The two lemmas enable us to compare the volumes of $\M_{n}$ and $\M_{n+1}$.
\begin{proof}[Proof of Proposition~\ref{prop:volume_comparison}]
  Recall the definition of $T^\rho(d)$ from~\eqref{eq:T_rho_d_def} and the definition of $F_A(d)$ from~\eqref{eq:F_A_def}. Recall also that $c = 1/30$ and let
  \begin{equation}\label{eq:varphi_def}
    \varphi:=6c, \qquad \rho := 3c, \quad \text{and} \quad \mu := 3c.
  \end{equation}
  Define
  \begin{align*}
    \M_n^1&:=\left\{d \in \M_n : \min_{A\subseteq \br{n}} F_A(d) > \exp(-n^{1-\varphi}) \right\},\\
    \M_n^2&:=\left\{\text{$T^\rho(d)$ contains no matching of size at least
    $n^{1-\mu}$}\right\}
  \end{align*}
  and let
  \[
    \M_n^*:=\G_n \cap \M_n^1 \cap \M_n^2.
  \]
  Since $2 - \varphi > 5/3$ and $\mu + 2\rho = 9c < 1/3$, we may use estimate~\eqref{eq:used_to_be_lem}, Lemma~\ref{lem:F-lemma}, Lemma~\ref{lem:short-matching}, and the estimate $\vol(\M_n) \ge 1$ to conclude that, for sufficiently large $n$,
  \begin{equation}\label{eq:M_n_*_volume}
    \vol(\M_n^*)\ge \frac{1}{2}\vol(\M_n).
  \end{equation}
  For $d\in\M_n$ define
  \begin{equation*}
    Q(d):=\left\{i\in \br{n} : \min_{j\in\br{n}\setminus\{i\}}
    \dist{i}{j}<\frac{1}{2} - \frac{n^{-\rho}}{2}\right\}
  \end{equation*}
  and let
  \begin{equation*}
    V(d):=\left\{\text{the vertices of a largest matching in $T^\rho(d)$}\right\},
  \end{equation*}
  where, if there are several largest matchings, we let $V(d)$ to be the vertex set of an arbitrary one of them. For the sake of brevity, from now on we shall write $Q$ and $V$ in place of $Q(d)$ and $V(d)$.

  Let $d \in\M_n^*$. We aim to define a set of metric spaces in $\M_{n+1}$ which extend $d$. More precisely, we shall find a voluminous family of metric spaces $d' \in\M_{n+1}$ which satisfy
  \begin{equation}\label{eq:d_n_plus_1_prop}
    d_{ij}' = d_{ij},\quad \{i,j\}\in
    \binom{\br{n}}{2}.
  \end{equation}
  To this end, define, for $\delta > 0$,
  \begin{equation*}
    I(\delta):=\left[\frac{3}{2} - \frac{\delta}{2}, \frac{3}{2} +
    \frac{\delta}{2}\right]
  \end{equation*}
  and the following quantities
  \begin{equation*}
    \delta_1(i) :=
    \begin{cases}
      n^{-\rho},& \text{if $i\in V$}, \\
      1, & \text{otherwise},
    \end{cases}
    \qquad
    \delta_2(i) :=
    \begin{cases}
      \min_{j\ne i}d_{ij}, & \text{if $i\in Q$}, \\
      1, & \text{otherwise}.
    \end{cases}
  \end{equation*}

  \begin{claim}
    \label{claim:d-extensions}
    Every $d' \in [0,2]^{\binom{\br{n+1}}{2}}$ satisfying \eqref{eq:d_n_plus_1_prop} and having
    \begin{equation*}
      d_{i,n+1}' \in I\left(\min\big\{\delta_1(i),\delta_2(i), 1-2n^{-\rho}\big\}\right)
    \end{equation*}
    belongs to $\M_{n+1}$.
  \end{claim}
  \begin{proof}
    Since $d$ is a metric space, by~\eqref{eq:d_n_plus_1_prop}, it suffices to verify the triangle inequality for triangles $\{i,j,n+1\}$ with $\{i,j\}\in\binom{\br{n}}{2}$. Note that $d_{i,n+1}',
    d_{j,n+1}' \ge 1$ whereas $d_{ij}' = d_{ij} \le 2$, so that we only need to verify that
    \begin{equation}
      \label{eq:triangle_i_j_n_plus_1}
      |d_{i,n+1}' - d_{j,n+1}'|\le d_{ij}.
    \end{equation}
    We consider three cases, according to the value of $d_{ij}$.
    \begin{itemize}
    \item
      If $d_{ij}<\frac 12 - \frac 12 n^{-\rho}$, then $i, j\in Q$ and
      \begin{equation*}
        |d_{i,n+1}' - d_{j,n+1}'| \le \frac{1}{2}\left(\min_{k\ne i}d_{ik} + \min_{k\ne j} d_{jk} \right) \le d_{ij}.
      \end{equation*}
    \item
      If $\frac{1}{2} - \frac{1}{2} n^{-\rho} \le d_{ij}<1-n^{\rho}$, then at least one of $i$ and $j$ is in $V$, as $\{i, j\} \in T^{\rho}(d)$ and $V$ is the vertex set of a largest matching in $T^{\rho}(d)$, and
    \begin{equation*}
      |d_{i,n+1}' - d_{j,n+1}'| \le \frac{1}{2}\left(n^{-\rho} + (1-2n^{-\rho})\right)=\frac 12-\frac12 n^{-\rho} \le d_{ij}.
    \end{equation*}
  \item
    Finally, if $d_{ij}\ge 1-n^{-\rho}$, then
    \begin{equation*}
      |d_{i,n+1}' - d_{j,n+1}'| \le \frac{1}{2}\left((1-2n^{-\rho}) + (1-2n^{-\rho})\right) \le d_{ij}.
    \end{equation*}
  \end{itemize}
  The proof of the claim is now complete.
\end{proof}

Claim~\ref{claim:d-extensions} implies a lower bound on the ratio of the volumes of $\M_{n+1}$ and $\M_n^*$. More precisely, for each $d \in \M_n^*$, the volume of extensions of $d$ to a $d' \in \M_{n+1}$ is at least
\[
  \prod_{i=1}^n \min\{\delta_1(i),\delta_2(i),1-2n^{-\rho}\}\ge
  \left(n^{-\rho}\right)^{|V|} \cdot \prod_{i\in Q}\min_{j\ne
    i}d_{ij} \cdot (1-2n^{-\rho})^n.
\]
By the definition of $F_{Q}(d)$, see~\eqref{eq:F_A_def},
\[
  \prod_{i \in Q} \min_{j \neq i} d_{ij} = 2^{-|Q|} \cdot F_{Q}(d),
\]
while the definition of $Q$ ensures that
  \begin{equation*}
    F_{Q}(d) \le \left(1 - n^{-\rho}\right)^{|Q|} \le \exp\big(-|Q| \cdot n^{-\rho}\big) \le \left(2^{-|Q|}\right)^{n^{-\rho}}.
  \end{equation*}
  Since the fact that $d \in \M_n^1$ gives $F_Q(d) > \exp(-n^{1-\varphi})$, we deduce that
  \[
    2^{-|Q|} \cdot F_{Q}(d) \ge F_{Q}(d)^{n^\rho+1} > \exp\left(-2n^{1-\varphi+\rho}\right).
  \]
  Finally, as $d \in \M_n^2$, we have $|V| \le 2n^{1-\mu}$ and we conclude that
  \begin{align*}
    \vol(\M_{n+1}) &\ge \left(n^{-\rho}\right)^{2n^{1-\mu}} \cdot \exp\big(-2n^{1 - \varphi + \rho}\big) \cdot \left(1-2n^{-\rho}\right)^n \cdot \vol(\M_n^*)\\
    &\ge \exp\left(-n^{1-3c} \log (n) \right) \cdot \vol(\M_n),
  \end{align*}
  where the last inequality follows from~\eqref{eq:varphi_def}, \eqref{eq:M_n_*_volume}, and our assumption that $n$ is sufficiently large.
\end{proof}

\section{Other approaches to estimating the
volume}\label{sec:other_approaches}

A \emph{coloured graph} on vertex set $V$ with \emph{palette} $C$ is simply a function in $C^{\binom{V}{2}}$.
Recall that a \emph{hereditary property} is a family of coloured graphs that is closed under taking subgraphs and isomorphisms.
Questions about the asymptotic growth rate of the volume and the distribution of the edge lengths for a random point in the metric polytope can be viewed as instances of the following very general class of problems:
Describe the distribution of a `uniformly sampled' \emph{coloured graph} (more generally, \emph{coloured hypergraph}) on $n$ vertices, conditioned to satisfy a given \emph{hereditary property}~$\mathcal{P}$, when $n$ is large.
The relevance to our setting is the following: We let $V = \br{n}$ and consider the collection of functions $d: \binom{V}{2} \to (0,2]$ satisfying the hereditary property:
\[
  d_{ik} \le d_{ij}+ d_{jk} \quad \text{for all $i,j,k \in V$}.
\]
A few related approaches have been used to study problems from this class:
exchangeable families of random variables,
Szemer\'edi's regularity lemma, graph limits, and the method of hypergraph containers.
We refer to the survey paper \cite{MR2426176} and references within for a discussion of the connection between exchangeability, the regularity lemma, and graph limits;
for an introduction to the method of hypergraph containers, the reader is referred to the survey paper~\cite{BalMorSam-ICM}.
In this section, we discuss the problem of estimating the volume of the metric polytope using some of these approaches. These approaches may also be used to obtain some structural information on typical samples from the metric polytope.

\begin{table}
\begin{centering}
  \begin{tabular}{| l | c |}
  \hline
    Method & Upper bound on $\log \vol(\M_n)$\\ \hline\hline
    Main result (Theorem~\ref{thm:volume_estimate}) & $O(n^{3/2})$\\ \hline
    Exchangeability & $o(n^2)$ \\ \hline
    Szemer\'edi regularity lemma & $o(n^2)$ \\ \hline
    Hypergraph container method & $O\big(n^{3/2} (\log n)^3\big)$ \\ \hline
    \texorpdfstring{K\H{o}v\'ari}{K\"ov\'ari}--S\'os--Tur\'an & $O\big(\frac{n^2 (\log\log n)^2}{\log n}\big)$\\ \hline
  \end{tabular}
\end{centering}
\smallskip
  \caption{The upper bounds on the volume of the metric polytope provided by our main result and by the alternative approaches presented in Section~\ref{sec:other_approaches}.}
\end{table}

\subsection{Limiting model and exchangeability}

The purpose of this section is to give a `soft' proof of a qualitative version of our main result on the volume. Precisely, we shall show that
\begin{equation}\label{eq:volume exponent}
  \log\vol(\M_n)=o(n^2),
\end{equation}
as in~\eqref{eq:limit constant is one}. The presented proof relies on \emph{exchangeability}.

To motivate the proof method, let us start by recalling de Finetti's theorem \cite{dF59}. It states that the distribution of an exchangeable sequence of random variables is a mixture of distributions of i.i.d.\ sequences of random variables. Here, we recall that: (i) a sequence $(X_n)$ of random variables is called \emph{exchangeable} if, for every finitely-supported permutation~$\sigma$, the sequence $(X_{\sigma(n)})$ has the same joint distribution as the sequence $(X_n)$; (ii) the distribution of a sequence is a \emph{mixture} of distributions of i.i.d.\ random variables if it can be sampled by first randomly sampling a distribution $D$ and then sampling the variables of the sequence independently from the distribution $D$.

De Finetti's theorem implies the following \emph{conditional independence} property: If $(X_n)$ is exchangeable, then, for each $n_0$, after conditioning on $\{X_n : n > n_0\}$ the random variables $X_1,\dotsc,X_{n_0}$ become independent and identically distributed. Indeed, the conditioning determines which distribution $D$ is used in the underlying i.i.d.\ sequence and independence follows.

As metric spaces $(d_{ij})$ are indexed by unordered pairs $\{i,j\}$, their relevant context is not that of exchangeable sequences but rather that of \emph{exchangeable arrays}. An exchangeable array is a two-dimensional array of random variables $(X_{ij})$, with the index set being all unordered pairs of distinct positive integers, such that, for each finitely-supported permutation~$\sigma$, the array $(X_{\sigma(i)\sigma(j)})$ has the same distribution as the array $(X_{ij})$. (The names \emph{weak exchangeability} and \emph{partial exchangeability} are also used for notions of this type. Higher-dimensional versions and variants where different permutations are applied to the coordinates have also been discussed in the literature.) A representation theorem similar to, but more complicated than, de Finetti's theorem exists for exchangeable arrays; see~\cite{aldous_rep1981,hoover_exchange1983,kallenberg_rep_1989} and especially \cite[Theorem 14.21]{ald85}. It again implies a conditional independence property, stated as follows.

\begin{lem}\label{lem:cond_independence}
  Let $(X_{ij})$ be an exchangeable array. For each integer $n_0\ge 1$, conditioned on $\{X_{ij} :\max\{i,j\} > n_0\}$ the random variables $\{X_{ij} : i,j\le n_0\}$ become independent.
\end{lem}

We note that, unlike de Finneti's theorem, the random variables $\{X_{ij} : i,j\le n_0\}$ are not necessarily identically distributed after the conditioning. For completeness, we provide a short proof.

\begin{proof}[Proof of Lemma~\ref{lem:cond_independence}]
The proof is by induction on $n_0$. If $n_0 \in \{1, 2\}$, then the assertion of the lemma is vacuously true (as the set $\{X_{ij} : i, j \le n_0\}$ is either empty or contains only one variable).  Suppose then that $n_0 \ge 3$ and that the result has already been established for $n_0 - 1$.

Write $\F_{n}^N$ and $\F_n$ for the sigma algebras generated by the collections $\{X_{ij} : n\le \max\{i,j\}\le N\}$ and $\{X_{ij} : \max\{i,j\}\ge n\}$, respectively. Let $A\subseteq\R^{\binom{\br{n_0-1}}{2}}$ be a Borel set. Levy's upward theorem (a consequence of the martingale convergence theorem) shows that
\begin{align}
  &\P\big((X_{ij})_{i,j\le n_0-1}\in A\mid\F_{n_0}^N\big) \to \P\big((X_{ij})_{i,j\le n_0-1}\in A\mid\F_{n_0}\big),\label{eq:Levy upward theorem}\\
  &\P\big((X_{ij})_{i,j\le n_0-1}\in A\mid\F_{n_0+1}^{N+1}\big) \to \P\big((X_{ij})_{i,j\le n_0-1}\in A\mid\F_{n_0+1}\big),\nonumber
\end{align}
as $N\to\infty$, almost surely. In addition, the fact that $(X_{ij})$ is an exchangeable array implies that, for each $N\ge n_0+1$,
\begin{equation*}
  \P\big((X_{ij})_{i,j\le n_0-1}\in A\mid\F_{n_0}^N\big) \eqd \P\big((X_{ij})_{i,j\le n_0-1}\in A\mid\F_{n_0+1}^{N+1}\big).
\end{equation*}
Consequently,
\begin{equation*}
  \P\big((X_{ij})_{i,j\le n_0-1}\in A\mid\F_{n_0}\big)\eqd \P\big((X_{ij})_{i,j\le n_0-1}\in A\mid\F_{n_0+1}\big),
\end{equation*}
which implies that, in fact,
\begin{equation}\label{eq:conditional equality almost surely}
  \P\big((X_{ij})_{i,j\le n_0-1}\in A\mid\F_{n_0}\big)= \P\big((X_{ij})_{i,j\le n_0-1}\in A\mid\F_{n_0+1}\big)
\end{equation}
almost surely. To see the last conclusion, observe that if $X$ is a random variable with finite second moment and $\mathcal{G}_1\subseteq\mathcal{G}_2$ are sigma algebras, then
\begin{equation*}
  \E\left[\E[X\mid\mathcal{G}_2]^2\right]=\E\left[\E[X\mid\mathcal{G}_1]^2\right] + \E\left[(\E[X\mid\mathcal{G}_2] - \E[X\mid\mathcal{G}_1])^2\right].
\end{equation*}
Thus, if $\E[X\mid\mathcal{G}_1] \eqd \E[X\mid\mathcal{G}_2]$, then $\E[X\mid\mathcal{G}_1]= \E[X\mid\mathcal{G}_2]$ almost surely.

As~\eqref{eq:conditional equality almost surely} holds for arbitrary Borel $A$, we conclude (recalling the definition of $\F_n$) that, conditioned on $\F_{n_0+1}$, the collection of random variables $\{X_{ij} : i,j\le n_0-1\}$ is independent of the collection $\{X_{ij} : \max\{i,j\} = n_0\}$. Together with the induction hypothesis this implies that, conditioned on $\F_{n_0+1}$, the random variables $\{X_{ij} : i,j\le n_0-1\}$ become independent. These facts together with another use of the exchangeability property imply the lemma.
\end{proof}

We proceed to discuss the metric polytope, aiming to prove~\eqref{eq:volume exponent}. It is convenient to pass to a discrete problem, to avoid questions on the existence of densities and convergence issues. Specifically, given integers $M$ and $n$, define the discrete metric polytope $\M_n^M$ by
\begin{equation*}
  \M_n^M := \left\{ (\dist{i}{j})\in\{1, \dotsc, M\}^{\binom{\br{n}}{2}} : \dist{i}{j}\le \dist{i}{k} + \dist{k}{j}\text{ for all $i,j,k$} \right\},
\end{equation*}
see also Section~\ref{sec:discrete-problem}. We shall prove that, for each fixed \emph{even} $M$,
\begin{equation}
  \label{eq:exchangeability-lemma-claim}
  \limsup_{n\to\infty}\frac{\log(|\M_n^M|)}{\binom{n}{2}} \le \log\left(\frac{M+2}{2}\right).
\end{equation}
As $\vol(\M_n) \le \left(\frac{2}{M}\right)^{\binom{n}{2}}|\M_n^M|$ for all $n,M$, see~\eqref{eq:card_MnM-vol_Mn} in Section~\ref{sec:discrete-problem} below, \eqref{eq:exchangeability-lemma-claim} will imply~\eqref{eq:volume exponent}.

Fix an even $M$. To apply Lemma~\ref{lem:cond_independence}, embed $\M_n^M$ into $\br{M}^{\binom{\mathbb{N}}{2}}$ by setting all distances involving points $i>n$ to zero. Denote by $\mu_n^M$ the uniform distribution on $\M_n^M$, viewed as a distribution on the space $\br{M}^{\binom{\mathbb{N}}{2}}$ via this embedding. As the set of probability measures on this space is compact with respect to convergence in distribution, there exists a subsequence $n_m$ on which the limit superior in~\eqref{eq:exchangeability-lemma-claim} is realized and such that $\mu_{n_m}^M$ converges in distribution. Denote the limit measure by $\mu_\infty^M$ and note that it is necessarily supported on the infinite-dimensional discrete metric polytope
\begin{equation*}
\M_\infty^M:=\left\{ (d_{ij}) \in \{1, \dotsc, M\}^{\binom{\mathbb{N}}{2}}  : d_{ij}\le d_{ik}+d_{kj}\text{ for all
$i,j,k$} \right\}.
\end{equation*}
Write $d^\infty = (d_{ij}^\infty)$ for a sample from $\mu_\infty^M$. Note that $d^\infty$ is an exchangeable array, inheriting its exchangeability properties from the measures $(\mu_n^M)$. As before, we write $\F_{n}^N$ and $\F_n$ for the sigma algebras generated by the collections $\{d_{ij}^\infty : n\le \max\{i,j\}\le N\}$ and $\{d_{ij}^\infty : \max\{i,j\}\ge n\}$, respectively.
Lemma~\ref{lem:cond_independence} shows that, conditioned on $\F_4$, the random variables $(d_{12}^\infty,d_{13}^\infty,d_{23}^\infty)$ become independent. Thus the support of their (conditional) joint distribution is in some axis-parallel discrete box fully contained in $\M_3^M$. An analogue of Lemma~\ref{lem:independent_max_volume} (with an analogous proof) shows that such a box has cardinality at most $\left(\frac{M+2}{2}\right)^3$. In particular,
\begin{equation}\label{eq:limiting cond entropy}
  H_{\text{S}}(d_{12}^\infty,d_{13}^\infty,d_{23}^\infty\mid \F_4)\le 3\log\left(\frac{M+2}{2}\right),
\end{equation}
where $H_{\text{S}}$ denotes Shannon's entropy. Recalling our use of Levy's upward theorem in~\eqref{eq:Levy upward theorem}, and noting that $(d_{12}^\infty, d_{13}^\infty, d_{23}^\infty)$ is supported on a finite set, we see that the conditional distribution of these random variables given $\F_4^N$ converges as $N\to\infty$ to their conditional distribution given $\F_4$, almost surely. In particular (again, by the finite support),
\begin{equation}\label{eq:towards the limiting cond entropy}
  \lim_{N\to\infty}H_{\text{S}}(d_{12}^\infty,d_{13}^\infty,d_{23}^\infty\mid \F_4^N) = H_{\text{S}}(d_{12}^\infty,d_{13}^\infty,d_{23}^\infty\mid \F_4).
\end{equation}
Let $\eps>0$. Combining~\eqref{eq:limiting cond entropy} and~\eqref{eq:towards the limiting cond entropy} shows that, for some $N_0$,
\begin{equation*}
  H_{\text{S}}(d_{12}^\infty,d_{13}^\infty,d_{23}^\infty\mid \F_4^{N_0})\le 3\log\left(\frac{M+2}{2}\right)+\eps.
\end{equation*}
Let $d^n = (d_{ij}^n)$ be a sample from $\mu_n^M$. Similar to the above, the fact that $\mu_{n_m}^M\to\mu_\infty^M$ and $\{d_{ij}^n : i,j\le N_0\}$ is finitely-supported implies that
\begin{equation*}
  H_{\text{S}}\left(d_{12}^{n_m},d_{13}^{n_m},d_{23}^{n_m}| \big\{d_{ij}^{n_m} : 4\le \max\{i,j\}\le N_0\big\}\right)\le 3\log\left(\frac{M+2}{2}\right)+2\eps
\end{equation*}
for all large $m$. By symmetry and monotonicity of conditional entropy, we conclude that, for all large $m$ and all distinct $i, j, k \in \{1, \dotsc, n_m - N_0+4\}$,
\begin{multline*}
  H_{\text{S}}\left(d_{ij}^{n_m},d_{ik}^{n_m},d_{jk}^{n_m}\mid \big\{d_{ij}^{n_m} : n_m-N_0+4\le \max\{i,j\}\le n_m\big\}\right)\\
  \le 3\log\left(\frac{M+2}{2}\right)+2\eps.
\end{multline*}
We may now apply the subadditivity argument from the proof outline, Section~\ref{sec:proof outline}, to obtain that, for all large $m$,
\begin{equation*}
  \log(|\M_{n_m}^M|)\le C \log(M) N_0 n_m + \left(\log\left(\frac{M+2}{2}\right)+\frac{2}{3}\eps\right)\cdot\binom{n_m}{2}
\end{equation*}
for an absolute constant $C$. Finally, recalling that the limit superior in~\eqref{eq:exchangeability-lemma-claim} is realized along $n_m$, and noting that $\eps$ is arbitrary and $N_0$ is a function only of $\eps$ and $\mu_\infty^M$, we conclude that~\eqref{eq:exchangeability-lemma-claim} holds.

\subsection{The Szemer\'edi regularity lemma approach}

In this section, we show how a fairly standard application of (a multi-coloured version of) Szemer\'edi's regularity lemma gives an alternative proof of~\eqref{eq:limit constant is one}. The argument presented here may be seen as an adaptation of the classical argument of Erd\H{o}s, Frankl, and R\"odl~\cite{ErFrRo86} proving that the number of $H$-free graphs with $n$ vertices is $2^{\exH + o(n^2)}$, where $\exH$ denotes the maximum number of edges in an $H$-free graph with $n$~vertices. This approach was independently pursued by Mubayi and Terry~\cite{mubayi2019discrete}.

Recall that a bipartite graph $G$ with parts $V_1$ and $V_2$ is \emph{$\eps$-regular} if, for every $W_1 \subseteq V_1$ with $|W_1| \ge \eps |V_1|$ and $W_2 \subseteq V_2$ with $|W_2| \ge \eps |V_2|$, we have
\[
  \left| \frac{e_G(W_1, W_2)}{|W_1| |W_2|} - \frac{e_G(V_1, V_2)}{|V_1| |V_2|} \right|  \le \eps,
\]
where $e_G(W_1,W_2)$ is the number of edges connecting a vertex of $W_1$ to a vertex of $W_2$. An~\emph{equipartition} of a set $V$ is a partition of $V$ into $V_1, \dotsc, V_k$ such that $\big| |V_i| - |V_j| \big| \le 1$ for all $i$ and $j$. The celebrated regularity lemma of Szemer\'edi~\cite{Sz78} states that, for every positive $\eps$, there exists a constant $R$ such that the vertex set of every graph $G$ admits an equipartition into at most $R$ parts with the property that the bipartite subgraphs of $G$ induced by all but at most an $\eps$-proportion of all pairs of parts are $\eps$-regular. We shall be needing the following straightforward generalisation of this statement to edge-coloured graphs. For the remainder of this section, given a positive integer $M$, we shall refer to a colouring of all pairs of elements of a set $V$ with elements of $\br{M}$ as an \emph{$M$-graph} with vertex set $V$. Moreover, given an $M$-graph $G$ and a $c \in \br{M}$, we shall denote by $G(c)$ the graph whose edges are all pairs of vertices to which $G$ assigns the colour $c$. The following straightforward generalisation of Szemer\'edi's regularity lemma to $M$-graphs was formulated in~\cite{AxMa11}. It may be easily deduced from the standard proof of the regularity lemma.

\begin{thm}[{\cite{AxMa11}}]
  \label{thm:colored-reg-lemma}
  For every $\eps > 0$, $M$, and $r_0$, there exists an integer $R$ with the following property.
  The vertex set of an arbitrary $M$-graph $G$ admits an equipartition $\{V_1, \dotsc, V_r\}$, where $r_0 \le r \le R$, such that, for all but at most $\eps \binom{r}{2}$ pairs $\{i,j\} \in \binom{\br{r}}{2}$, the bipartite subgraph of $G(c)$ induced by $V_i$ and $V_j$ is $\eps$-regular for every $c \in \br{M}$.
\end{thm}

For the sake of brevity, we shall refer to partitions satisfying the assertion of the theorem as \emph{$\eps$-regular partitions}. As in most standard applications of the regularity lemma, we shall use the following straightforward property of $\eps$-regular graphs, the \emph{embedding lemma} for triangles. For a more general version of the embedding lemma, we refer the reader to the classical survey of Koml\'os and Simonovits~\cite{KoSi96}.

\begin{prop}
  \label{prop:triangle-embedding}
  Let $\eps \in (0,1/2)$, suppose that $V_1$, $V_2$, and $V_3$ are
  pairwise disjoint sets, and let $G$ be a graph with vertex set $V_1
  \cup V_2 \cup V_3$. If, for each pair $\{i,j\} \in \binom{\br{3}}{2}$,
  the bipartite subgraph of $G$ induced by $V_i$ and $V_j$ is
  $\eps$-regular and satisfies $e_G(V_i, V_j) \ge 2\eps|V_i||V_j|$,
  then $G$ contains a triangle.
\end{prop}

As in the previous section, given integers $M$ and $n$, we define
\[
\M_n^M := \left\{(\dist{i}{j})\in\{1, \dotsc, M\}^{\binom{\br{n}}{2}} : \dist{i}{j}\le \dist{i}{k} + \dist{k}{j}\text{ for all $i,j,k$}\right\},
\]
see also Section~\ref{sec:discrete-problem} below. We shall prove that
\begin{equation}
  \label{eq:reg-lemma-claim}
  |\M_n^M| \le M^{\delta n^2} \cdot \left(\frac{M+2}{2}\right)^{\binom{n}{2}}
\end{equation}
for each fixed even $M$ and $\delta > 0$, provided that $n$ is sufficiently large. We remark here that Mubayi and Terry~\cite{mubayi2019discrete} independently used a similar approach, combined with a delicate stability analysis, to prove the much more accurate estimate $|\M_n^M| = (1+e^{-\Omega(n)}) \left(\frac{M+2}{2}\right)^{\binom{n}{2}}$ for each fixed even $M$. As $\vol(\M_n) \le \left(\frac{2}{M}\right)^{\binom{n}{2}}|\M_n^M|$, see~\eqref{eq:card_MnM-vol_Mn} in Section~\ref{sec:discrete-problem} below, \eqref{eq:reg-lemma-claim} will imply that $\vol(\M_n) = 2^{o(n^2)}$.

As proofs of both~\eqref{eq:reg-lemma-claim} and the improved estimate of~\cite{mubayi2019discrete} rely on the regularity lemma, the rate of convergence implicit in the $o(n^2)$ term in the exponent is very slow. Possibly, one could use weaker forms of the regularity lemma to improve this rate of convergence. We do not pursue this direction here, but only mention that one such regularity lemma, guaranteeing a regular partition whose number of parts is only exponential in $\eps^{-2}$, was obtained by Frieze and Kannan~\cite{FriKan96,FriKan99}, see also~\cite[Section~1.4]{MR2989432}. (In our context, a multi-coloured version of such a regularity lemma would most likely have been required.)

Fix an even integer $M$ and $\delta \in (0,1/2)$ and let $\eps = \frac{\delta}{10M\log(1/\delta)}$ and $r_0 = 2/\delta$. Choose an arbitrary $G \in \M_n^M$, which may be viewed as an $M$-graph with vertex set $\br{n}$, and apply Theorem~\ref{thm:colored-reg-lemma} to $G$ to obtain an $\eps$-regular partition $\{V_1, \dotsc, V_r\}$ of $\br{n}$ with $r_0 \le r \le R$ for some constant $R = R(M, \delta)$. For every pair $\{i,j\} \in \binom{\br{r}}{2}$, define
\[
D_{ij} = \left\{ c \in \br{M} : e_{G(c)}(V_i, V_j) \ge 2\eps |V_i||V_j| \right\}
\]
and observe that all but at most a $2M\eps$-proportion of pairs in $V_i \times V_j$ are coloured with an element of $D_{ij}$. Call a triple $\{i,j,k\} \in \binom{\br{r}}{3}$ \emph{regular} if the bipartite subgraphs of $G(1), \dotsc, G(M)$ induced by $(V_i,V_j)$, $(V_i, V_k)$, and $(V_j, V_k)$ are all $\eps$-regular. It follows from Proposition~\ref{prop:triangle-embedding} that, for every regular triple $\{i, j, k\}$, we must have $D_{ij} \times D_{ik} \times D_{jk} \subseteq \M_3^M$. Indeed, otherwise $G$ would contain a triple of distances that do not satisfy the triangle inequality. A discrete analogue of Lemma~\ref{lem:independent_max_volume} (with an analogous proof) shows that $D_{ij} \times D_{ik} \times D_{jk}$ has cardinality at most $\left(\frac{M+2}{2}\right)^3$. As $\{V_1, \dotsc, V_r\}$ is an $\eps$-regular partition of $G$, all but at most $3\eps\binom{r}{3}$ triples $\{i,j,k\} \in \binom{\br{r}}{3}$ are regular. Consequently,
\begin{multline}
  \label{eq:prod-Dij}
  \prod_{\{i,j\} \in \binom{\br{r}}{2}} |D_{ij}| = \left( \prod_{\smash{\{i,j,k\} \in \binom{\br{r}}{3}}} |D_{ij}||D_{ik}||D_{jk}|\right)^{\frac{1}{r-2}} \displaybreak[0]\\
  \le \left( \left(\frac{M+2}{2}\right)^{3(1-3\eps)\binom{r}{3}} M^{9\eps\binom{r}{3}} \right)^{\frac{1}{r-2}} \le \left(\frac{M+2}{2^{1-3\eps}}\right)^{\binom{r}{2}}.
\end{multline}

Since $G$ was arbitrary, the above analysis shows that one may construct each element of $\M_n^M$ as follows. First, choose $r$, the equipartition $\{V_1,\dotsc,V_r\}$, and the sets $D_{ij}\subseteq \br{M}$; the number of choices for all three combined is $2^{O(n)}$ (with implicit constant depending on $M$ and $\delta$). Next, for each $\{i,j\} \in \binom{\br{r}}{2}$, choose a set $X_{ij} \subseteq V_i \times V_j$ of at most $2M\eps |V_i| |V_j|$ pairs whose colour will not belong to $D_{ij}$; there are at most $\Big(\genfrac{}{}{0pt}{}{\binom{n}{2}}{\lfloor 2M\eps \binom{n}{2}\rfloor}\Big) \le \exp\big(M\eps \log(e/(2M\eps)) n^2\big)$ ways to do it. Finally, choose colours for all $\binom{n}{2}$ pairs in such a way that each pair in $V_i \times V_j \setminus X_{ij}$ is assigned a colour from $D_{ij}$; the number of ways one can do this is
\begin{equation}
  \label{eq:SzRL-number-of-choices}
  \prod_{i,j}|D_{ij}|^{|V_i\times V_j\setminus X_{ij}|} \cdot M^{\binom{n}{2} - \sum_{i,j} |V_i \times V_j \setminus X_{ij}|}.
\end{equation}
Recalling that $|V_i \times V_j| \ge \lfloor n/r \rfloor^2$ and $|X_{ij}| \le 2M\eps |V_i||V_j|$ for each $\{i,j\} \in \binom{\br{r}}{2}$, inequality~\eqref{eq:prod-Dij} implies that~\eqref{eq:SzRL-number-of-choices} is at most
\begin{equation}
  \label{eq:SzRL-bound}
  \left(\frac{M+2}{2^{1-3\eps}}\right)^{\binom{r}{2} \cdot (1-2M\eps)\lfloor n/r \rfloor^2} \cdot M^{\binom{n}{2} - \binom{r}{2} \cdot (1-2M\eps)\lfloor n/r \rfloor^2}.
\end{equation}
Finally, as $\binom{n}{2} - \binom{r}{2} \lfloor n/r \rfloor^2 \le \frac{1}{r} \binom{n}{2} + r(n-1)$ and $r_0 \le r \le R$, a straightforward calculation shows that~\eqref{eq:SzRL-bound} is at most
\[
\left(\frac{M+2}{2}\right)^{\binom{n}{2}} 2^{\left(\frac{1}{r_0}+2M\eps+3\eps+\frac{2R}{n}\right)\binom{n}{2}}.
\]
This yields the claimed upper bound on $|\M_n^M|$ stated in~\eqref{eq:reg-lemma-claim}, provided that $n$ is sufficiently large, by our choice of $\eps = \eps(M, \delta)$ and~$r_0 = r_0(\delta)$.

\subsection{The hypergraph container method}
\label{sec:container-method}

In this section, which is based on joint work with Rob Morris, we shall show how the method of hypergraph containers can be used to derive a volume estimate of the form
\begin{equation}
  \label{eq:vol-Mn-container}
  \vol(\M_n) \le \exp \left( C n^{3/2} (\log n)^3 \right),
\end{equation}
which falls just a little short of the upper bound established in Theorem~\ref{thm:volume_estimate} using entropy methods. We point out that the arguments presented here are inspired by the (earlier) work of Balogh and Wagner~\cite{BalWag16}, who were the first to use the container method for enumerating finite metric spaces and obtained the bound $\vol(\M_n) \le \exp(n^{11/6+o(1)})$.

The hypergraph container theorems, proved simultaneously, but separately, in~\cite{BalMorSam} and~\cite{SaxTho}, state that the family of independent sets of any uniform hypergraph whose edges are sufficiently evenly distributed can be covered by a small family of \emph{containers}, subsets of vertices of the hypergraph that themselves are nearly independent. The wide applicability of this abstract statement stems from the fact that many discrete structures may be naturally represented as independent sets of some auxiliary hypergraph; in particular, this is the case with the metric spaces in $\M_n^M$. The particular version of the hypergraph container theorem stated below was proved in~\cite{MorSamSax}; see also~\cite{BalMorSam-ICM} for a survey.

Suppose that $\HH$ is a $k$-uniform hypergraph, i.e.\ each (hyper)edge has exactly $k$ vertices. We write $V(\HH)$ to denote the vertex set of $\HH$ and we identify $\HH$ with its (hyper)edge set; we denote by $v(\HH)$ and $e(\HH)$ the numbers of vertices and edges of $\HH$, respectively. A set $I\subseteq V(\HH)$ is called \emph{independent} if it contains no edges of $\HH$. We moreover define, for every $\ell \in \{1, \dotsc, k\}$,
\[
  \Delta_\ell(\HH) = \max\left\{ |\{ S \in \HH \colon T \subseteq S \}| \colon T \in \binom{V(\HH)}{\ell} \right\}.
\]
In other words, $\Delta_\ell(\HH)$ is the maximum number of edges of $\HH$ that a single $\ell$-element set of vertices can be contained in.

We say that a family $\C$ of subsets of $V(\HH)$ is a family of \emph{containers} for (the independent sets of) $\HH$ if every independent set is contained in some $B \in \C$.  Every hypergraph $\HH$ admits two trivial families of containers: the one-element family $\{V(\HH)\}$ and the family of all (maximal) independent sets of $\HH$.  The following proposition guarantees the existence of a family of containers that interpolates between these two extremes: it is much smaller than the family of all independent sets but each of the containers is significantly smaller than $V(\HH)$.

\begin{prop}
  \label{prop:containers-main}
  Let $\HH$ be a~non-empty $k$-uniform hypergraph. Suppose that positive integers $b$ and $r$ satisfy
  \[
    \Delta_\ell(\HH) \le \left(\frac{b}{v(\HH)}\right)^{\ell-1} \frac{e(\HH)}{r}
  \]
  for every $\ell \in \{1, \dotsc, k\}$. Then there exists a collection $\C$ of at most $\exp\big(kb\log(v(\HH))\big)$ subsets of $V(\HH)$ such that:
  \begin{enumerate}[label=(\roman*)]
  \item
    every independent set of $\HH$ is contained in some $B \in \C$;
  \item
    $|B| \le v(\HH) - 2^{-k(k+1)} \cdot r$ for every $B \in \C$.
  \end{enumerate}
\end{prop}

In a typical application of the proposition, such as the one presented in this section, one takes $r$ to be close to $v(\HH)$ while $b = v(\HH)^\alpha$ for some $\alpha \in (0,1)$.

Call a triple $(a,b,c)$ of numbers \emph{non-metric} if some permutation of $(a,b,c)$ does not satisfy the triangle inequality, that is, if $a+b < c$, $a+c < b$, or $b+c < a$. Given positive integers $n$ and $M$, define the hypergraph $\HH_n^M$ of \emph{non-metric triangles} as follows. The vertex set of $\HH_n^M$ is $\binom{\br{n}}{2} \times \br{M}$ and its edges are all triples $\{(e_i,d_i)\}_{i=1}^3$ such that
\begin{itemize}
\item
  $\{e_1, e_2, e_3\}$ is the set of edges of some triangle in the complete graph on $\br{n}$,
\item
  $(d_1, d_2, d_3)$ is a non-metric triple.
\end{itemize}
It is not hard to see that the elements of $\M_n^M$ are in a one-to-one correspondence with independent subsets of $\HH_n^M$ that contain exactly one element of the set $\{e\} \times \br{M}$ for each $e \in \binom{\br{n}}{2}$.

Now, given a set $A \subseteq \binom{\br{n}}{2} \times \br{M}$, define, for each $e \in \binom{\br{n}}{2}$,
\[
  A_e := \{d \in \br{M} \colon (e, d) \in A\}.
\]
Viewing $A$ as a representation of the product set $\prod_e A_e$, we define its volume by
\[
  \vol(A) := \prod_{e \in \binom{\br{n}}{2}} |A_e|,
\]
which is precisely the number of sets $I \subseteq A$ that contain exactly one element of the set $\{e\} \times \br{M}$ for each $e \in \binom{\br{n}}{2}$.

The following supersaturation statement for $\HH_n^M$ is the key ingredient in our application of the container method to the setting of discrete metric spaces.

\begin{prop}
  \label{prop:supersaturation-global}
  Let $n$ and $M$ be positive integers, with $M$ \emph{even} and $n \ge 3$. Suppose that  $A \subseteq \binom{\br{n}}{2} \times \br{M}$ satisfies
  \[
  \vol(A) \ge \left(\frac{(1+\eps)M}{2}\right)^{\binom{n}{2}}
  \]
  for some $\eps \ge 16/M$. Then there exist an $m \in \br{M}$ and a set $A' \subseteq A$ with $|A'| \le mn^2$ such that
  \begin{itemize}
  \item
    $e(\HH_n^M[A']) \ge \eps m^2 M \binom{n}{3} / (32 \log_2 M)$,
  \item
    $\Delta_1(\HH_n^M[A']) \le 4nm^2$,
  \item
    $\Delta_2(\HH_n^M[A']) \le 2m$,
  \end{itemize}
  where $\HH[B]$ denotes the subhypergraph of $\HH$ induced by the subset $B$, that is, the hypergraph whose vertex set is $B$ and whose edges are all edges of $\HH$ that are fully contained in $B$.
\end{prop}

The basic building block in the proof of Proposition~\ref{prop:supersaturation-global} is the following elementary lemma, which one can prove combining the ideas in the proofs of Lemmas~\ref{lem:independent-box-poly} and~\ref{lem:independent_max_volume}.

\begin{lem}
  \label{lem:supersaturation-local}
  Let $M$ and $m$ be positive integers, with $M \ge 16$ even, and suppose that $A, B, C \subseteq \br{M}$. Let $A' \subseteq A$ comprise the $m$ largest and the $m$ smallest elements of $A$ and define $B'$ and $C'$ analogously. If $|A| \cdot |B| \cdot |C| \ge (M/2+2m)^3$, then the set $A' \times B' \times C'$ contains $m^3$ non-metric triples.
\end{lem}

\newcommand{\smax}{s_{\max}}

\begin{proof}[Proof of Proposition~\ref{prop:supersaturation-global}]
  As $\vol(A) \le M^{\binom{n}{2}}$, we may assume that $\eps \le 1$ and hence $M \ge 16$. Let $\T$ be the family of edge sets of all triangles in the complete graph with vertex set $\br{n}$. Since each edge (of the complete graph) belongs to exactly $n-2$ triangles,
  \begin{equation}
    \label{eq:prod-T-vol-A}
    \prod_{e_1e_2e_3 \in \T} \left( |A_{e_1}| |A_{e_2}| |A_{e_3}|\right)^{1/3} = \vol(A)^{\frac{n-2}{3}} \ge \left(\frac{(1+\eps) M}{2}\right)^{\binom{n}{3}}.
  \end{equation}
  We partition the family $\T$ as follows. Set $\smax = \lfloor \log_2 M \rfloor - 2$ and, for each $s \in \{0, \ldots, \smax\}$, define
  \[
  \T_s := \left\{e_1e_2e_3 \in \T \colon \left( |A_{e_1}| |A_{e_2}| |A_{e_3}|\right)^{1/3} \in \left[\frac{M}{2} + 2^{s+1}, \frac{M}{2} + 2^{s+2} \right) \right\};
  \]
  moreover, let $\T_* := \T \setminus \bigcup_{s=0}^{\smax} \T_s$. Observe that $\T_*$ contains only $e_1e_2e_3$ with $|A_{e_1}||A_{e_2}||A_{e_3}| < (\frac{M}{2}+2)^3$, as $2^{\smax+2} > M/2$, and thus
  \begin{equation}
    \label{eq:prod-Ts-vol-A}
    \prod_{e_1e_2e_3}  \left( |A_{e_1}| |A_{e_2}| |A_{e_3}|\right)^{1/3} \le \left(\frac{M}{2}+2\right)^{|\T_*|} \cdot \prod_{s=0}^{\smax} \left(\frac{M}{2} + 2^{s+2}\right)^{|\T_s|}.
  \end{equation}
  We claim that there is an $s \in \{0, \ldots, \smax\}$ satisfying
  \[
    |\T_s| \ge \frac{\eps M}{2^{s+5} \log_2M} \binom{n}{3}.
  \]
  Indeed, if this were not true, then~\eqref{eq:prod-Ts-vol-A} would contradict~\eqref{eq:prod-T-vol-A}, as $16/M \le \eps \le 1$ and $\smax + 1 \le \log_2 M$ (we omit the straightforward calculation).

  Finally, let $m = 2^s$ and let $A'$ be the set of all pairs $(e, d) \in A$ such that $d$ is among the $m$ largest or the $m$ smallest elements of $A_e$. This definition guarantees that $|A'| \le 2m \binom{n}{2} \le mn^2$, that $\Delta_1(\HH_n^M[A']) \le (2m)^2n$, and that $\Delta_2(\HH_n^M[A']) \le 2m$. For each $e_1e_2e_3 \in \T_s$, we may invoke Lemma~\ref{lem:supersaturation-local} with $(A,B,C) \leftarrow (A_{e_1}, A_{e_2}, A_{e_3})$ to deduce that the set $A'_{e_1} \times  A'_{e_2} \times A'_{e_3}$ contains at least $m^3$ non-metric triples. In particular,
  \[
    e(\HH_n^M[A']) \ge |\T_s| \cdot m^3 \ge \frac{\eps Mm^2}{32 \log_2M} \binom{n}{3},
  \]
  which concludes the proof of the proposition.
\end{proof}

Fix a large integer $n$ and let $M = 2\lfloor \frac{n}{2} \rfloor$. Suppose that $A \subseteq \binom{\br{n}}{2} \times \br{M}$ satisfies $\vol(A) = \left(\frac{(1+\eps)M}{2}\right)^{\binom{n}{2}}$ for some $16/M \le \eps \le 1$ and let $m$ and $A'$ be as in Proposition~\ref{prop:supersaturation-global}. It is straightforward to verify that the ($3$-uniform) hypergraph $\HH_n^M[A']$ satisfies the assumption of Proposition~\ref{prop:containers-main} with
\[
  b := \left\lceil n^{3/2} \right\rceil \qquad \text{and} \qquad r := \left\lfloor \frac{\eps M \binom{n}{3}}{128 n \log_2 M} \right\rfloor \ge \frac{\eps M n^2}{2^{10} \log_2 M}.
\]
The proposition supplies a family $\C'$ of at most $\exp\big(3n^{3/2} \log(n^2M)\big)$ containers for independent sets of $\HH_n^M[A']$, each of cardinality at most $|A'| - \frac{\eps M n^2}{2^{22} \log_2 M}$. Therefore, the collection
\[
  \C := \C(A) := \{ (A \setminus A') \cup B' : B' \in \C'\}
\]
is a family of containers for independent sets of $\HH_n^M[A]$, with the same cardinality as $\C$, that satisfies, for every $B \in \C$,
\[
  \vol(B) \le \left(\frac{M-1}{M}\right)^{\frac{\eps M n^2}{2^{22} \log_2M}} \cdot \vol(A) = \left(\frac{(1+\eps')M}{2}\right)^{\binom{n}{2}},
\]
for some $\eps' \le \left(1 - \frac{1}{2^{22}\log_2M}\right)\eps$.

We build a family $\C$ of containers for the independent sets of $\M_n^M$ recursively as follows. We start with the trivial family containing only the set $\binom{\br{n}}{2} \times \br{M}$. As long as our family contains some set $A$ with
\[
\vol(A) > \left(\frac{(1 + \eps_0) M}{2}\right)^{\binom{n}{2}},
\]
where $\eps_0 := 1/\sqrt{n} \ge 16/M$, we replace $A$ with the elements of the family $\C(A)$ defined above.  We claim that the depth of the recursion is bounded by $t := C \log_2(M) \log(n)$, for some large constant $C$.  Indeed, if a set $B$ reached the $t$-th level of the recursion, then
\[
  \vol(B) \le \left(\frac{\left(1+\eps_t\right)M}{2}\right)^{\binom{n}{2}},
\]
where
\[
  \eps_t = \max\left\{ \left(1-\frac{1}{2^{22}\log_2M}\right)^t, \frac{16}{M} \right\} \le \eps_0,
\]
a contradiction.  It follows that
\[
  |\C| \le \exp\left(3n^{3/2} \log(n^2M) \cdot t\right) \le \exp\left(Cn^{3/2}(\log n)^3\right).
\]
Since each space in $\M_n^M$ corresponds to an independent set of $\HH_n^M$ and is thus described by one of the containers, we obtain
\[
  |\M_n^M| \le \sum_{B \in \C} \vol(B) \le \exp\left(Cn^{3/2}(\log n)^3+ n^{3/2}\right) \cdot \left(\frac{M}{2}\right)^{\binom{n}{2}}.
\]
Finally, this translates to the following upper bound on the volume:
\[
\vol(\M_n) \le \left(\frac{2}{M}\right)^{\binom{n}{2}} \cdot |\M_n^M| \le \exp\left(Cn^{3/2}(\log n)^3\right),
\]
see~\eqref{eq:card_MnM-vol_Mn} in Section~\ref{sec:discrete-problem} below.

\subsection{The \texorpdfstring{K\H{o}v\'ari}{K\"ov\'ari}--S\'os--Tur\'an approach}

In this section, we shall show yet another approach to the volume estimate. The estimate it gives is
\begin{equation}
  \label{eq:vol-Mn-KST}
  \vol(\M_n) \le \exp \left( \frac{C n^2 (\log\log n)^2}{\log n}\right),
\end{equation}
better than what we obtained using the exchangeability or the regularity lemma approaches, but not as good as what is proved by the entropy or the hypergraph container methods. Our argument bears similarities to the classical work of Erd\H{o}s, Kleitman, and Rotschild~\cite{ErKlRo76}, which estimates the number of graphs that do not contain a clique of a given size.

Given a positive integer $t$, we shall write $K_{t,t}$ for the complete bipartite graph with $t$ vertices on each side. The \emph{Tur\'an number} for $K_{t,t}$, denoted $\ext$, is the largest number of edges in an $n$-vertex graph that does not contain $K_{t,t}$ as a (not necessarily induced) subgraph. The following well-known upper bound on $\ext$ was obtained by K\H{o}v\'ari, S\'os, and Tur\'an \cite{KoSoTu54}, see also~\cite[Section~3]{FurSim13}.

\begin{thm}[K\H{o}v\'ari--S\'os--Tur\'an~\cite{KoSoTu54}]
  \label{thm:KST}
  For every $t \ge 2$,
  \[
  \ext \le \frac{1}{2} \left((t-1)^{1/t} n^{2-1/t} + (t-1)n\right).
  \]
\end{thm}

Fix integers $n$ and $t \ge 2$ and a real $\delta \in (0,1)$. For a $d \in \M_n$, let
\[
  T(d) := \left\{\{i, j\} \in \binom{\br{n}}{2} : d_{ij} < 1-\delta \right\}
\]
and partition $\M_n$ into $\Mnd$ and $\Mndc$, where
\[
  \Mnd := \{d \in \M_n : T(d) \nsupseteq K_{t,t}\} \quad \text{and} \quad \Mndc := \M_n \setminus \Mnd.
\]
Since $|T(d)| \le \ext$ for every $d \in \Mnd$, we have
\[
\begin{split}
  \vol(\Mnd) & \le \binom{\binom{n}{2}}{\ext} \cdot 2^{\ext} \cdot (1+\delta)^{\binom{n}{2} - \ext} \\
  & \le \exp\left(3\ext \log n + \delta n^2\right).
\end{split}
\]
It follows from Theorem~\ref{thm:KST} and simple calculus that, if $n \ge t^2 \ge 4$,
\begin{equation}
  \label{eq:KST-vol-Mn1}
  \vol(\Mnd) \le \exp\left( 5n^{2-1/t}\log n + \delta n^2 \right).
\end{equation}
We now derive an upper bound on the volume of $\Mndc$.

\begin{lem}
  \label{lem:KST-vol-Mn2}
  If $t \ge 6$, $n \ge 4t^2$, and $\delta \ge 3 \log(4t)/t$, then
  \[
  \vol(\Mndc) \le e^{-n} \cdot \vol(\M_{n-2t}).
  \]
\end{lem}
\begin{proof}
  Suppose that $d \in \Mndc$. By definition, we may find two disjoint $t$-element sets $I, J \subseteq \br{n}$ such that $d_{ij} < 1-\delta$ for every pair $(i,j) \in I \times J$. Fix any such pair $(I,J)$ and suppose that $k \in \br{n} \setminus (I \cup J)$. Let
  \[
  a_I = \min_{i \in I} d_{ik}, \quad b_I = \max_{i \in I} d_{ik},
  \quad a_J = \min_{j \in J} d_{jk}, \quad b_J = \max_{j \in J} d_{jk}.
  \]
  Since all distances between $I$ and $J$ are shorter than $1- \delta$, both $b_J - a_I$ and $b_I - a_J$ must be smaller than $1-\delta$ and, consequently,
  \[
  (b_I - a_I)(b_J - a_J) \le \left(\frac{(b_I - a_I) + (b_J - a_J)}{2}\right)^2 < (1-\delta)^2.
  \]
  In other words, all distances between $k$ and $I$ and between $k$ and $J$ fall into intervals $A_I$ and $A_J$, respectively, where $|A_I| \cdot |A_J| < (1-\delta)^2$. In particular, if $W$ denotes the set of all $2t$-dimensional vectors $(d_{ik}')_{i \in I \cup J}$ which may be used to complete $(d_e)_{e \in \binom{I \cup J}{2}}$ to a metric space on $I \cup J \cup \{k\}$, then
  \[
  \vol(W) \le t^4 \cdot 2^4 \cdot (1-\delta)^{2t-4},
  \]
  as there are at most $t^4$ choices for the $i,i' \in I$ and $j,j' \in J$ for which $a_I = d_{ik}$, $b_I = d_{i'k}$, $a_J = d_{jk}$, and $b_J = d_{j'k}$. By our assumption on $t$ and $\delta$,
  \[
    \vol(W) \le \left(2te^{-\delta(t-2)/2}\right)^4 \le \left(2te^{-\delta t /3}\right)^4 \le 2^{-4}.
  \]

  We may now bound the volume of $\Mndc$ as follows. First, the number of choices for $I$ and $J$ is at most $\binom{n}{t}^2$ and the volume of the distances between pairs in $I \cup J$ does not exceed $2^{\binom{2t}{2}}$. Next, bounding the volume of $(d_{ik})_{i \in I \cup J, k \notin I \cup J}$ as above and the volume of $(d_{ij})_{i,j \notin I \cup J}$ by $\vol(\M_{n-2t})$, we obtain
  \[
  \begin{split}
    \vol(\Mndc) & \le \binom{n}{t}^2 \cdot 2^{\binom{2t}{2}} \cdot \vol(W)^{n-2t} \cdot \vol(\M_{n-2t}) \\
    & \le n^{2t} \cdot 2^{2t^2} \cdot 2^{-4n + 8t} \cdot \vol(\M_{n-2t}),
  \end{split}
  \]
  which, with our assumption on $n$ and $t$, implies the claimed bound.
\end{proof}

One may now derive~\eqref{eq:vol-Mn-KST} by induction on $n$ using Lemma~\ref{lem:KST-vol-Mn2} and the upper bound on $\vol(\Mnd)$ given by~\eqref{eq:KST-vol-Mn1}. In the inductive step, one may take $t = \log n / (2 \log \log n)$ and $\delta = 3 \log(4t) / t$, say. We leave the details to the reader.

\section{Discussion and open questions}\label{sec:discussion_and_open_questions}

\subsection{Further questions}\label{sec:further_questions}

As we remarked, we were not able to decide whether $\vol(\M_n)$ is increasing in $n$.
If one could prove that this is indeed the case, this would greatly simplify our proof of Theorem~\ref{thm:minimal_distance} on the shortest distance in the metric space sampled uniformly from~$\M_n$.

Suppose that $d$ is a metric space sampled uniformly from $\M_n$. A key ingredient in our proof of Theorem~\ref{thm:minimal_distance} is the upper bound on $\P(d_{12} < 1)$ established in Proposition~\ref{prop:distance-lower-tail}. It would be interesting to obtain additional information about the distribution of $d_{12}$. In particular, is it true that $\P(d_{12} < 1) = \Theta(n^{-1/2})$? We believe that this is the case and our belief seems to be supported by the lower bound of Proposition~\ref{prop:distance-lower-tail}. Going even further and writing $f_n$ for the density of the random variable $d_{12}$, one may ask whether the function $[0, \infty) \ni t \mapsto f_n(1-\frac{t}{\sqrt{n}})$ has a limit as $n \to \infty$? It would also be very interesting to estimate the probability $\P(d_{12} < 1-\frac{t}{\sqrt{n}})$ for $t \gg 1$. Propositions~\ref{prop:exp-bound-alpha} and~\ref{prop:volume_comparison} imply that $\P(d_{12} < \alpha)$ is exponentially small in $n$ for every fixed $\alpha < 1/2$, see also~\eqref{eq:exponential_probability_for_very_small_distances}, but we are not ready to make any conjectures about the range $t \le \sqrt{n}/2$.

Do the empirical measures of individual distances (and tuples of distances)
satisfy a large deviation principle? If so, what is the rate
function? Is it possible to recover our result about the minimum distance from such a large deviation estimate?

\subsection{Relation with the discrete problem}

\label{sec:discrete-problem}

One may naturally consider a discrete analogue of the problem we study in this paper, where we require the distances between every pair of points to be integers. More specifically, given integers $M\ge 1$ and $n\ge 2$, one may consider the space $\M_n^M$ defined by
\begin{equation*}
  \M_n^M := \left\{ (\dist{i}{j})\in\{1, \dotsc, M\}^{\binom{\br{n}}{2}} : \dist{i}{j}\le \dist{i}{k} + \dist{k}{j}\text{ for all $i,j,k$} \right\},
\end{equation*}
which is closely related to the metric polytope $\M_n$. Indeed, for every~$n$, $\M_n$ is naturally obtained as a limit of $\left(\frac{2}{M}\right)\M_n^M$ as $M$ tends to infinity. We proceed to discuss some of the quantitative aspects of this relation.

As with the continuous problem, observing that the cube
\begin{equation}\label{eq:discrete_problem_cube_structure}
  \left\{\left\lceil\frac{M}{2}\right\rceil,\left\lceil\frac{M}{2}\right\rceil+1,\dotsc, M\right\}^{\binom{n}{2}}
\end{equation}
is fully contained in $\M_n^M$, one gets the following simple lower bound on the cardinality of~$\M_n^M$:
\begin{equation}\label{eq:discrete_problem_naive_estimate}
  |\M_n^M|\ge \left\lceil\frac{M+1}{2}\right\rceil^{\binom{n}{2}}.
\end{equation}

In fact, one may obtain bounds on $|\M_n^M|$ from bounds on $\vol(\M_n)$ and vice-versa. In one direction, consider the map $\varphi \colon (0,2]^{\binom{n}{2}} \to \{1, \dotsc, M\}^{\binom{n}{2}}$ defined by
\[
\varphi(d)_{ij} = \left\lceil \frac{M d_{ij}}{2} \right\rceil.
\]
Observe that $\varphi$ maps $\M_n$ to $\M_n^M$ (as $\lceil x \rceil + \lceil y \rceil \ge \lceil z \rceil$ whenever $x + y \ge z$) and that $\vol(\varphi^{-1}(d')) = \left(\frac{2}{M}\right)^{\binom{n}{2}}$ for any $d' \in \{1, \dotsc, M\}^{\binom{n}{2}}$. Consequently,
\[
\vol(\M_n) \le \left(\frac{2}{M}\right)^{\binom{n}{2}} |\M_n^M|.
\]
In the other direction, consider the map $\psi$ from $\M_n^M$ to the power set of $(0,2]^{\binom{n}{2}}$ defined by
\[
\psi(d) = \prod_{i,j} \left(\frac{2}{M+2}\left(d_{ij} + 1\right), \frac{2}{M+2}\left(d_{ij} + 2\right)\right].
\]
Observe that $\psi$ maps each $d \in \M_n^M$ to a cube that is fully contained in $\M_n$ (as $x+y \ge z$ implies that $(x+\Delta x)+(y+\Delta y) \ge (z+\Delta z)$ for all $\Delta x, \Delta y, \Delta z \in (1,2]$) so that cubes corresponding to different $d$ are disjoint. It follows that
\[
|\M_n^M| \left(\frac{2}{M+2}\right)^{\binom{n}{2}} \le \vol(\M_n).
\]
Putting these bounds together yields
\begin{equation}
  \label{eq:card_MnM-vol_Mn}
  \left(\frac{M}{2}\right)^{\binom{n}{2}}\vol(\M_n)\le |\M_n^M|\le \left(\frac{M}{2}+1\right)^{\binom{n}{2}}\vol(\M_n).
\end{equation}

\smallskip
Concurrently with the writing of this paper, Mubayi and Terry~\cite{mubayi2019discrete} studied the discrete problem in the regime where \emph{$M$ is fixed} and $n$ tends to infinity, proving that
\begin{equation}
  \label{eq:Mubayi-Terry}
  |\M_n^M| =
  \begin{cases}
    \big(1+e^{-\Omega(n)}\big)\left(\frac{M}{2}+1\right)^{\binom{n}{2}} & \text{if $M$ is even,} \\
    \left(\frac{M+1}{2}\right)^{\binom{n}{2} + o(n^2)} & \text{if $M$ is odd}
  \end{cases}
\end{equation}
(with additional structural information in the odd $M$ case).

The above bound reveals that for even $M$, the structure of a uniformly chosen space from $\M_n^M$ is very rigid: the probability that even a single distance lies outside the discrete interval $\{M/2, \dotsc, M\}$ is exponentially small. This strong rigidity property stems from the assumption that $M$ is fixed and does not hold in the continuous setting. Indeed, the bound~\eqref{eq:min-dij-lower} shows that the minimum distance is smaller than $1 - \frac{c}{\sqrt n}$ in typical samples from $\M_n$. Handling such microscopic fluctuations contributes to the difficulty in controlling the volume of $\M_n$ and understanding the structure of typical samples from it.

\subsection{Metric preserving maps}

A map $\phi \colon [0,\infty)\to[0,\infty)$ is metric preserving if $\phi(d) = \big(\phi(d_{ij})\big)$ is a metric
on some set whenever $d = (d_{ij})$ is, e.g., the ceiling operation from the previous subsection.
There are many interesting examples of such maps, see~\cite{Corazza}. Every metric preserving map $\phi$
such that $\sup_{ x\in [0,2]} \phi(x) \le 2$  induces a self-map of the metric polytope.
We wonder how metric preserving maps can be utilized to further study the structure of the metric polytope.

\subsection{Other models for random metric spaces}
In this paper we investigated a certain model of a `random metric space', which in some sense is natural.
The conclusion of our results is that on a large scale this model essentially reduces to the `trivial' model where all distances are in the interval $[1,2]$, and the triangle inequality is trivially satisfied.
It would be interesting to find other models for a `random metric space', which are `natural' on the one hand, and `interesting' on the other hand, in the sense that they reveal new phenomena about metric spaces.
In \cite{MR2086637} Vershik considered one natural candidate for a
random metric space, and proved that it is essentially the Urysohn universal metric space. As remarked within the paper, `An obvious drawback of our construction is that it is not invariant with respect
to the numbering of points'.

\subsection*{Acknowledgments}
We thank Itai Benjamini for asking the question and Gil Kalai for
informing us that this object is known as the metric polytope. We
thank Omer Angel, Dor Elboim, Ronen Eldan, Ehud Friedgut, Shoni Gilboa, Rob Morris, Bal\'asz
R\'ath and Johan W\"astlund for many
interesting discussions. Special thanks are due to Rob Morris, who kindly agreed to us presenting
the results of our joint work with him in Section~\ref{sec:container-method} of this paper.

\bibliographystyle{abbrv}
\bibliography{metric_polytope}
\end{document}